\documentclass[12pt,oneside, a4paper,reqno]{amsart}
\usepackage[utf8]{inputenc}
\usepackage[T1]{fontenc}
\usepackage{lmodern}
\usepackage{etoolbox}
\usepackage[protrusion=false]{microtype}
\usepackage{leftindex}

\usepackage{marginnote}
\usepackage{graphicx}
\usepackage{geometry}
\usepackage[all]{xy}
\pdfpagewidth=\paperwidth
\pdfpageheight=\paperheight
\usepackage[dvipsnames]{xcolor}

\usepackage{amsmath,amsfonts,amssymb,mathtools,manfnt,amsthm}

\numberwithin{equation}{section}

\newtheorem{thm}{Theorem}[section]
\newtheorem{cor}[thm]{Corollary}
\newtheorem{lem}[thm]{Lemma}
\newtheorem{prop}[thm]{Proposition}

\theoremstyle{definition}
\newtheorem{dfn}[thm]{Definition}

\newtheorem{ntn}[thm]{Notation}
\newtheorem{setup}[thm]{Set-up}

\theoremstyle{remark}
\newtheorem{rmk}[thm]{Remark}

\newtheorem{example}[thm]{Example}

\usepackage{tikz}
\usepackage{manfnt}
\usetikzlibrary{matrix}
\usetikzlibrary{arrows.meta}
\usetikzlibrary{positioning}
\usetikzlibrary{calc}
\usetikzlibrary{shapes.misc}

\usetikzlibrary{decorations.pathmorphing}
\usepackage{tikz-cd}

\usepackage{enumitem}
\setenumerate{listparindent=\parindent}
\setlist[enumerate]{
	font=\normalfont,
	label=(\roman*),
	topsep=3pt,
	itemsep=-0.3ex,
	partopsep=1ex,
	parsep=1ex}

\def\namedlabel#1#2{\begingroup
	#2%
	\def\@currentlabel{#2}%
	\phantomsection\label{#1}\endgroup
}
\makeatother

\makeatletter
\g@addto@macro\bfseries{\boldmath}
\makeatother

\usepackage{bbm}

\usepackage{datetime}
\usepackage{url}
\usepackage{tensor}

\usepackage[unicode=true,
colorlinks=true,
linkcolor=blue!75!black,
citecolor=blue!75!black,
urlcolor=purple,
breaklinks=true]{hyperref}

\DeclareMathOperator{\id}{id}
\DeclareMathOperator{\im}{im}

\DeclareMathOperator{\lsp}{span}
\DeclareMathOperator{\clsp}{\overline{\lsp}}

\DeclareMathOperator{\Hom}{Hom}

\DeclareMathOperator{\MCE}{MCE}
\DeclareMathOperator{\Aut}{Aut}

\DeclareMathOperator{\ran}{ran}

\DeclareMathOperator{\coker}{coker}

\DeclareMathOperator{\sgn}{sgn}
\DeclareMathOperator{\ox}{\otimes}

\DeclareMathOperator{\Tor}{Tor}
\DeclareMathOperator{\Tot}{Tot}
\DeclareMathOperator{\Sh}{Sh}
\DeclareMathOperator{\bowtiemap}{{\bowtie}}

\DeclareMathOperator{\Per}{Per}

\newcommand{\lequiv}{[\![}
\newcommand{\requiv}{]\!]}

\newcommand\restr[2]{{
		\left.\kern-\nulldelimiterspace 
		#1 
		\vphantom{|} 
		\right|_{#2} 
}}

\newcommand{\ol}[1]{\overline{#1}}
\newcommand{\fibre}[2]{\tensor*[^{}_{#1}]{*}{^{}_{#2}}}
\newcommand{\bs}{\backslash}

\newcommand{\la}{\triangleright}
\newcommand{\ra}{\triangleleft}
\newcommand{\bla}{\blacktriangleright}
\newcommand{\bra}{\blacktriangleleft}

\newcommand{\bowtiel}[1]{\leftindex_{#1}{\bowtie{}}}

\newcommand{\cb}{\color{blue}}
\newcommand{\cre}{\color{red}}

\newcommand{\catname}[1]{{\normalfont \mathsf{#1}}}

\newcommand{\Ab}{\catname{Ab}}
\newcommand{\MP}{\catname{MP}}
\newcommand{\Ch}{\catname{Ch}_{\bullet}}

\newcommand{\NN}{\mathbb{N}}

\newcommand{\RR}{\mathbb{R}}
\newcommand{\TT}{\mathbb{T}}

\newcommand{\ZZ}{\mathbb{Z}}

\newcommand{\Aa}{\mathcal{A}}
\newcommand{\Bb}{\mathcal{B}}
\newcommand{\Cc}{\mathcal{C}}
\newcommand{\Dd}{\mathcal{D}}
\newcommand{\Ee}{\mathcal{E}}
\newcommand{\Ff}{\mathcal{F}}
\newcommand{\Gg}{\mathcal{G}}
\newcommand{\Hh}{\mathcal{H}}

\newcommand{\Mm}{\mathcal{M}}

\newcommand{\Tt}{\mathcal{T}}

\newcommand{\Sss}{\mathfrak{s}}
\newcommand{\Ttt}{\mathfrak{t}}
\newcommand{\Uuu}{\mathfrak{u}}

\newcommand{\Www}{\mathfrak{w}}

\newcommand\RPi{\rotatebox[origin=c]{180}{$\Pi$}}
\newcommand\RPsi{\rotatebox[origin=c]{180}{$\Psi$}}

\newcommand{\bone}{\mathbbm{1}}

\begin{document}
	
\DeclareRobustCommand{\subtitle}[1]{\\[1.5em] #1}	

\title[Homology for self-similar actions and Zappa--Sz\'ep products]{Homology and twisted \texorpdfstring{$C^*$}{C*}-algebras for self-similar actions and Zappa--Sz\'ep products}

\author[A. Mundey]{Alexander Mundey}
\email{amundey@uow.edu.au}
\author[A. Sims]{Aidan Sims}
\email{asims@uow.edu.au}

\thanks{This research was supported by Australian Research Council grant DP220101631.
The first author was supported by University of Wollongong AEGiS CONNECT grant 141765.}

\address{School of Mathematics and Applied Statistics \\
	University of Wollongong\\
	NSW  2522\\ Australia}

\date{\today}
\subjclass[2020]{18G15 (primary); 18A32, 46L05 (secondary)}
\keywords{Self-similar action; Zappa--Sz\'ep product; homology; twisted $C^*$-algebra; $k$-graph}

\begin{abstract}
We study the categorical homology of Zappa--Sz\'ep products of small categories, which include all self-similar actions. We prove that the categorical homology coincides with the homology of a double complex, and so can be computed via a spectral sequence involving homology groups of the constituent categories. We give explicit formulae for the isomorphisms involved, and compute the homology of a class of examples that generalise odometers. We define the $C^*$-algebras of self-similar groupoid actions on $k$-graphs twisted by $2$-cocycles arising from this homology theory, and prove some fundamental results about their structure.
\end{abstract}

\maketitle

\renewcommand{\subtitle}[1]{}

\tableofcontents

\section{Introduction}

This paper achieves three main objectives:
\begin{enumerate}[label=(\arabic*), ref=(\arabic*)]
	\item to introduce a unifying framework, which we call matched pairs of categories, for self-similar actions, graphs of groups, Zappa--Sz\'ep products, and $k$-graphs;
	\item to introduce homology and cohomology for matched pairs, and develop practical tools for computing them; and
	\item to associate twisted $C^*$-algebras to self-similar groupoid actions on $k$-graphs, and establish fundamental structure theorems for these $C^*$-algebras.
\end{enumerate}

Self-similar groups of automorphisms of trees were introduced in the early 1980s as models for new classes of groups. Grigorchuk used self-similar groups to describe the first example of a finitely generated group with intermediate growth \cite{Gri80, Gri84}, and Nekrashevych recently used them to produce the first simple groups of intermediate growth \cite{Nek18}. Self-similar groups have been studied intensively ever since Grigorchuk's work, including, since the seminal work of Nekrashevych \cite{Nek05}, via their $C^*$-algebras. Nekrashevych studies $C^*$-algebraic representations not just of a self-similar group, but of the entire self-similar system: a unitary representation of the group and a Cuntz representation of the alphabet being acted upon. The resulting $C^*$-algebra encodes information about the self-similar system through both $K$-theory \cite{Nek05} and KMS-data \cite{LRRW14, EP17, LRRW18}. The former suggests that homological invariants of self-similar actions could be a profitable avenue of study.

Initially, self-similar actions were presented with an apparent asymmetry between the role of the group and the role of the alphabet. But recent generalisations \cite{LRRW14, EP17, LRRW18, LiYang, LV22} make it increasingly clear that the roles of the two objects are symmetric, and that self-similar actions are closely related to Zappa--Sz\'ep products.

Introduced by Zappa \cite{Zap42} and Sz\'ep \cite{Sze50}, Zappa--Sz\'ep products of groups are a generalisation of semidirect products in which each of the two constituent groups acts on the other; so both embed as (not necessarily normal) subgroups of the product. Subsequent generalisations include Zappa--Sz\'ep-style products of increasingly general pairs of algebraic objects: \cite{Bri05,Law08,LRRW14,BPRRW17,BKQS18, LRRW18,LV22,OP23,DL22,DL23}.

Here, we start with a \emph{matched pair of categories}: small categories $\Cc$ and $\Dd$ with common object set, a left action $(c, d) \mapsto c \la d$ of $\Cc$ on $\Dd$ and a right action $(c, d)\mapsto c \ra d$ of $\Dd$ on $\Cc$ satisfying  the compatibility conditions of \cite{Zap42,Sze50}. Each such pair determines a Zappa--Sz\'ep-product category $\Cc \bowtie \Dd$; this can be viewed  either ``externally'' as the fibred product $\Dd \mathbin{{_s*_r}} \Cc$ under a suitable multiplication, or ``internally'' as the universal category containing copies of $\Cc$ and $\Dd$ with a strict factorisation system as in \cite{RW02} that implements $\la$ and $\ra$. All of the algebraic product constructions mentioned above fit into this framework, as do graphs of groups \cite{Bas93,Ser80} and $k$-graphs \cite{KP00}. Their $C^*$-algebraic representations all boil down to representations, in the sense of Spielberg \cite{Spe20}, of the associated Zappa--Sz\'ep-product category.

In the study of $C^*$-algebras associated to algebraic  or combinatorial objects, there is a well-established principle that interesting $C^*$-algebraic properties emerge when we twist the multiplication by a $\TT$-valued $2$-cocycle. The archetypal examples are the noncommutative tori $A_\theta$, which are obtained by twisting the multiplication in unitary representations of $\ZZ^k$ by $\TT$-valued $2$-cocycles---which themselves are simply computed in terms of characters (1-cocycles) on the constituent factors of $\ZZ$ in $\ZZ^k$ \cite{OPT}. To generalise this to matched pairs we need both a suitable definition of cohomology, and effective tools for computing it in terms of the cohomology of the constituent categories.

For topological spaces $X$ and $Y$, the classical \emph{Eilenberg--Zilber Theorem} gives an isomorphism between the (singular) homology of the chain complex $C_\bullet(X \times Y)$ and the total homology of the tensor product double complex $C_{\bullet}(X) \ox_{\ZZ} C_{\bullet}(Y)$. We take this as our inspiration for analysing homology of Zappa--Sz\'ep products. We consider the classical categorical homology of $\Cc \bowtie \Dd$: $n$-chains are $\ZZ$-linear combinations of composable $n$-tuples, and boundary maps are alternating sums of the maps obtained by deleting the first or last entry in a composable tuple, or composing adjacent terms. We show that this homology can be computed in terms of a double complex, called the \emph{matched complex}: its columns are chain complexes for the homology of $\Dd$ with coefficients in modules spanned by composable tuples in $\Cc$; and its rows are chain complexes for the homology of $\Cc$ with coefficients in modules spanned by composable tuples in $\Dd$. The matched complex $C_{\bullet,\bullet}$ is not the tensor product $C_{\bullet}(\Cc) \ox_\ZZ C_{\bullet}(\Dd)$, but its terms are fibred products of a similar form.

The matched complex admits two natural homology theories---diagonal homology $H^\Delta_\bullet(\Cc, \Dd)$ and total homology $H^{\Tot}_\bullet(\Cc, \Dd)$. These are isomorphic via explicit chain equivalences called the Eilenberg--Zilber map and the Alexander--Whitney map. The total homology $H^{\Tot}_\bullet(\Cc, \Dd)$, is defined in terms of the homology of the constituent categories $\Cc$ and $\Dd$. So to see that the categorical homology $H^{\bowtie}_\bullet(\Cc, \Dd)$ of $\Cc  \bowtie \Dd$ suits our purposes, we use the method of acyclic models to construct explicit chain equivalences between the chain complex $C^{\bowtie}_\bullet(\Cc, \Dd)$ defining $H^{\bowtie}_{\bullet}(\Cc, \Dd)$ and the diagonal chain complex $C^\Delta_\bullet(\Cc, \Dd)$. Combined with the Eilenberg--Zilber map, this gives a computable isomorphism $H^{\Tot}_\bullet(\Cc, \Dd) \cong H^{\bowtie}_\bullet(\Cc, \Dd)$. Dualising yields isomorphisms $H^\bullet_{\Tot}(\Cc, \Dd; \TT) \cong H^\bullet_{\bowtie}(\Cc, \Dd; \TT)$ in cohomology.

As an aside, this shows that if a category admits a strict factorisation system, then its categorical homology can be computed in terms of that of the embedded subcategories. This yields, for example, a potential iterative approach to computing homology for $k$-graphs.

We use our results to compute the homology of a class of self-similar groupoid actions on graphs that generalise the odometer. We calculate the homology in terms of the two nonzero homology groups of the underlying graph $E$, and the kernel and cokernel of an $E^0 \times E^1$ matrix encoding the orders of the odometers involved. \emph{En passant}, we establish useful general results about homology for matched pairs in which one factor is the path category of a directed graph or a bundle of monoids, with stronger results when the monoids are copies of $\ZZ$. These results would be well suited to computing the homology of Exel--Pardo systems \cite{EP17}.

The main motivation for our work on homology is to study twisted $C^*$-algebras of matched pairs. The point is that the natural definition of a $C^*$-algebraic representation of a matched pair, as made clear by Spielberg's work \cite{Spe20}, is as a multiplicative map $\zeta \mapsto t_\zeta$ from its Zappa--Sz\'ep product category to a semigroup of partial isometries. Consequently, the natural definition of a twisted representation is in terms of a categorical $2$-cocycle $c$ on the Zappa--Sz\'ep product: we twist by the formula $s_\zeta s_\eta = c(\zeta, \eta) s_{\zeta\eta}$. However, the total homology (and cohomology) is a more computable theory, and clearly reflects the decomposition of the Zappa--Sz\'ep product category into its constituent components. Our main homology theorem allows us to define and analyse the $C^*$-algebras in the natural way via categorical $2$-cocycles, but pass to total cohomology when we wish to identify the possible twists for a given matched pair or produce nontrivial cocycles in concrete examples.

We explore this in the context of matched pairs consisting of a groupoid $\Gg$ and a row-finite $k$-graph $\Lambda$ with no sources (self-similar actions of groupoids on such $k$-graphs). This covers a fairly general class of examples with relatively complex cohomology, for which $C^*$-algebraic representations in the sense of Spielberg of the associated untwisted pair are well understood. Given a categorical $2$-cocycle $c \in C^2_{\bowtie}(\Gg, \Lambda; \TT)$, we define a universal twisted Toeplitz algebra $\Tt C^*(\Gg, \Lambda; c)$ and a universal Cuntz--Krieger algebra $C^*(\Gg, \Lambda; c)$ in both of which all the generators are nonzero. We show that $\Tt C^*(\Lambda, \restr{c}{\Lambda^2})$ embeds in $\Tt C^*(\Gg, \Lambda; c)$ and likewise that $C^*(\Lambda, \restr{c}{\Lambda^2})$ embeds in $C^*(\Gg, \Lambda; c)$. We also show that if $\Gg$ is amenable then $C^*(\Gg, \restr{c}{\Gg^2})$ embeds in $\Tt C^*(\Gg, \Lambda; c)$, and that an addition condition developed by Yusnitha \cite{Yusnitha} ensures that it also embeds in $C^*(\Gg, \Lambda; c)$. We
establish a gauge-invariant uniqueness theorem for $C^*(\Gg, \Lambda; c)$ and prove that cohomologous $2$-cocycles yield isomorphic twisted $C^*$-algebras.

We then construct twisted $C^*$-algebras $\Tt C^*_\varphi(\Gg, \Lambda)$ and $C^*_\varphi(\Gg, \Lambda)$ associated to a total 2-cocycle $\varphi \in C^2_{\Tot}(\Gg, \Lambda; \TT)$. We prove that our cochain equivalence $\Psi^* \colon C^2_{\Tot}(\Gg, \Lambda; \TT) \to C^2_{\bowtie}(\Gg, \Lambda; \TT)$ induces isomorphisms $\Tt C^*_{\varphi} (\Gg, \Lambda) \cong \Tt C^*(\Gg,\Lambda,\Psi^*(\varphi))$ and $ C^*_\varphi(\Gg, \Lambda) \cong C^*(\Gg, \Lambda; \Psi^*(\varphi))$.

\smallskip

The paper is organised as follows. In Section~\ref{sec:prelims} we establish some background: on categories; on actions of one category on another; and on directed graphs and their path categories.

In Section~\ref{sec:matched_pairs} we discuss matched pairs of small categories. We show that each matched pair admits a Zappa--Sz\'ep product, and discuss internal and external descriptions of this object and its relationship to strict factorisation systems. We show how the actions in a matched pair extend to actions on the categories of composable tuples in the categories involved. We give a number of concrete examples of matched pairs, including the key \emph{model matched pairs} that serve as local models for composable tuples in arbitrary matched pairs.

In Section~\ref{sec:homology} we introduce the three homology theories for matched pairs. We first introduce categorical homology of a small category, described in terms of simplicial sets. We then introduce the \emph{matched complex}---a double complex associated to a matched pair---in terms of a bisimplicial group, and show that the assignment of the matched complex to a matched pair is functorial. We then define the diagonal complex, the total complex, and the associated homology theories of a matched pair.

In Section~\ref{sec:main_theorem} we prove our main homology theorem: categorical homology, total homology, and diagonal homology coincide. In Section~\ref{subsec:the_chain_maps}, we describe the three chain maps that appear in our main theorem: the first is the Eilenberg--Zilber map for double complexes---we just give a formula for use in computations; the other two, $\Pi \colon C^\Delta_\bullet \to C^{\bowtie}_\bullet$ and $\Psi \colon C^{\bowtie}_\bullet \to C^{\Tot}_\bullet$, are specific to our situation.  In Section~\ref{subsec:the actual theorem}, we state the main theorem, Theorem~\ref{thm:cohomologies_are_the_same}, describe the Alexander--Whitney map, which induces the inverse of the Eilenberg--Zilber map, and outline the strategy of the proof. In Section~\ref{sec:homology_of_models}, we show that our model matched pairs are acyclic in both diagonal and categorical homology, and describe functors from the Zappa--Sz\'ep-product categories of model matched pairs into $\Cc \bowtie \Dd$ that realise all generators of each chain complex. In Section~\ref{subsec:proof_of_theorem} we invoke the method of acyclic models to characterise chain equivalences between the diagonal and categorical complexes. In Subsection~\ref{subsec:PiPsidef} we show that the concrete chain maps described in Section~\ref{subsec:the_chain_maps} are such chain equivalences and describe their inverses. Finally, in Section~\ref{subsec:consequences}, we describe a spectral sequence that computes the homology of a matched pair, and a K\"unneth theorem for matched pairs of monoids.

In Section~\ref{sec:further_examples}, we compute the homology of a concrete class of examples: ``graphs of odometers.'' We consider a finite directed graph $E$ together with a labelling $p \colon  E^1 \to \{1, 2, \dots\}$ of its edges by strictly positive integers. We build the augmented graph, $F$ that has a bundle $\{e\} \times \ZZ/p(e)\ZZ$ of $p(e)$ parallel edges for each edge $e \in E$. We consider a matched pair $(E^0 \times \ZZ, F^*)$ in which the copies of $\ZZ$ behave, collectively, like odometers. In Section~\ref{subsec:mp_path_categories}, we show that in a matched pair where the second factor is the path category of a graph, only the first two rows of the second page of the spectral sequence obtained above are nonzero. In Section~\ref{subsec:mp_monoid_bundles}, we show that for matched pairs where the first factor is a bundle of monoids, the homology groups each decompose as the direct sum of the corresponding homology groups (with appropriate coefficients) of the monoids. In Section~\ref{subsec:mp_integer_bundles}, we prove that if the first factor is a bundle of copies of $\ZZ$, only the first two columns of the spectral sequence are nonzero, and the homology of each column is computable via a chain complex very similar to the bar resolution of $\ZZ$. In Section~\ref{subsec:odometer_graphs} we restrict to graphs of odometers, and write down an $E^0 \times E^1$ matrix over $\ZZ$ whose kernel and cokernel, together with the homology of the graph $E$, compute the homology of the system (Theorem~\ref{thm:example homology} and Corollary~\ref{cor:strongly connected}).

In Section~\ref{sec:CStars}, we consider twisted $C^*$-algebras associated to matched pairs.  Section~\ref{subsec:categorical twists} deals with twists by categorical cocycles, and establishes some fundamental results about the associated $C^*$-algebras: we prove that the generators are all nonzero and give sufficient conditions under which the twisted $C^*$-algebra of $\Gg$ embeds in each of $\Tt C^*(\Gg, \Lambda; c)$ and $C^*(\Gg, \Lambda; c)$ in Proposition~\ref{prop:faithfulness}; we prove our gauge-invariant uniqueness theorem, Corollary~\ref{cor:giut}; and we show that the isomorphism classes of $\Tt C^*(\Gg, \Lambda; c)$ and $C^*(\Gg, \Lambda; c)$ only depend on the cohomology class of $c$. Section~\ref{subsec:total twists} describes twists by total cocycles, and shows that these correspond to twists by categorical 2-cocycles via the isomorphism of cohomology induced by our main theorem above (Theorem~\ref{thm:total twist universal}).

\section{Preliminaries}
\label{sec:prelims}

Throughout this article, $\Cc$ and $\Dd$ denote small categories.
We identify $\Cc$ with its set of morphisms and write $\Cc^0 \subseteq \Cc$ for the set of identity morphisms (identified with objects).
We write $r,s \colon \Cc \to \Cc^0$ for the maps assigning to $c \in \Cc$ (the identity morphisms at) its codomain and domain.
For $n \ge 1$, we write $\Cc^n$ for the set of composable $n$-tuples in $\Cc$ and define $r,s \colon \Cc^n \to \Cc^0$ by $r(c_1,\ldots,c_n) = r(c_1)$ and $s(c_1,\ldots, c_n) = s(c_n)$. For $x,\,y \in \Cc^0$ we write $x\Cc^n \coloneqq \{ (c_1,\ldots,c_n) \mid r(c_1) = x\}$, $\Cc^ny \coloneqq \{ (c_1,\ldots,c_n) \mid s(c_n) = y\}$, and $x \Cc^n y \coloneqq x \Cc^n \cap \Cc^n y$.

If $\Cc^0 = \Dd^0$ we define
\[
\Cc *\Dd \coloneqq \Cc \fibre{s}{r} \Dd = \{ (c,d) \in \Cc \times \Dd \mid s(c) = r(d)\}.
\]

If $\Cc$, $\Cc'$, and $\Dd$ have the same objects and $f \colon \Cc \to \Cc'$ satisfies $f \circ s = s \circ f$, then $f * 1_{\Dd} \colon \Cc * \Dd \to \Cc' * \Dd$ is the map $(f * 1_{\Dd})(c,d) \coloneqq (f(c), d)$. Similarly, if $r \circ f = f \circ r$, then  $1_{\Dd} * f \colon \Dd * \Cc \to \Dd * \Cc'$ is the map$ (1_{\Dd} * f)(d,c) \coloneqq (d, f(c))$.

An \emph{action} of a category $\Cc$ on the left of a set $X$ consists of maps $a \colon X \to \Cc^0$ and $\la \colon \Cc \fibre{s}{a} X \to X$ such that $a(x) \la x = x$, $a(c_2 \la x) = r(c_2)$, and $(c_1c_2) \la x = c_1 \la(c_2 \la x)$  for all $(c_1,c_2) \in \Cc^2$ and $x \in X$ with $s(c_2) = a(x)$. If $X = \Dd$ is a category such that $\Cc^0 = \Dd^0$ we only consider left actions for which $a= r \colon \Dd \to \Cc^0$, so $\la$ is a map $\la \colon \Cc * \Dd \to \Dd$. We also then require that the action of $\Cc$ on $\Dd$ restricts to an action of $\Cc$ on $\Dd^0 = \Cc^0$. That is, $c \la s(c) = r(c)$ for all $c \in \Cc$.
 
A right action of $\Cc$ is defined similarly: it is a left action of the opposite category $\Cc^{\operatorname{op}}$ on $X$ (on $\Dd^{\operatorname{op}}$ if $X = \Dd$ is a category with $\Cc^0 = \Dd^0$).

A \emph{groupoid} $\Gg$ is a small category in which every morphism $g \in \Gg$ has an inverse $g^{-1} \in \Gg$ such that $g g^{-1} = r(g)$ and $g^{-1}g = s(g)$. The set of identity morphisms is called the \emph{unit space} of $\Gg$. In this paper, $\Gg$ always denotes a countable discrete groupoid.

A \emph{directed graph} is a quadruple $E = (E^0,E^1,r,s)$, consisting of countable sets $E^0$ of \emph{vertices} and $E^1$ of \emph{edges}, and maps $r \colon E^1 \to E^0$ and $s \colon E^1 \to E^0$ called the \emph{range} and \emph{source} maps.
For $n \ge 1$, we denote by $E^n \coloneqq \{\mu = \mu_1\cdots \mu_n \mid \mu_i \in E^1,\, s(e_i) = r(e_{i+1})\}$, the \emph{paths of length $n$} in the graph $E$. If $\mu \in E^m$ and $\nu \in E^n$ with $s(\mu) = r(\nu)$, then $\mu\nu \coloneqq \mu_1 \cdots \mu_m \nu_1 \cdots \nu_n \in E^{m+n}$ is the \emph{concatenation} of $\mu$ and $\nu$.
The range and source maps extend
to $E^n$: $r(\mu) = r(\mu_1) $ and $s(\mu) = s(\mu_n)$. We regard elements $v$ of $E^0$ as paths of length $0$ with $r(v) = s(v) = v$, and we extend concatenation by the formula $r(\mu)\mu = \mu = \mu s(\mu).$

The \emph{path category} of a directed graph $E$ is the collection $E^* \coloneqq \bigsqcup_{n=0}^\infty E^n$ of all finite paths in $E$. The objects of $E^*$ are $E^0$, and the range and source maps on the $E^n$ extend to domain and codomain maps $r \colon E^* \to E^0$ and $s \colon E^* \to E^0$. Composition is concatenation. We use $|\mu|$ to denote the length of $\mu \in E^*$, so $|\mu| = n$ if and only if $\mu \in E^n$.

\section{Matched pairs, Zappa--Sz\'ep products, and factorisation systems}
\label{sec:matched_pairs}

\subsection{Matched pairs}

In this subsection we introduce matched pairs of small categories and associated Zappa--Sz\'ep-product categories. We also examine how factorisation rules and strict factorisation systems are related to these constructions. Our notation parallels~\cite{DL23}.

\begin{dfn}\label{dfn:matched_pair}
A \emph{matched pair} is a quadruple $(\Cc,\Dd, \la , \ra)$ consisting of small categories $\Cc$, $\Dd$ with $\Cc^0 = \Dd^0$, a left action $\la \colon \Cc  * \Dd \to \Dd$ of $\Cc$ on $\Dd$, and a right action $\ra  \colon  \Cc * \Dd \to \Cc$ of $\Dd$
on $\Cc$ such that for all $(c_1,c_2,d_1,d_2) \in \Cc^2 * \Dd^2$,
	\begin{enumerate}[label=(MP\arabic*), ref=(MP\arabic*)]
		\item\label{itm:ZS-rs} $s(c_2 \la d_1) = r(c_2 \ra d_1)$,
		\item\label{itm:ZS-dot} $c_2 \la (d_1d_2) = (c_2 \la d_1) ((c_2 \ra d_1) \la d_2)$, and
		\item\label{itm:ZS-diamond} $(c_1c_2) \ra d_1 = (c_1 \ra (c_2 \la d_1)) (c_2 \ra d_1)$.	
	\end{enumerate}
	We often just say that $(\Cc,\Dd)$ is a matched pair, and suppress the actions $\la, \ra$.

	The category $\MP$ of matched pairs has matched pairs as objects and morphisms $f = (f^L,f^R) \colon (\Cc,\Dd) \to (\Cc',\Dd')$ consisting of pairs of functors $f^L \colon \Cc \to \Cc'$ and $f^R \colon \Dd \to \Dd'$ such that for all $(c,d)
\in \Cc * \Dd$,
	\begin{enumerate}[label=(\roman*), ref=(\roman*)]
	\item $(f^L(c),f^R(d)) \in \Cc' * \Dd'$,
	\item $f^L(c)\la f^R(d) = f^R(c \la d)$, and
	\item $f^L(c) \ra f^R(d) = f^L(c \ra d).$
	\end{enumerate}

\end{dfn}

\begin{rmk}
	We are unsure of the provenance of the term \emph{matched pair}. It is used for various related notions: matched pairs of groupoids in \cite{AA05}; and matched pairs of Hopf algebras in \cite{Sin72}. For previous work on matched pairs with $\Cc^0 \ne \Dd^0$, see \cite[Definition~2.2]{DL23}.

In general, we could define and study matched pairs of categories that do not have the same object set as follows. Let $\Cc$ and $\Dd$ be small categories. Let $\Xi$ be a set and let $\phi : \Xi \to \Cc^0$ and $\psi : \Xi \to \Dd^0$ be surjections. The ampliation category $\phi^*\Cc \coloneqq \{(x, c, y) \in \Xi \times \Cc \times \Xi : \phi(x) = r(c)\text{ and }\phi(y) = s(c)\}$ has objects $(\phi^*\Cc)^0 = \{(x, \phi(x), x) : x \in \Xi\}$, which we can identify with $\Xi$, and structure maps $r(x, c, y) = x$, $s(x, c, y) = y$ and $(x, c, y)(y, c, z) = (x, cc', z)$. The ampliations $\phi^*\Cc$ and $\psi^*\Dd$ are small categories with $(\phi^* \Cc)^0 = (\psi^*\Dd)^0$. So we could define a matched-pair structure on $(\Cc, \Dd)$ to be a tuple $(\Cc, \Dd, \Xi, \phi, \psi, \rhd, \lhd)$ in which $\phi : \Xi \to \Cc^0$ and $\psi : \Xi \to \Dd^0$ are surjections and $(\phi^*\Cc, \psi^*\Dd, \rhd, \lhd)$ is a matched pair as above. As (non-topologised) categories, the matched pairs of \cite[Definition~2.2]{DL23} have this form. In their notation, $\Xi \coloneqq \mathcal{X}^{(0)}$, $\psi \coloneqq \id_{\Xi}$ and $\phi \coloneqq \rho_{\mathcal{X}}^{(0)}$.
\end{rmk}

Given a matched pair $(\Cc,\Dd)$, we define $r \colon \Dd * \Cc \to \Cc^0$ and $s \colon \Dd * \Cc \to \Cc^0$ by $r(d,c) \coloneqq r(d)$ and $s(d,c) \coloneqq s(c)$.

\begin{dfn}
	Let $\Cc$ and $\Dd$ be small categories with $\Cc^0 = \Dd^0$.  A \emph{factorisation rule} on $(\Cc, \Dd)$ is a map $\bowtie \colon \Cc * \Dd \to \Dd * \Cc$ such that
	\begin{enumerate}[label={(FR\arabic*)}]
		\item \label{itm:matching} $r(c \bowtie d) = r(c)$ and $s(c \bowtie d) = s(d)$ for all $(c,d) \in \Cc * \Dd$, 
		\item \label{itm:units} $c \bowtie s(c) = (r(c),c)$ and $r(d) \bowtie d = (d ,s(d))$ for all $c \in \Cc$ and $d \in \Dd$, 
		and
		\item \label{itm:twist} if $\mu_\Cc \colon \Cc^2 \to \Cc$ and $\mu_\Dd \colon \Dd^2 \to \Dd$ denote the composition maps, then the following diagrams commute:
		\[
		\begin{tikzcd}
			\Cc^2 * \Dd \arrow[d, "\mu_\Cc * 1_\Dd"] \arrow[r, "1_\Cc * \bowtiemap"] & \Cc * \Dd * \Cc \arrow[r, "\bowtiemap*1_\Cc"]
			& \Dd * \Cc^2 \arrow[d, "1_\Dd * \mu_\Cc"] \\
			\Cc * \Dd \arrow[rr, "\bowtiemap"]
			&
			& \Dd * \Cc
		\end{tikzcd}
		\quad \quad
		\begin{tikzcd}
			\Cc * \Dd^2 \arrow[d, "1_{\Cc} * \mu_\Dd"] \arrow[r, "\bowtiemap * 1_\Dd"]
			& \Dd * \Cc * \Dd \arrow[r, "1_\Dd\\ * \bowtiemap"]
			& \Dd^2 * \Cc \arrow[d, "\mu_\Dd * 1_\Cc"]  \\
			\Cc * \Dd \arrow[rr, "\bowtiemap"]
			&
			& \Dd * \Cc\hbox to 0pt{.}
		\end{tikzcd}
		\]
	\end{enumerate}
\end{dfn}


The reason for the name \emph{factorisation rule} becomes clear in the context of Zappa--Sz\'ep products (Definition~\ref{dfn:zappa_szep_product}).
Matched pairs and factorisation rules are equivalent in the following sense.

\begin{lem}\label{lem:intertwiner_is_pair}
	Let $\Cc$ and $\Dd$ be small categories with the same object set.
If $(\Cc,\Dd, \la ,\ra)$ is a matched pair, then the formula
	\begin{equation}\label{eq:mp_to_bowtie}
		c \bowtie d \coloneqq (c \la d, c \ra d)
	\end{equation}
determines a factorisation rule $\bowtie \colon \Cc * \Dd \to \Dd * \Cc$. Conversely, if $\bowtie \colon \Cc * \Dd \to \Dd * \Cc$ is a factorisation rule and $p_\Dd \colon \Dd * \Cc \to \Dd$ and $p_\Cc \colon \Dd * \Cc \to \Cc$ are the coordinate projections, then $\la \colon \Cc * \Dd \to \Dd$ and $\ra \colon \Cc * \Dd \to \Cc$ given by
\begin{equation}\label{eq:bowtie_to_mp}
	c \la d \coloneqq p_\Dd (c \bowtie d)
	\qquad \text{and} \qquad
	c \ra d \coloneqq p_\Cc (c \bowtie d)
\end{equation}
make $(\Cc,\Dd,\la,\ra)$ a matched pair.
\end{lem}

\begin{proof}
	First suppose that $\bowtie \colon \Cc * \Dd \to \Dd * \Cc$ is a factorisation rule. Define $\la,\ra$ by \eqref{eq:bowtie_to_mp}.
	Since $c \bowtie d \in \Dd * \Cc$, \ref{itm:ZS-rs} holds.  The left-hand diagram of \ref{itm:twist} implies that	
	\begin{align*}
		((c_1c_2) \la d, (c_1c_2) \ra d)
		=(c_1c_2) \bowtie d
		= (c_1 \la (c_2 \la d), (c_1 \ra (c_2 \la d)) (c_2 \ra d)).
	\end{align*}
	Together with \ref{itm:units} this implies that $\la$ is a left action and \ref{itm:ZS-diamond} holds.
	Symmetrically, $\ra$ is a right action and \ref{itm:ZS-dot} holds.

	Now suppose that $(\Cc,\Dd,\la,\ra)$ is a matched pair and define $\bowtie \colon \Cc * \Dd \to \Dd * \Cc$ by \eqref{eq:mp_to_bowtie}.
	Then $c \bowtie d \in \Dd * \Cc$ by
\ref{itm:ZS-rs}.
Since $\la,\,\ra$ are actions $r(c \la d) = r(c)$ and $s(c \ra d) = s(d)$, giving \ref{itm:matching}. Also, $r(c) \la c = c$, $r(c) \ra c = s(c)$, $d \la s(d) = r(d)$, and $d \ra s(d) = d$ giving~\ref{itm:units}.
 For \ref{itm:twist} we use \ref{itm:ZS-diamond} at the second equality to compute,
	\begin{align*}
		(c_1c_2) \bowtie d
		&= (c_1c_2 \la d , c_1c_2 \ra d)
		=(c_1 \la (c_2 \la d), (c_1 \ra (c_2 \la d))(c_2 \ra d))\\
		&= (1_\Dd * \mu_\Cc) \circ (\bowtie * 1_\Cc) \circ (1_\Cc * \bowtie) (c_1,c_2,d),
	\end{align*}
	and symmetrically $c \bowtie (d_1d_2) = (\mu_d * 1_\Cc) \circ (1_\Dd * \bowtie) \circ (\bowtie \circ 1_{\Dd}) (c_1,d_2,d_2)$.
\end{proof}

We use Lemma~\ref{lem:intertwiner_is_pair} without comment to move between matched pairs and factorisation rules. Importantly, \ref{itm:ZS-rs}--\ref{itm:ZS-diamond} give the fibre product $\Dd * \Cc$ the structure of a category.

\begin{lem}\label{lem:zappa_szep_category}
Suppose that $(\Cc,\Dd)$ is a matched pair and let $\mu_\Cc$ and $\mu_\Dd$ denote the composition maps on $\Cc$ and $\Dd$ respectively. Define $\mu_{\bowtie} \colon (\Dd * \Cc)^2 \to \Dd * \Cc$ by
	\[
	\mu_{\bowtie} \coloneqq (\mu_\Dd * \mu_\Cc) \circ (1_\Dd * \bowtiemap * 1_\Cc ).
	\]
Then for $(d_1,c_1),\,(d_2,c_2) \in \Dd * \Cc$ such that $s(c_1) = r(d_2)$,
	\begin{equation}\label{eq:mu_bowtie}
			\mu_{\bowtie} ((d_1,c_1),(d_2,c_2))
		=(d_1(c_1 \la d_2), (c_1 \ra d_2) c_2).
	\end{equation}
	Moreover, $\Dd * \Cc$ is a small category with $(\Dd*\Cc)^0 = \Cc^0 = \Dd^0$, $r(d,c) \coloneqq r(d)$, $s(d,c) \coloneqq s(c)$, and composition $(d_1,c_1)(d_2,c_2) \coloneqq \mu_{\bowtie} ((d_1,c_1),(d_2,c_2))$.
	 The maps $\iota_\Cc \colon \Cc \to \Dd * \Cc$ and $\iota_\Dd \colon \Dd \to \Dd * \Cc$ defined by $\iota_\Cc(c) = (r(c),c)$ and $\iota_\Dd(d) = (d,s(d))$ are faithful functors.
\end{lem}

\begin{proof}
	Equation~\eqref{eq:mu_bowtie} follows from~\eqref{eq:mp_to_bowtie}.
	For associativity, we calculate:
	\begin{align*}
		\mu_{\bowtie}\big((d_1,c_1) &\mu_{\bowtie}((d_2,c_2),(d_3,c_3))\big)\\
		&=\mu_{\bowtie}\big((d_1,c_1 ),(d_2 (c_2 \la d_3), (c_2 \ra d_3) c_3)\big)\\
		&= \big(d_1 (c_1 \la (d_2 (c_2 \la d_3)) ), c_1 \ra (d_2 (c_2 \la d_3)) (c_2 \ra d_3)c_3 \big)\\
		&= \big( d_1 (c_1 \la d_2) ((c_1 \ra d_2)c_2 \la d_3),
		((c_1 \ra d_2) \ra (c_2 \la d_3)) (c_2 \ra d_3)c_3 \big)\\
		&= \big(d_1 (c_1 \la d_2) ((c_1 \ra d_2)c_2 \la d_3),
		(c_1 \ra d_2(c_2\la d_3)) (c_2 \ra d_3)c_3  \big)\\
		&= \mu_{\bowtie}\big( d_1 (c_1 \la d_2) , (c_1 \ra d_2) c_2, d_3,c_3\big)\\
		&= \mu_{\bowtie}\big(\mu_{\bowtie}((d_1,c_1),(d_2,c_2)),(d_3,c_3)\big).
	\end{align*}
	For functoriality of $\iota_\Cc$ we calculate
	\begin{align*}
	\mu_{\bowtie}	(\iota_\Cc (c_1),\iota_\Cc(c_2) ) = (r(c_1)(c_1 \la s(c_1)),(c_1 \ra s(c_1)) c_2) = (r(c_1),c_1c_2) = \iota_{\Cc}(c_1c_2).
	\end{align*}
Functoriality of $\iota_\Dd$ follows analogously. Faithfulness is clear.
\end{proof}

\begin{dfn}\label{dfn:zappa_szep_product}
	We call the small category $\Dd * \Cc$ with the composition $\mu_{\bowtie}$ of Lemma~\ref{lem:zappa_szep_category} the \emph{Zappa--Sz\'ep product} of $\Cc$ and $\Dd$, and denote it by $\Cc \bowtie \Dd$.
\end{dfn}

We identify $\Cc$ and $\Dd$ with the subcategories $\iota_\Cc(\Cc)$ and $\iota_\Dd(\Dd)$ of $\Cc \bowtie \Dd$. In particular, for $(d,c) \in \Dd * \Cc$ we write $dc \coloneqq (d,c)\in \Cc \bowtie \Dd$. So, for example, for $c_i \in \Cc$ and $d_i \in \Dd$,  \[d_1c_1d_2c_2 = d_1(c_1 \la d_2)(c_1 \ra d_2)c_2.\]

\begin{example}
	 If $\Cc$ is a small category, then $(\Cc,\Cc^0)$ is a matched pair with actions $c \la s(c) = r(c)$ and $c \ra s(c) = c$. We have $\Cc \bowtie \Cc^0 \cong \Cc \cong \Cc^0 \bowtie \Cc$.
\end{example}

\begin{example}
	Suppose that $G$ and $H$ are groups and suppose that $(G,H, \la ,\ra)$ is a matched pair. Then $G \bowtie H$ is the Zappa--Sz\'ep product of $G$ and $H$ from \cite{Zap42,Sze50}.
	If $\ra$ is the trivial right action of $H$ on $G$, then for $h_i \in H$ and $g_i \in G$, we have
	\[
	(h_1,g_1)(h_2,g_2) = (h_1(g_1 \la h_2), g_1g_2),
	\]
	so $G \bowtie H$ is the semidirect product $G \ltimes H$.
\end{example}

Zappa--Sz\'ep products have the following universal property.

\begin{prop}	\label{prop:Zappa--Szep_universality}
	Suppose that $(\Cc,\Dd)$ is a matched pair, let $\Aa$ be a small category such that $\Aa^0 = \Cc^0 = \Dd^0$, and suppose that $j_\Cc \colon \Cc \to \Aa$ and $j_\Dd \colon \Dd \to \Aa$ are functors satisfying
	\begin{equation}\label{eq:jCjD behaviour}
		j_\Cc(c)j_\Dd(d) = j_\Dd(c \la d) j_\Cc (c \ra d)
	\end{equation}
	for all $(c, d) \in \Cc * \Dd$. Then there exists a unique functor $j_\Cc \bowtie j_\Dd \colon \Cc \bowtie \Dd \to \Aa$ such that $(j_\Cc \bowtie j_\Dd) \circ \iota_\Cc = j_\Cc$ and $(j_\Cc \bowtie j_\Dd) \circ \iota_\Dd = j_\Dd$. If $\Bb$ is a small category with $\Bb^0 = \Cc^0$ and $k_\Cc : \Cc \to \Bb$ and $k_\Dd : \Dd \to \Bb$ are functors satisfying~\eqref{eq:jCjD behaviour} and with the same universal property, then $k_\Cc \bowtie k_\Dd$ is an isomorphism $\Cc \bowtie \Dd \to \Bb$.
\end{prop}
\begin{proof}
	Define $j_\Cc \bowtie j_\Dd \colon \Cc \bowtie \Dd \to \Aa$ by $(j_\Cc \bowtie j_\Dd)(d,c) = j_\Dd(d) j_\Cc(c)$. Clearly, $(j_\Cc \bowtie j_\Dd) \circ \iota_\Cc = j_\Cc$ and $(j_\Cc \bowtie j_\Dd) \circ \iota_\Dd = j_\Dd$. For functoriality we compute,
\begin{align*}
	(j_\Cc \bowtie j_\Dd)(d_1c_1d_2c_2) &= j_\Dd(d_1)j_\Dd(c_1 \la d_2) j_\Cc(c_1 \ra d_2) j_\Cc(c_2) \\ &= j_\Dd(d_1) j_\Cc(c_1) j_\Dd(d_2) j_\Cc(c_2)
	= (j_\Cc \bowtie j_\Dd)(d_1c_1)\,(j_\Cc \bowtie j_\Dd)(d_2c_2).
\end{align*}
If $f \colon \Cc \bowtie \Dd \to \Aa$ is a functor satisfying $f \circ \iota_\Cc = j_\Cc$ and $f \circ \iota_\Dd = j_\Dd$, then $f(dc) = f(\iota_\Dd(d) \iota_\Cc(c)) = j_\Dd(d)j_\Cc(c) = (j_\Cc \bowtie j_\Dd)(dc)$. If $(\Bb, k_\Cc, k_\Dd)$ has the same universal property, then that universal property applied to $\iota_\Cc$ and $\iota_\Dd$ yields a functor $\theta : \Bb \to \Cc \bowtie \Dd$ inverse to $k_\Cc \bowtie k_\Dd$.
\end{proof}

\begin{cor}\label{cor:matched_pair_morphisms}
The assignment $(\Cc,\Dd) \mapsto \Cc \bowtie \Dd$ is functorial: given a matched-pair morphism $(h^L,h^R) \colon (\Cc,\Dd) \to (\Cc',\Dd')$, there is a functor $h \colon \Cc \bowtie \Dd \to \Cc' \bowtie \Dd'$ such that $h(dc) = h^R(d)h^L(c)$ for all $(d,c) \in \Dd * \Cc$. This functor satisfies $h \circ \iota_{\Cc} = \iota_{\Cc'} \circ  h^L$ and $h \circ \iota_\Dd = \iota_{\Dd'} \circ h^R$.
Conversely, if $h \colon \Cc \bowtie \Dd \to \Cc' \bowtie \Dd'$ is a functor such that $h(\Cc) \subseteq\Cc'$ and $h (\Dd) \subseteq \Dd'$, then $(h|_{\Cc}, h|_{\Dd}) \colon (\Cc,\Dd) \to (\Cc',\Dd')$ is a matched pair morphism.
\end{cor}
\begin{proof}
	To obtain $h$, apply Proposition~\ref{prop:Zappa--Szep_universality} to $\iota_\Cc \circ h^L$ and $\iota_\Dd \circ h^R$. The second statement follows from a one-line calculation.
\end{proof}

As with groups, we can take either an ``external'' or an ``internal'' view of Zappa--Sz\'ep products of categories. Recall that a \emph{wide subcategory} of a category $\Ee$ is a subcategory containing $\Ee^0$.
\begin{dfn}
	A \emph{strict factorisation system} for a category $\Ee$ is a pair $[\Dd,\Cc]$ of wide subcategories of $\Ee$ such that for every $e \in \Ee$ there are unique $d \in \Dd$ and $c \in \Cc$ satisfying $e = dc$.
\end{dfn}

\begin{rmk}
	In~\cite[Theorem~3.8]{RW02} it is shown that strict factorisations systems in categories correspond to distributive laws on monads. 
\end{rmk}

\begin{prop}\label{prop:factorisation_system_is_matched_pair}
Let $(\Cc,\Dd)$ be a matched pair. Then $[\Dd,\Cc]$ is a strict factorisation system for $\Cc \bowtie \Dd$. Conversely, let $[\Dd,\Cc]$ be a strict factorisation system for a small category $\Ee$. For $(c,d) \in \Cc * \Dd$,  let $c \la d \in \Dd$ and $c \ra d \in \Cc$ be the unique elements such that $cd = (c \la d) (c \ra d)$. Then $(\Cc,\Dd,\la,\ra)$ is a matched pair and $(d,c) \mapsto dc$ is an isomorphism $\Cc \bowtie \Dd \cong \Ee$.
\end{prop}

\begin{proof}
Suppose that $(\Cc,\Dd)$ is a matched pair. Since $\Cc \bowtie \Dd=\Dd * \Cc$ as sets, each $e \in \Cc \bowtie \Dd$ factors uniquely as $e = dc$.
	
Conversely, suppose that $[\Dd,\Cc]$ is a unique factorisation system for $\Ee$. Since $c \in \Cc$ uniquely factors as $cs(c) = r(c) c$ we have $c \la s(c) = r(c)$ and $c \ra r(c) = c$. Similarly, $r(d) \la d = d$ and $r(d) \ra d = s(d)$ for all $d \in \Dd$. 
Fix $(c_1,c_2,d_1,d_2) \in \Cc^2 * \Dd^2$. Let $d',\,d'',\,d''' \in \Dd$ and $c',\,c'',\,c''' \in \Cc$ be the unique elements such that $c_1c_2d_1 =d'c'$, $c_2d_1 =d''c''$, and $c_1d'' = d'''c'''$. Then $d'''(c'''c'') = c_1d''c'' = c_1c_2d_1 = d'c'$, so uniqueness of factorisations gives $(c_1c_2) \la d_1 = d' = d''' = c_1 \la d'' = c_1 \la (c_2 \la d_1)$. So $\la$ is an action of $\Cc$ on $\Dd$.
	
Now let $d',\,d'',\,d''' \in \Dd$ and $c',\,c'',\,c''' \in \Cc$ be the unique elements such that $c_2d_1d_2 = d'c'$, $c_2d_1 = d''c''$, and $c''d_2 = d'''c'''$. Then $d''d'''c''' = d'c'$, so uniqueness of factorisations gives
$c_2 \la (d_1d_2) = d' = d'' d''' = (c_2 \la d_1)(c'' \la d_2) = (c_2 \la d_1)((c_2 \ra d_1) \la d_2)$, verifying \ref{itm:ZS-dot}.
	
Symmetrically, $\ra$ defines a right action of $\Dd$ on $\Cc$ satisfying \ref{itm:ZS-diamond}. Condition~\ref{itm:ZS-rs} follows from the composition laws in $\Ee$.
\end{proof}

\begin{rmk}
	Proposition~\ref{prop:factorisation_system_is_matched_pair} says that the internal and external views of Zappa--Sz\'ep products are equivalent. Given a matched pair $(\Cc,\Dd)$ we can equivalently: (a) build the concrete product $\Cc \bowtie \Dd$; or (b) say that $\Ee$ is a Zappa--Sz\'ep product if it contains copies of $\Dd$ and $\Cc$ as wide subcategories such that $[\Dd,\Cc]$ is a strict factorisation system implementing the given actions.
\end{rmk}

For $C^*$-algebraic representations \`a la Spielberg \cite{Spe20} it is important to know when a small category $\Cc$ is \emph{left cancellative} in the sense that if $c_1c_2 = c_1c_3$, then $c_2 = c_3$. The following lemma provides a sufficient condition under which Zappa--Sz\'ep products are left cancellative.
\begin{lem}
	If $(\Cc, \Dd)$ is matched pair in which $\Cc$ and $\Dd$ are both left cancellative and for each $c \in \Cc$ the map $c \la {\cdot} \colon s(c)\Dd \to r(c)\Dd$ is injective, then  $\Cc \bowtie \Dd$ is left cancellative.
\end{lem}
\begin{proof}
	Suppose that $d_1c_1, d_2c_2 \in \Cc \bowtie \Dd$ satisfy $d_1c_1d_2c_2 = d_1c_1d_3c_3$. Then
	$
	d_1(c_1 \la d_2) = d_1(c_1 \la d_3)$ in $\Dd$  and   $(c_1 \ra d_2)c_2 = (c_1 \ra d_3) c_3$ in $\Cc.
	$
	Since $\Dd$ is left cancellative, $c_1 \la d_2 = c_1 \la d_3$, and so injectivity of the left action gives $d_2 = d_3$. Consequently, $(c_1 \ra d_2)c_2 = (c_1 \ra d_2)c_3$. Left cancellation in $\Cc$ implies that $c_2 = c_3$, so $\Cc \bowtie \Dd$ is left cancellative.
\end{proof}

\begin{example}\label{ex:left_canc}
	If $\Cc$ is a groupoid then it acts cancellatively on both itself and $\Dd$ because it has inverses. So if $\Dd$ is left-cancellative then so is $\Cc \bowtie \Dd$.
\end{example}

\subsection{Extending matched pairs to composable tuples}

We define homology for matched pairs in terms of associated categories of composable tuples, so it is important to understand how $\la$ and $\ra$ extend to these categories.

\begin{dfn}
 The \emph{free category} (or \emph{path category}) of a small category $\Cc$ is the category $\Cc^*$ with morphisms $\bigcup_{k\ge 0} \Cc^k$, identity morphisms $\Cc^0$, and composition (for non-identity morphisms) given by concatenation.
\end{dfn}

\begin{rmk}
	There is a subtlety here. The set $\Cc^1$ of 1-tuples in $\Cc^*$ contains the set $\{(v) \mid v \in \Cc^0\}$ of 1-tuples, but this set is disjoint from $\Cc^0 \subseteq \Cc^*$. This is reflected in the composition law: for $v \in \Cc^0$ and $(c_1,\ldots,c_k) \in \Cc^k$ with $r(c_1) = v$, we have $v (c_1,\ldots,c_k) = (c_1,\ldots,c_k)$ while $(v) (c_1,\ldots,c_k) = (v,c_1,\ldots,c_k)$.
\end{rmk}

\begin{lem}\label{lem:bowtie_to_bowtie*}
	Let $(\Cc,\Dd)$ be a matched pair. Define $\bowtie_{n} \colon \Cc * \Dd^n \to \Dd^n * \Cc$ inductively by $\bowtie_1 \coloneqq \bowtie$, and
	\[
	\bowtie_n \coloneqq (1_{\Dd^{n-1}} * \bowtie) \circ (\bowtie_{n-1} * 1_\Dd)
	\]

	for $n \ge 2$. 	Define $\bowtie_* \colon \Cc * \Dd^* \to \Dd^* * \Cc$ by $\restr{{{}\bowtiemap_*{}}\!}{\Cc * \Dd^n} \coloneqq {}\bowtiemap_n{}$. Then
	\begin{enumerate}[label=(\roman*), ref=(\roman*)]
		\item\label{itm:bowtie_order_doesnt_matter} for each $n \ge 1$ and $1 \le p < n$,
			\begin{equation}\label{eq:bowtie_order_doesnt_matter}
				\bowtiemap_n = (1_{\Dd^{n-p}} * \bowtiemap_{p}) \circ (\bowtiemap_{n-p} * 1_{\Dd^{p}}),
			\end{equation}
		\item\label{itm:bowtie*_fact_rule} $\bowtie_*$ is a factorisation rule, and
		\item\label{itm:bowtie*_mult}
		if $\mu_\Dd \colon \Dd^* \to \Dd$ is the map $\mu_\Dd (d_1,d_2,\ldots,d_n) = d_1 d_2\cdots d_n$, then  $(1_\Cc, \mu_\Dd) \colon (\Cc,\Dd^*) \to (\Cc,\Dd)$ is a matched-pair morphism.
	\end{enumerate}
\end{lem}
\begin{proof}
	\ref{itm:bowtie_order_doesnt_matter}
	When $n = 1$, we have $p=1$, so~\eqref{eq:bowtie_order_doesnt_matter} is vacuous. For $n \ge 2$, equation~\eqref{eq:bowtie_order_doesnt_matter} holds for $p = 1$ by definition of $\bowtie_n$. Fix $n_0 \ge 2$ and $1 \le p_0 \le n_0$, and suppose inductively, that~\eqref{eq:bowtie_order_doesnt_matter}
holds for all $n < n_0$ and $1 \le  p < p_0$. In the diagram
	%
	\[\begin{tikzcd}[column sep = 80pt, row sep = 40pt, ampersand replacement=\&]
		{\Cc * \Dd^{n_0}} \& {\Dd^{n_0-p_0}*\Cc*\Dd^{p_0}} \& {\Dd^{n_0}*\Cc} \\
		\& {\Dd^{n_0 - p_0 + 1} *\Cc * \Dd^{p_0 - 1}}
		\arrow["{\bowtiemap_{n_0-p_0} * 1_{\Dd^{p_0}}}", from=1-1, to=1-2]
		\arrow["{1_{\Dd^{n_0 - p_0}}* \bowtiemap_{p_0}}", from=1-2, to=1-3]
		\arrow["{\bowtiemap_{n_0 - p_0 +1} * 1_{\Dd^{p-1}}}"', from=1-1, to=2-2]
		\arrow["{1_{\Dd^{n_0 - p_0 +1}} * \bowtiemap_{p_0-1}}"', from=2-2, to=1-3]
		\arrow[pos=0.3, "1_{\Dd^{n_0-p_0}} * \bowtiemap * 1_{\Dd^{p_0 - 1}}", from=1-2, to=2-2]
	\end{tikzcd}\]
	the left-hand triangle commutes by the inductive definition of $\bowtie_{n_0 - p_0}$, and the right-hand triangle commutes by induction since $p_0 < n_0$. Since $p_0 - 1 < p_0$, the composition of the maps along the bottom of the triangle is
$\bowtie_{n_0}$ by induction. So \eqref{eq:bowtie_order_doesnt_matter} holds for all $n$ and $p$.
	
	\ref{itm:bowtie*_fact_rule} A routine induction verifies~\ref{itm:matching}, and~\ref{itm:units} holds since $\restr{{{}\bowtiemap_*{}}\!}{\Cc * \Dd} = \bowtiemap$. To see that the first diagram of \ref{itm:twist} for $\bowtie_*$ commutes, consider the following diagram.
		\[
	%
	\begin{tikzcd}[column sep = 15pt]
		\Cc^2 * \Dd^{n+1} \arrow[ddd, "\mu_\Cc * 1_{\Dd^{n+1}}"] \arrow[rr, "1_\Cc *\bowtie_{n+1}"] \arrow[rd, "1_\Cc * \bowtie_n * 1_\Dd"] &                                                                                                            & \Cc *
\Dd^{n+1} * \Cc \arrow[rr, "\bowtie_{n+1} * 1_{\Cc}"] \arrow[rd, "\bowtie_n * 1_{\Dd * \Cc}"'] &                                                           & \Dd^{n+1}* \Cc^2 \arrow[ddd, "1_{\Dd^{n+1}} * \mu_\Cc"'] \\
		& \Cc * \Dd^n * \Cc * \Dd \arrow[ru, "1_{\Cc*\Dd^n} * \bowtie"'] \arrow[rd, "\bowtie_n *1_{\Cc * \Dd}"'] &                                                                                                              & \Dd^n * \Cc * \Dd * \Cc
\arrow[ru, "1_{\Dd^n} * \bowtie"] &                                                         \\
		&                                                                                                            & \Dd^n * \Cc^2 * \Dd \arrow[d, "1_{\Dd^{n}} * \mu_\Cc * 1_\Dd"] \arrow[ru, "1_{\Dd^n*\Cc} * \bowtie"']        &
&                                                         \\
		\Cc * \Dd^{n+1} \arrow[rr, "\bowtie_n * 1_\Dd"]                                                                                           &                                                                                                            &
\Dd^n * \Cc * \Dd \arrow[rr, "1_{\Dd^n} * \bowtie"]                                                          &                                                           & \Dd^{n+1} * \Cc
	\end{tikzcd}
	\]
	The central diamond clearly commutes, the top-left and top-right triangles commute by the definition of $\bowtie_{n+1}$, and the bottom-right pentagon commutes by \ref{itm:twist} for $\bowtie$. The composition of the maps along the bottom
row of the diagram is $\bowtie_{n+1}$ by definition. So the whole diagram commutes if and only if the bottom-left pentagon commutes. An induction now shows that the first diagram of \ref{itm:twist} commutes for $\bowtie_*$.

For $m,n \ne 0$, the multiplication map $\mu_{\Dd^*} \colon \Dd^{m} * \Dd^{n} \to \Dd^{m+n}$ is the obvious bijection. So the second diagram of \ref{itm:twist} commutes by \ref{itm:bowtie_order_doesnt_matter}. Hence, $\bowtie_*$ is a
factorisation rule.
	
	\ref{itm:bowtie*_mult} By Lemma~\ref{lem:intertwiner_is_pair} and Corollary~\ref{cor:matched_pair_morphisms}, it suffices to show that $(\mu_\Dd * 1_\Cc) \circ \bowtie_n \colon \Cc * \Dd^n \to \Dd * \Cc$ and $\bowtie \circ (1_\Cc * \mu_\Dd) \colon \Cc * \Dd^n \to \Dd * \Cc$ are equal for all $n$. For $n = 1$ this is trivial, so suppose equality holds for $n-1$, and consider the following diagram.
	\[\begin{tikzcd}[column sep = 70pt, ampersand replacement=\&]
		{\Cc * \Dd^{n}} \& {\Dd^{n-1}*\Cc*\Dd} \& {\Dd^n*\Cc} \\
		{\Cc*\Dd^2} \& {\Dd*\Cc*\Dd} \& {\Dd^2*\Cc} \\
		{\Cc*\Dd} \&\& {\Dd*\Cc}
		\arrow["{\bowtiemap_{n-1}*1_\Dd}", from=1-1, to=1-2]
		\arrow["{1_{\Dd^{n-1}}*\bowtie}", from=1-2, to=1-3]
		\arrow["{1_\Cc * \mu_\Dd * 1_\Dd}"', from=1-1, to=2-1]
		\arrow["{\mu_\Dd * 1_{\Cc*\Dd}}"', from=1-2, to=2-2]
		\arrow["{\mu_\Dd * 1_{\Dd * \Cc}}", from=1-3, to=2-3]
		\arrow["{\bowtiemap * 1_\Dd}", from=2-1, to=2-2]
		\arrow["{1_\Dd * \bowtie}", from=2-2, to=2-3]
		\arrow["{1_\Cc * \mu_\Dd}"', from=2-1, to=3-1]
		\arrow["\bowtiemap", from=3-1, to=3-3]
		\arrow["{\mu_\Dd * 1_\Cc}", from=2-3, to=3-3]
	\end{tikzcd}\]
The bottom pentagon commutes by \ref{itm:twist} for $\bowtie$. The top-right square clearly commutes. The top left square commutes by the inductive hypothesis, and so the whole diagram commutes. The composition along the top row is equal to $\bowtie_n$, and the compositions along the left and right columns are $1_\Cc *\mu_\Dd$ and $\mu_\Dd * 1_\Cc$. So  $(\mu_\Dd * 1_\Cc) \circ \bowtie_n = \bowtie \circ (1_\Cc * \mu_\Dd)$.
\end{proof}

Lemmas~\ref{lem:bowtie_to_bowtie*}~and~\ref{lem:intertwiner_is_pair} imply that $(\Cc, \Dd^*)$ is a matched pair.
The left action of $\Cc$ on $\Dd^{k}$ is given explicitly by
\begin{align*}
	c \la (d_1,\ldots,d_k)
	&\coloneqq (c \la d_1 , (c \ra d_1) \la (d_2,\ldots,d_k) )\\
	&=(c_1 \la d_1,(c_1 \ra d_1) \la d_2, (c_1 \ra (d_1d_2)) \la d_3,\ldots, (c_1 \ra (d_1 \cdots d_{k-1})) \la d_k).
\end{align*}
and the right action of $\Dd^k$ on $\Cc$ is given by
$
	c \ra (d_1,\ldots,d_k) = c \ra (d_1\cdots d_k).
$

We can also define ${{}\bowtiel{n}} \colon \Cc^n * \Dd \to \Dd * \Cc^n$ inductively by ${{}\bowtiel{1}} \coloneqq \bowtie$, and
\[
\bowtiel{n}  \coloneqq  (\bowtie * 1_{\Cc^{n-1}}) \circ (1_{\Cc} * \bowtiel{n-1})
\]
and then $\bowtiel{*} \colon \Cc^* * \Dd \to \Dd * \Cc^*$ by $\restr{\bowtiel{*}}{\Cc^n * \Dd} = \bowtiel{n}$. Lemma~\ref{lem:bowtie_to_bowtie*} applied to opposite categories implies that $\bowtiel{*}$ is a factorisation rule with properties analogous
to those of $\bowtie_*$. The left action of $\Cc^k$ on $\Dd$ is given by
$
(c_1,\ldots,c_k) \la d = (c_1\cdots c_k) \la d
$
and the right action of $\Dd$ on $\Cc^k$ is given by
\begin{align*}
	(c_1,\ldots,c_k) \ra d &= ((c_1,\ldots,c_{k-1}) \ra (c_k \la d)  , c_k \ra d)\\
	&= (c_1 \ra ((c_2 \cdots c_k) \la d),\ldots,c_{k-2} \ra ((c_{k-1}c_{k}) \la d),c_{k-1} \ra (c_k \la d), c_k \ra d).
\end{align*}

Since $(\Cc^*,\Dd)$ is itself a matched pair, $(\Cc^*,\Dd^*)$ can also be equipped with the structure of a matched pair via Lemma~\ref{lem:bowtie_to_bowtie*}.

\begin{prop}\label{prop:bowtie_**}
	Let $(\Cc,\Dd)$ be a matched pair. For $m,n \ge 1$ define $\bowtie_{m,n} \colon \Cc^m * \Dd^n \to \Dd^n * \Cc^m$ inductively by $\bowtie_{1,n} \coloneqq {}\bowtie_n \colon \Cc * \Dd^n \to \Dd^n *\Cc$ and
	\[
	\bowtie_{m,n} \coloneqq (\bowtie_n * 1_{\Cc^{m-1}}) \circ  (1_\Cc * \bowtie_{m-1,n}).
	\]
	Define $\bowtie_{*,*} \colon \Cc^* * \Dd^* \to \Dd^* * \Cc^*$ by ${\bowtie_{*,*}}|^{\mbox{}}_{\Cc^m * \Dd^n} = {\bowtie_{m,n}}$. Then
	\begin{enumerate}[label=(\roman*), ref=(\roman*)]
		\item\label{itm:bowtie**_fact_rule} $\bowtie_{*,*}$ is a factorisation rule, and
		\item\label{itm:bowtie**_mult} $(\mu_\Cc, \mu_\Dd) \colon (\Cc^*,\Dd^*) \to (\Cc,\Dd)$ is a matched-pair morphism.
	\end{enumerate}
\end{prop}
\begin{proof}
	
	That $(\Cc,\Dd^*)$ is a matched pair, together with an additional application of
	Lemma~\ref{lem:bowtie_to_bowtie*}, gives~\ref{itm:bowtie**_fact_rule}.
Two applications of Lemma~\ref{lem:bowtie_to_bowtie*}\ref{itm:bowtie*_mult} give~\ref{itm:bowtie**_mult}.
\end{proof}

We often write $\bowtie \colon \Cc^* * \Dd^* \to \Dd^* * \Cc^*$ for the map $\bowtie_{*,*}$ of Proposition~\ref{prop:bowtie_**}, which implies that $\Cc^* \bowtie \Dd^*$ is a category with strict factorisation system $[\Dd^*,\Cc^*]$. We identify $\Cc^* \bowtie \Dd^*$ as a set with $\Dd^* * \Cc^*$. For each composable $k$-tuple
$\gamma \in \Cc^* \bowtie \Dd^*$ there exist $p_i, q_i \ge 0$ such that $\gamma \in \Dd^{p_1}*\Cc^{q_1}* \cdots * \Dd^{p_k} * \Cc^{q_k}$; its product belongs to $\Dd^{p_1+\cdots+p_k} * \Cc^{q_1 + \cdots + q_k}$.

The map $\bowtie_{m,n}$ can be computed in any order in the following sense.

\begin{cor}\label{cor:bowtie**_order_doesnt_matter}
	Let $(\Cc,\Dd)$ be a matched pair.  For each $1 \le p < m$ and $1 \le q < n$, the diagram
\begin{equation}\label{eq:bowtie**_order_doesnt_matter}
	\begin{tikzcd}[ampersand replacement=\&, column sep=40pt]
	{\Cc^m*\Dd^n} \& {} \& {\Dd^n*\Cc^m} \\
	{\Cc^{m-p}*\Dd^{q}*\Cc^{p}*\Dd^{n-q}} \&\& {\Dd^{q}*\Cc^{m-p}*\Dd^{n-q}*\Cc^p}
	\arrow["{\bowtiemap_{m,n}}", from=1-1, to=1-3]
	\arrow["{1_{\Cc^{m-p}}*\bowtiemap_{p,q}*1_{\Dd^{n-q}}}"', from=1-1, to=2-1]
	\arrow["{1_{\Dd^q}*\bowtiemap_{m-p,n-q}*1_{\Cc^p}}"', from=2-3, to=1-3]
	\arrow["{\bowtiemap_{m-p,q} * \bowtiemap_{p,n-q}}", from=2-1, to=2-3]
\end{tikzcd}
\end{equation}
 commutes.
\end{cor}
\begin{proof}
Elements of $\Cc^m * \Dd^n = \Cc^{m-p} * \Cc^p * \Dd^q * \Dd^{n-q}$ may considered as composable 4-tuples in $\Cc^* \bowtie \Dd^*$. Since the 3-map composition around the bottom of \eqref{eq:bowtie**_order_doesnt_matter} is an iterated product in $\Cc^* \bowtie \Dd^*$, uniqueness of factorisation implies that the diagram commutes.
\end{proof}

Corollary~\ref{cor:bowtie**_order_doesnt_matter} gives
$
\bowtie_{m,n} = (1_{\Dd^{n-1}}*\bowtie_{m,1}) \circ (\bowtie_{m,n-1} * 1_\Dd).
$ So we could also have applied Lemma~\ref{lem:bowtie_to_bowtie*} to $(\Cc,\Dd^*)$ to obtain the matched-pair structure on $(\Cc^*,\Dd^*)$ of Proposition~\ref{prop:bowtie_**}.

\subsection{Model matched pairs}

We introduce a class of model categories that will play a central role in our computation of homology (Theorem~\ref{thm:cohomologies_are_the_same}) below.

Let $X_n \coloneqq \{(p,q) \in \NN \times \NN \mid 0\le p+q \le n\}$. We denote elements of $X_n$ using bold font. Given $\mathbf{a} \in X_n$ we write $\mathbf{a} = (a_L,a_R)$ to indicate the left and right coordinates of $\mathbf{a}$.

\begin{dfn}
	Let $\Gamma_n = \{(\mathbf{a},\mathbf{b}) \in X_n \times X_n \mid a_L \le b_L \text{ and } a_R \ge b_R\}$. Define $r,s \colon \Gamma_n \to X_n$ by
	$
	r(\mathbf{a},\mathbf{b}) = \mathbf{a}$ and $s(\mathbf{a},\mathbf{b}) = \mathbf{b}.
	$
	Identify $X_n$ with $\{(\mathbf{a},\mathbf{a}) \mid \mathbf{a} \in X_n\}$.
\end{dfn}
With composition defined by $(\mathbf{a},\mathbf{b}) (\mathbf{b},\mathbf{c}) \coloneqq (\mathbf{a},\mathbf{c})$, the set $\Gamma_n$ is a small category. It can also be realised as the Zappa--Sz\'ep product of the path categories of two graphs. Let $E_n$ be the directed graph with $E_n^0 = X_n$ and $E_n^1 = \{e_{p,q}\colon (p,q) \in X_n \text{ and } p+q < n\}$, with $r(e_{p,q}) = (p,q)$ and $s(e_{p,q}) = (p+1,q)$. Let $F_n$ be the directed graph with $F_n^0 = X_n$ and $F_n^1 = \{f_{p,q}\colon (p,q) \in X_n \text{ and } p+q < n\}$, with $s(f_{p,q}) = (p,q)$ and $r(f_{p,q}) = (p,q+1)$. We draw $E_n$ and $F_n$ using coloured arrows (blue and solid for $E$, red and dashed for $F$).
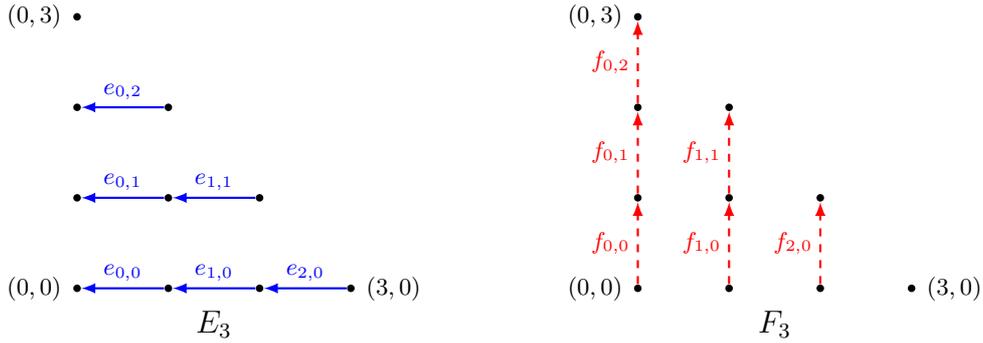
\[
\begin{tikzpicture}
	[vertex/.style={circle, fill=black, inner sep=1pt},
	edge/.style={-latex,thick,blue},
	edgev/.style={-latex,thick,dashed,red},
	scale = 1.2]
	
	\node at (1.5,-0.4) {$E_3$};%
	
	
	\node[vertex, label={ left:{\scriptsize{$(0,0)$}}}] (0-0) at (0,0) {};%
	\node[vertex] (1-0) at (1,0) {};%
	\node[vertex] (2-0) at (2,0) {};%
	\node[vertex, label={right:{\scriptsize{$(3,0)$}}}] (3-0) at (3,0) {};%
	
	\node[vertex] (0-1) at (0,1) {};%
	\node[vertex] (1-1) at (1,1) {};%
	\node[vertex] (2-1) at (2,1) {};%
	
	\node[vertex] (0-2) at (0,2) {};%
	\node[vertex] (1-2) at (1,2) {};%
	
	\node[vertex, label={left:{\scriptsize{$(0,3)$}}}] (0-3) at (0,3) {};%

	
	\draw[edge] (1-0) to node[midway, anchor=south,inner sep=2.5pt] {\scriptsize{$e_{0,0}$}} (0-0);
	\draw[edge] (2-0) to node[midway, anchor=south,inner sep=2.5pt] {\scriptsize{$e_{1,0}$}} (1-0);
	\draw[edge] (3-0) to node[midway, anchor=south,inner sep=2.5pt] {\scriptsize{$e_{2,0}$}} (2-0);
	
	\draw[edge] (1-1) to node[midway, anchor=south,inner sep=2.5pt] {\scriptsize{$e_{0,1}$}} (0-1);
	\draw[edge] (2-1) to node[midway, anchor=south,inner sep=2.5pt] {\scriptsize{$e_{1,1}$}} (1-1);
	
	\draw[edge] (1-2) to node[midway, anchor=south,inner sep=2.5pt] {\scriptsize{$e_{0,2}$}} (0-2);
	
\end{tikzpicture}
\qquad
\qquad
\begin{tikzpicture}
	[vertex/.style={circle, fill=black, inner sep=1pt},
	edge/.style={-latex,thick,blue},
	edgev/.style={-latex,thick,dashed,red},
	scale = 1.2]
	
	\node at (1.5,-0.4) {$F_3$};%
	
	
	\node[vertex, label={ left:{\scriptsize{$(0,0)$}}}] (0-0) at (0,0) {};%
	\node[vertex] (1-0) at (1,0) {};%
	\node[vertex] (2-0) at (2,0) {};%
	\node[vertex, label={right:{\scriptsize{$(3,0)$}}}] (3-0) at (3,0) {};%
	
	\node[vertex] (0-1) at (0,1) {};%
	\node[vertex] (1-1) at (1,1) {};%
	\node[vertex] (2-1) at (2,1) {};%
	
	\node[vertex] (0-2) at (0,2) {};%
	\node[vertex] (1-2) at (1,2) {};%
	
	\node[vertex, label={left:{\scriptsize{$(0,3)$}}}] (0-3) at (0,3) {};%

	
	\draw[edgev] (0-0) to node[midway, anchor=east,inner sep=2.5pt] {\scriptsize{$f_{0,0}$}} (0-1);
	\draw[edgev] (0-1) to node[midway, anchor=east,inner sep=2.5pt] {\scriptsize{$f_{0,1}$}} (0-2);
	\draw[edgev] (0-2) to node[midway, anchor=east,inner sep=2.5pt] {\scriptsize{$f_{0,2}$}} (0-3);
	
	\draw[edgev] (1-0) to node[midway, anchor=east,inner sep=2.5pt] {\scriptsize{$f_{1,0}$}} (1-1);
	\draw[edgev] (1-1) to node[midway, anchor=east,inner sep=2.5pt] {\scriptsize{$f_{1,1}$}} (1-2);
	
	\draw[edgev] (2-0) to node[midway, anchor=east,inner sep=2.5pt] {\scriptsize{$f_{2,0}$}} (2-1);
\end{tikzpicture}
\]
Let $\Ff_n \coloneqq F_n^*$ and $\Ee_n \coloneqq E_n^*$ denote the path categories of $F_n$ and $E_n$, respectively.

\begin{lem}
	The subcategory $\{((p,q),(p',q)) \mid p \le p' \le n, \, q \le n\}$ of $\Gamma_n$ is isomorphic to $\Ee_n$ and the subcategory $\{((p,q),(p,q'))\mid p \le n, \, n \ge q \ge q' \}$ is isomorphic to $\Ff_n$. Moreover, $[\Ff_n, \Ee_n]$ is a strict factorisation system for $\Gamma_n$; the pair $(\Ee_n,\Ff_n)$ is a matched pair, with
	\[
	e_{p,q} \la f_{p+1,q-1} = f_{p,q-1} \quad \text{and} \quad 	e_{p,q} \ra f_{p+1,q-1} = e_{p,q-1};
	\]
	and $\Gamma_n \cong \Ee_n \bowtie \Ff_n$.
\end{lem}
\begin{proof}
	Since $\Ee_n$ is freely generated by $E_n$, the map $e_{p,q} \mapsto ((p,q),(p+1,q))$ identifies $\Ee_n$ with $\{((p,q),(p',q)) \mid p \le p'\}$. Similarly, $\Ff_n \cong \{((p,q),(p,q'))\mid q \ge q'\}$ via $f_{p,q} \mapsto ((p,q+1),(p,q))$.
	Both $\Ee_n$ and $\Ff_n$ are wide subcategories. For each $(\mathbf{a},\mathbf{b}) \in \Gamma_n$, the elements $\alpha \coloneqq ((a_L,a_R),(a_L,b_R))$ and $\beta \coloneqq ((a_L,b_R),(b_L,b_R))$ are the unique elements of $\Ff_n$ and $\Ee_n$, respectively such that $(\mathbf{a},\mathbf{b}) = \alpha \beta$. So $[\Ff_n, \Ee_n]$ is a strict factorisation system for $\Gamma_n$.
	
	The remaining statements follow from Proposition~\ref{prop:factorisation_system_is_matched_pair}.
\end{proof}

\begin{dfn}
	We refer to the matched pairs $(\Ee_n, \Ff_n)$ as \emph{model matched pairs}.
\end{dfn}

Each $\Gamma_n$ can be visualised as a commuting diagram incorporating both $E_n$ and $F_n$.
\begin{equation}
	\label{eq:Gammas}
	\begin{tikzpicture}
		[vertex/.style={circle, fill=black, inner sep=1pt},
		edge/.style={-latex,thick,blue},
		edgev/.style={-latex,thick,dashed,red},
		scale = 1.1,
		baseline=0pt]
		
		\node at (0,-0.4) {$\Gamma_0$};%
		
		
		\node[vertex, label={left:{\scriptsize{$(0,0)$}}}] (0-0) at (0,0) {};%
	\end{tikzpicture}
	\,\,
	\begin{tikzpicture}
		[vertex/.style={circle, fill=black, inner sep=1pt},
		edge/.style={-latex,thick,blue},
		edgev/.style={-latex,thick,dashed,red},
		scale = 1.1,
		baseline=0pt]
		
		\node at (0.5,-0.4) {$\Gamma_1$};%
		
		
		\node[vertex, label={ left:{\scriptsize{$(0,0)$}}}] (0-0) at (0,0) {};%
		\node[vertex, label={right:{\scriptsize{$(1,0)$}}}] (1-0) at (1,0) {};%
		
		\node[vertex, label={left:{\scriptsize{$(0,1)$}}}] (0-1) at (0,1) {};%
		
		
		\draw[edge] (1-0) to node[midway, anchor=south,inner sep=2.5pt] {\scriptsize{$e_{0,0}$}} (0-0);
		
		
		\draw[edgev] (0-0) to node[midway, anchor=east,inner sep=2.5pt] {\scriptsize{$f_{0,0}$}} (0-1);
	\end{tikzpicture}
	\,\,
	\begin{tikzpicture}
		[vertex/.style={circle, fill=black, inner sep=1pt},
		edge/.style={-latex,thick,blue},
		edgev/.style={-latex,thick,dashed,red},
		scale = 1.1,
		baseline=0pt]
		
		\node at (1,-0.4) {$\Gamma_2$};%
		
		
		\node[vertex, label={ left:{\scriptsize{$(0,0)$}}}] (0-0) at (0,0) {};%
		\node[vertex] (1-0) at (1,0) {};%
		\node[vertex, label={right:{\scriptsize{$(2,0)$}}}] (2-0) at (2,0) {};%
		
		\node[vertex] (0-1) at (0,1) {};%
		\node[vertex] (1-1) at (1,1) {};%
		
		\node[vertex, label={left:{\scriptsize{$(0,2)$}}}] (0-2) at (0,2) {};%
		
		
		\draw[edge] (1-0) to node[midway, anchor=south,inner sep=2.5pt] {\scriptsize{$e_{0,0}$}} (0-0);%
		\draw[edge] (2-0) to node[midway, anchor=south,inner sep=2.5pt] {\scriptsize{$e_{1,0}$}} (1-0);%
		
		\draw[edge] (1-1) to node[midway, anchor=south,inner sep=2.5pt] {\scriptsize{$e_{0,1}$}} (0-1);%

		
		\draw[edgev] (0-0) to node[midway, anchor=east,inner sep=2.5pt] {\scriptsize{$f_{0,0}$}} (0-1);%
		\draw[edgev] (0-1) to node[midway, anchor=east,inner sep=2.5pt] {\scriptsize{$f_{0,1}$}} (0-2);%
		
		\draw[edgev] (1-0) to node[midway, anchor=east,inner sep=2.5pt] {\scriptsize{$f_{1,0}$}} (1-1);%
	\end{tikzpicture}
	\,\,
	\begin{tikzpicture}
		[vertex/.style={circle, fill=black, inner sep=1pt},
		edge/.style={-latex,thick,blue},
		edgev/.style={-latex,thick,dashed,red},
		scale = 1.1,
		baseline=0pt]
		
		\node at (1.5,-0.4) {$\Gamma_3$};%
		
		
		\node[vertex, label={ left:{\scriptsize{$(0,0)$}}}] (0-0) at (0,0) {};%
		\node[vertex] (1-0) at (1,0) {};%
		\node[vertex] (2-0) at (2,0) {};%
		\node[vertex, label={right:{\scriptsize{$(3,0)$}}}] (3-0) at (3,0) {};%
		
		\node[vertex] (0-1) at (0,1) {};%
		\node[vertex] (1-1) at (1,1) {};%
		\node[vertex] (2-1) at (2,1) {};%
		
		\node[vertex] (0-2) at (0,2) {};%
		\node[vertex] (1-2) at (1,2) {};%
		
		\node[vertex, label={left:{\scriptsize{$(0,3)$}}}] (0-3) at (0,3) {};%
		
		
		\draw[edge] (1-0) to node[midway, anchor=south,inner sep=2.5pt] {\scriptsize{$e_{0,0}$}} (0-0);%
		\draw[edge] (2-0) to node[midway, anchor=south,inner sep=2.5pt] {\scriptsize{$e_{1,0}$}} (1-0);%
		\draw[edge] (3-0) to node[midway, anchor=south,inner sep=2.5pt] {\scriptsize{$e_{2,0}$}} (2-0);%
		
		\draw[edge] (1-1) to node[midway, anchor=south,inner sep=2.5pt] {\scriptsize{$e_{0,1}$}} (0-1);%
		\draw[edge] (2-1) to node[midway, anchor=south,inner sep=2.5pt] {\scriptsize{$e_{1,1}$}} (1-1);%
		
		\draw[edge] (1-2) to node[midway, anchor=south,inner sep=2.5pt] {\scriptsize{$e_{0,2}$}} (0-2);%
		
		
		\draw[edgev] (0-0) to node[midway, anchor=east,inner sep=2.5pt] {\scriptsize{$f_{0,0}$}} (0-1);%
		\draw[edgev] (0-1) to node[midway, anchor=east,inner sep=2.5pt] {\scriptsize{$f_{0,1}$}} (0-2);%
		\draw[edgev] (0-2) to node[midway, anchor=east,inner sep=2.5pt] {\scriptsize{$f_{0,2}$}} (0-3);%
		
		\draw[edgev] (1-0) to node[midway, anchor=east,inner sep=2.5pt] {\scriptsize{$f_{1,0}$}} (1-1);%
		\draw[edgev] (1-1) to node[midway, anchor=east,inner sep=2.5pt] {\scriptsize{$f_{1,1}$}} (1-2);%
		
		\draw[edgev] (2-0) to node[midway, anchor=east,inner sep=2.5pt] {\scriptsize{$f_{2,0}$}} (2-1);%
	\end{tikzpicture}
\end{equation}

For each $n$, we draw $E_n$ and $F_n$ on the same vertex set. Each picture in~\eqref{eq:Gammas} is a commuting
diagram in the corresponding $\Gamma_n$. A morphism $(\mathbf{a}, \mathbf{b}) \in \Gamma_n$ is equal to the composition
of any of the paths in~\eqref{eq:Gammas} from the vertex at $\mathbf{b}$ to the one at $\mathbf{a}$.

The matched pairs $(\Ee_n,\Ff_n)$ are---in the following sense---\emph{free} in the category $\MP$.
\begin{lem}\label{lem:matched_pair_universality}
	Let $(\Cc,\Dd)$ be a matched pair. For every $\gamma = (d_0c_0,\ldots,d_{n-1}c_{n-1}) \in (\Cc \bowtie \Dd)^{n}$, there is a unique matched pair morphism $h_{\gamma} \colon (\Ee_n, \Ff_n) \to (\Cc,\Dd)$ such that for all $0 \le k < n$,
	\begin{equation}\label{eq:free_h_definition}
	h_\gamma^L(e_{k,n-1-k}) = c_k \quad \text{and} \quad h_{\gamma}^R(f_{k,n-1-k}) = d_k.
	\end{equation}
	Moreover, every matched pair morphism $(\Ee_n,\Ff_n) \to (\Cc,\Dd)$ is of this form.
\end{lem}
\begin{proof}
	For each $0 \le k < n$ let $d_{k,n-1-k} \coloneqq d_k$ and $c_{k,n-1-k} \coloneqq c_k$, and for each $0 \le p + q < n-1$ define $d_{p,q} \in \Dd$ and $c_{p,q} \in \Cc$ inductively by
	$d_{p,q} \coloneqq c_{p,q+1} \la d_{p+1,q}$ and $c_{p,q} \coloneqq c_{p,q+1} \ra d_{p+1,q}$.
	Since $\Ee_n$ is freely generated by edges, there is a unique functor $h_\gamma^L \colon \Ee_n \to \Cc$ satisfying $h_\gamma^L(e_{p,q}) = c_{p,q}$ for all $0 \le p+q < n$. Similarly there is a unique functor $h_\gamma^R \colon \Ff_n \to \Dd$ satisfying $h_\gamma^R(f_{p,q}) = d_{p,q}$. Since $d_{p,q}c_{p,q} = c_{p,q+1}d_{p+1,q}$, it follows that $h_\gamma = (h_\gamma^L,h_\gamma^R)$ is a morphism of matched pairs.
	
	For uniqueness, fix a matched pair morphism $h \colon (\Ee_n, \Ff_n) \to (\Cc,\Dd)$. Then $h^L (e_{p,q}) = h^L(e_{p,q+1} \ra f_{p+1,q}) = h^L(e_{p,q+1}) \ra h^R(f_{p+1,q})$ and similarly $h^R (f_{p,q}) = h^L(e_{p,q+1}) \la h^R(f_{p+1,q})$. Since
	$h^L \colon \Ee_n \to \Cc$ and $h^R \colon \Ff_n \to \Dd$ are functors, $h^L$ and $h^R$ are determined by the values $h^L(e_{k,n-1-k})$ and $h^R(f_{k,n-1-k})$ for $0 \le k < n$. So $h_\gamma$ is uniquely determined by \eqref{eq:free_h_definition}.
\end{proof}

Corollary~\ref{cor:matched_pair_morphisms} says that $h_\gamma$ is the unique functor $\Gamma_n \to \Cc \bowtie \Dd$ such that $h_{\gamma}(f_{k,n-1-k}e_{k,n-1-k}) = d_kc_k$ for all $0 \le k < n$.

\subsection{Further examples}\label{subsection:examples}

\subsubsection{$k$-graphs}
Here we describe $k$-graphs \cite{KP00} using matched pairs. The generalisations of higher-rank graphs of \cite{LV22} also fit into our framework, but we do not discuss them here.

\begin{dfn}
	A \emph{$k$-graph} is a countable category $\Lambda$ together with a functor $d \colon \Lambda \to \NN^k$, called the \emph{degree map}, which satisfies the following factorisation property: if $d(\lambda) = m+n$, then there exist
	unique elements  $\mu,\nu \in \Lambda$ such that $\lambda = \mu\nu$, $d(\mu) = m$ and $d(\nu) = n$. For each $n \in \NN^k$, we define $\Lambda^n \coloneqq d^{-1}(n)$.
\end{dfn}

We show that every $(k_1 + k_2)$-graph is a Zappa--Sz\'ep product of a $k_1$-graph and a $k_2$-graph.

\begin{lem}\label{lem:k+lgraph->ZSp}
	Fix $k_1, k_2 \in \NN$ and let $\Sigma$ be a $(k_1 + k_2)$-graph. Let $\Lambda \coloneqq d^{-1}(\NN^{k_1} \times \{0\})$ regarded as a $k_1$-graph, and let $\Gamma \coloneqq  d^{-1}(\{0\} \times \NN^{k_2})$ regarded as a $k_2$-graph. There are unique actions $\la$ of $\Lambda$ on $\Gamma$ and $\ra$ of $\Gamma$ on $\Lambda$ such that $\lambda \gamma = (\lambda \la \gamma) (\lambda \ra \gamma)$ in $\Sigma$ for all composable pairs $(\lambda,\gamma) \in \Lambda * \Gamma$. These make $(\Lambda, \Gamma,\la,\ra)$ a matched pair, and $(\gamma,\lambda) \mapsto \gamma\lambda$ is an isomorphism $\Lambda \bowtie \Gamma \to \Sigma$. We have $d(\lambda \la \gamma) = d(\gamma)$ and $d(\lambda \ra \gamma) = d(\lambda)$ for all $\lambda,\gamma$.
\end{lem}
\begin{proof}
	Everything except the final statement follows the factorisation property and Proposition~\ref{prop:factorisation_system_is_matched_pair}. The final statement follows from the factorisation property.
\end{proof}

We now describe a converse to Lemma~\ref{lem:k+lgraph->ZSp}. An \emph{edge} in a $k$-graph is a path $e$ such that $d(e)$ is a standard generator of $\NN^k$. We write $E(\Lambda)$ for the set of edges of $\Lambda$. Let $\Sigma, \Lambda, \Gamma, \la$ and $\ra$ be as in Lemma~\ref{lem:k+lgraph->ZSp}, and write $d_\Lambda \colon \Lambda \to \NN^{k_1}$ and $d_\Gamma \colon \Gamma \to \NN^{k_2}$ for the degree functors. Then
\begin{enumerate}[label=(K\arabic*), ref=(K\arabic*)]
	\item\label{itm:k-rs} $s(\nu \la \mu) = r(\nu \ra \mu)$ for all $(\nu,\mu) \in E(\Lambda) * E(\Gamma)$,
	\item\label{itm:k-rasplit} $\nu_1 \nu_2 \ra \mu = (\nu_1 \ra (\nu_2 \la \mu))(\nu_2 \ra \mu)$ for all $(\nu_1,\nu_2,\mu) \in E(\Lambda)*E(\Lambda)*E(\Gamma)$,
	\item\label{itm:k-lasplit} $\nu \la \mu_1\mu_2 = (\nu \la \mu_1)((\nu \ra \mu_1) \la \mu_2)$ for all $(\nu,\mu_1,\mu_2) \in E(\Lambda) * E(\Gamma) * E(\Gamma)$,
	\item\label{itm:k-degree} $d_\Gamma(\nu \la \mu) = d_\Gamma(\mu)$ and $d_\Lambda(\nu \ra \mu) = d_\Lambda(\nu)$ for all $(\nu,\mu) \in E(\Lambda) * E(\Gamma)$, and
	\item\label{itm:k-injective} for each $(\mu,\nu) \in E(\Gamma)* E(\Lambda)$ there exist unique $\mu' \in E(\Gamma)$ and $\nu' \in E(\Lambda)$ such that $\mu = \nu' \la \mu'$ and $\nu = \nu' \ra \mu'$.
\end{enumerate}

\begin{lem}\label{lem:ZSp->k+lgraph}
	Let $\Lambda$ be a $k_1$-graph and let $\Gamma$ be a $k_2$-graph such that $\Lambda^0 = \Gamma^0$. Suppose that $\la \colon \Lambda  * \Gamma \to \Gamma$ and $\ra \colon \Lambda * \Gamma \to \Gamma$ are actions satisfying \ref{itm:k-rs}--\ref{itm:k-injective}. Then $(\Lambda, \Gamma)$ is a matched pair, and the map $d\colon\Lambda \bowtie \Gamma \to \NN^{k_1 + k_2}$ given by $d(\gamma, \lambda) = (d_\Lambda(\lambda) , d_\Gamma(\gamma))$ makes $\Lambda \bowtie \Gamma$ into a $(k_1 + k_2)$-graph.
\end{lem}
\begin{proof}
	Let $E$ be directed graph with edges $E^1 = E(\Lambda) \sqcup E(\Gamma)$, vertices $E^0 = \Lambda^0 = \Gamma^0$ and range and source maps inherited from $\Lambda$ and $\Gamma$. Define $c\colon E^1 \to \{1, \dots, k_1+k_2\}$ by $c(\alpha) = i$ if $\alpha \in \Lambda^{e_i}$ and $c(\alpha) = k_1+j$ if $\alpha \in \Gamma^{e_j}$, which we regard as a colouring of $E^1$ by $k_1 + k_2$ colours. Define a collection of squares in the sense of~\cite{HRSW13, LV22b} by $\alpha\beta \sim \beta'\alpha'$ if
	\begin{itemize}
		\item $\alpha\beta = \beta'\alpha'$ in one of $\Lambda$ or $\Gamma$, or
		\item $\alpha \in E(\Lambda)$ and $\beta \in E(\Gamma)$ and $\alpha \la \beta = \beta'$ and $\alpha \ra\beta = \alpha'$, or
		\item $\alpha \in E(\Gamma)$ and $\beta \in E(\Lambda)$ and $\alpha' \la \beta' = \beta$ and $\alpha' \ra\beta' = \alpha$.
	\end{itemize}
	The factorisation properties and~\ref{itm:k-injective} ensure that this is a complete collection of squares.
	
	We claim that this is an associative collection of squares. For this we must check that if $\alpha,\beta,\gamma \in E^1$ are composable and of distinct colours, and if
	\begin{align*}
		\alpha\beta &\sim \beta_1\alpha_1, & \alpha_1\gamma &\sim \gamma_1\alpha_2, && \text{ and}& \beta_1\gamma_1 &\sim \gamma_2\beta_2; && \text{ and} \\
		\beta\gamma &\sim \gamma^1\beta^1, & \alpha\gamma^1 &\sim \gamma^2\alpha^1,  &&\text{ and}& \alpha^1\beta^1 &\sim \beta^2\alpha^2;
	\end{align*}
	then $\alpha^2 = \alpha_2$, $\beta^2 = \beta_2$ and $\gamma^2 = \gamma_2$.
	
	If $\alpha,\beta,\gamma$ all belong to either $\Lambda$ or $\Gamma$, this follows from associativity of composition, so we just need to consider when this is not the case. We treat the case where $\alpha \in \Lambda$ and $\beta, \gamma \in \Gamma$; the calculations for the other cases are similarly straightforward. We have
	\[
	\beta_1 = \alpha \la \beta,\quad \alpha_1 = \alpha \ra \beta,\quad \gamma_1 = \alpha_1 \la \gamma,\quad \alpha_2 = \alpha_1 \ra \gamma, \text{ and } \quad
	\beta_1 \gamma_1 = \gamma_2 \beta_2\text{ in $\Gamma$}.
	\]
	That is, $\gamma_2\beta_2 = (\alpha \la \beta)((\alpha \ra \beta) \la \gamma) = \alpha \la (\beta\gamma)$, and $\alpha_2 = (\alpha \ra \beta) \ra \gamma = \alpha \ra (\beta\gamma)$.
	Similarly,
	\[
	\beta\gamma = \gamma^1\beta^1, \quad \gamma^2 = \alpha\la\gamma^1,\quad \alpha^1 = \alpha\ra\gamma^1,\quad \beta^2 = \alpha^1\la\beta^1, \text{ and } \quad \alpha^2 = \alpha^1\ra\beta^1.
	\]
	That is, $\gamma^2\beta^2 = \alpha \la (\gamma^1\beta^1) = \alpha \la (\beta\gamma)$, and $\alpha^2 = (\alpha\ra\gamma^1) \ra \beta^1 = \alpha \ra (\gamma^1\beta^1) = \alpha \ra (\beta\gamma)$. So $\gamma^2\beta^2 = \gamma_2\beta_2$ forcing $\gamma^2 = \gamma_2$ and $\beta^2 = \beta_2$ by uniqueness of factorisations in $\Gamma$, and $\alpha^2 = \alpha_2$.
	
	By \cite[Theorem~4.4]{HRSW13} there is a unique $(k_1 + k_2)$-graph $\Sigma$ with skeleton $E$ and the specified factorisation rules. Lemma~\ref{lem:k+lgraph->ZSp}, yields a $k_1$-graph $\Lambda'$ and a $k_2$-graph $\Gamma'$ such that $\Sigma \cong \Lambda' \bowtie \Gamma'$. By construction, $\Lambda'$ has the same skeleton and factorisation rules as $\Lambda$ so they are isomorphic by \cite[Theorem~4.5]{HRSW13} (see also \cite{LV22b}), and likewise $\Gamma' \cong \Gamma$. These isomorphisms intertwine the actions of $\Lambda'$ and $\Gamma'$ on one another with those of $\Lambda$ and $\Gamma$.
\end{proof}

Taken together, Lemmas~\ref{lem:k+lgraph->ZSp} and \ref{lem:ZSp->k+lgraph} prove the following.

\begin{prop}
	Let $\Gamma$ be a $k_1$-graph and let $\Lambda$ be a $k_2$-graph with actions $\la \colon \Lambda  * \Gamma \to \Gamma$ and $\ra \colon \Lambda * \Gamma \to \Gamma$ satisfying \ref{itm:k-rs}--\ref{itm:k-injective}. The Zappa--Sz\'ep product $\Lambda \bowtie \Gamma$ is a $(k_1 + k_2)$-graph. Moreover, every $(k_1 + k_2)$-graph $\Pi$ is isomorphic to the Zappa--Sz\'ep product of the $k_1$-graph $\Lambda = d_{\Pi}^{-1}(\NN^{k+1} \times \{0\})$ and the $k_2$-graph $\Gamma = d_{\Pi}^{-1}(\{0\} \times \NN^{k_2})$ with actions satisfying \ref{itm:k-rs}--\ref{itm:k-injective}.
\end{prop}

\subsubsection{Self-similar actions}
We discuss self-similar actions of groupoids on $k$-graphs as in \cite{ABRW}. These include self-similar actions of groupoids and of groups on graphs as in  \cite{Nek05, EP17, LRRW14, LRRW18}. We show that each such self-similar action determines a matched pair in which the left action respects the degree map. Later we will study $C^*$-algebras associated to such matched pairs; the framework of matched pairs allows us to dispense with the faithfulness condition traditionally imposed in the study of self-similar actions.

Recall that an edge in a $k$-graph is a path $f$ with $d(f) = e_i$ for some $i \le k$.

\begin{dfn}[{\cite[Definition 3.3]{LRRW18}}] \label{dfn:selfsimilaraction}
Let $\Lambda$ be a $k$-graph and let $\Gg$ be a groupoid with $\Gg^0 = \Lambda^0$. A \emph{faithful self-similar action} of $\Gg$ on $\Lambda$ is a left action $\cdot \colon \Gg * \Lambda \to \Lambda$ of $\Gg$ on $\Lambda$ such that
	\begin{enumerate}[labelindent=0pt,labelwidth=\widthof{\ref{itm:ssa-ss}},label=(SSA\arabic*), ref=(SSA\arabic*),leftmargin=!]
		\item \label{itm:ssa-length} for each $n \in \NN^k$ and $g \in \Gg$, we have $g \cdot (s(g) \Lambda^n) = r(g) \Lambda^n$,
        \item \label{itm:ssa-faithful} if $s(g_1) = s(g_2)$ and $g_1 \cdot \mu = g_2 \cdot \mu$ for all $\mu \in s(g_1) \Lambda$, then $g_1 = g_2$, and
        \item \label{itm:ssa-ss} for every $g \in \Gg$ and every edge $e \in s(g)\Lambda$ there exists $h \in s(e)\Gg$ such that $g \cdot (e\mu) = (g \cdot e)(h \cdot \mu)$ for all $\mu \in s(e) \Lambda$.
	\end{enumerate}
\end{dfn}

\begin{rmk}
If $\Gg$ acts self-similarly on $\Lambda$ then \ref{itm:ssa-faithful} implies that there is a unique $h$ satisfying \ref{itm:ssa-ss}. We denote this element by $\restr{g}{e}$ and call it the \emph{restriction} of $g$ to $\mu$.
\end{rmk}

As discussed immediately after Definition~3.3 in \cite{ABRW}, the map $g \mapsto \restr{g}{e}$ extends to a map $(g, \mu) \mapsto \restr{g}{\mu}$ from $\Gg * \Lambda$ to $\Lambda$  by the recursive formula $\restr{g}{e \mu} = \restr{(\restr{g}{e})}{(\restr{g}{e})\cdot\mu}$.

\begin{prop}\label{prop:ssa_is_matched_pair}
Let $\Lambda$ be a $k$-graph and $\Gg$ a groupoid with $\Gg^0 = \Lambda^0$. Suppose that $\cdot \colon \Gg * \Lambda \to \Lambda$ is a faithful self-similar action. Define $\la \colon \Gg * \Lambda \to \Lambda$  by $g \la \mu = g \cdot \mu$ and $\ra \colon \Gg * \Lambda \to \Gg$ by $g \ra \mu = \restr{g}{\mu}$. Then $(\Gg, \Lambda, \la ,\ra)$ is a matched pair such that
	\begin{enumerate}
		\item\label{itm:ssa_la_faithful} if $g_1 \la \mu = g_2 \la \mu$ for all $\mu \in s(g_1)\Lambda$, then $g_1= g_2$, and
		\item\label{itm:ssa_length}  $d(g \la \mu) = d(\mu)$ for all $(g,\mu) \in \Gg * \Lambda$.
	\end{enumerate}
	Conversely, if $(\Gg, \Lambda, \la ,\ra)$ is a matched pair satisfying \ref{itm:ssa_la_faithful} and \ref{itm:ssa_length}, then $\la$ defines a faithful self-similar action of $\Gg$ on $\Lambda$ with restriction map $\restr{g}{\mu} \coloneqq g \ra \mu$.
\end{prop}
\begin{proof}
	First suppose that $\cdot \colon \Gg * \Lambda \to \Lambda$ is a faithful self-similar action. That $d(g \la \mu) = d(\mu)$ for all $(g,\mu) \in \Gg * \Lambda$ follows from \ref{itm:ssa-length},
	and \cite[Lemma~3.4]{ABRW} implies that $(\Gg,\Lambda)$ is a matched pair. Condition~\ref{itm:ssa_la_faithful} follows from
\ref{itm:ssa-faithful}.
	
	Now, suppose that $(\Gg,\Lambda)$ is a matched pair satisfying \ref{itm:ssa_la_faithful}~and~\ref{itm:ssa_length}. Then for each $g \in \Gg$ and each edge $e \in s(g)\Lambda$, the element $h \coloneqq g \la \mu$ satisfies~\ref{itm:ssa-ss}. Condition~\ref{itm:ZS-dot} gives~\ref{itm:ssa-ss}. That $d(g \la \mu) = d(\mu)$ for all $(g,\mu) \in \Gg * \Lambda$ implies that $g \la \cdot$ restricts to a map $g \la \cdot \colon s(g) \Lambda^n \to r(g) \Lambda^n$. Invertibility of $g$ implies that $g \la \cdot$ is bijective, giving \ref{itm:ssa-length}. Condition~\ref{itm:ssa_la_faithful} implies~\ref{itm:ssa-faithful}.
\end{proof}

Motivated by Proposition~\ref{prop:ssa_is_matched_pair} we introduce a generalisation of the faithful self-similar actions of \cite{ABRW} (this is related to the definition in \cite{LiYang}).
\begin{dfn}
	\label{dfn:ssa_not_faithful} A \emph{self-similar action of a groupoid on a $k$-graph} is a matched pair $(\Gg, \Lambda)$ in which $\Gg$ is a groupoid, $\Lambda$ is a $k$-graph, and $d(g \la \mu) = d(\mu)$ for all $(g, \mu) \in \Gg * \Lambda$.
\end{dfn}

\begin{example}\label{ex:ssa on graph}
Let $E = (E^0, E^1, r, s)$ be a directed graph as in Section~\ref{sec:prelims}. Then $E^*$ is a $1$-graph with degree map given by the length functor. Moreover, every $1$-graph is of this form. The definition of a faithful self-similar action of $\Gg$ on $E^*$ as above reduces to the definition of a self-similar action of a groupoid on a graph in \cite{LRRW18}. This in turn generalises the self-similar groups of automorphisms of trees discussed in, for example, \cite{Nek05} (these correspond to the case where $E$ has just one vertex). The definitions in \cite{EP17} and \cite{Yusnitha}, which do not impose a faithfulness condition, are also instances of Definition~\ref{dfn:ssa_not_faithful} with $k=1$.
\end{example}

\subsubsection{Graphs of groups and group actions on trees}

An \emph{undirected graph} $\Gamma = (\Gamma^0, \Gamma^1,r,s,\overline{\,\cdot\,})$ is a directed graph endowed with a map $\overline{\,\cdot\,} \colon \Gamma^1 \to \Gamma^1$ such that $\ol{\ol{e}} = e \ne \ol{e}$ and $r(\ol{e}) = s(e)$ for all $e$.

\begin{dfn}
	A \emph{graph of groups} is a pair $(\Gamma,G)$ consisting of: an undirected graph $\Gamma$; assignments $v \mapsto G_v$ and $e \mapsto G_e$ of a group to each $v \in \Gamma^0$ and $e \in \Gamma^1$, such that
	 $G_e = G_{\ol{e}}$ for all $e \in \Gamma^1$; and injective homomorphisms $\alpha_e \colon G_e \to G_{r(e)}$ for each $e \in \Gamma^1$.
\end{dfn}

The Bass--Serre Theorem~\cite{Bas93,Ser80} describes a duality between graphs of groups and edge-reversal-free actions of groups on trees.

Building on the observations of \cite[Theorem~5.4]{MR21}, we show that every graph of groups $(\Gamma, G)$ gives rise to a matched pair.
For each $e \in \Gamma^1$, let $\Sigma_e$ be a complete set of coset representatives for $G_{r(e)}/\alpha_e(G_e)$. We assume that $1_{G_{r(e)}} \in \Sigma_e$ so $1_{G_{r(e)}}$ is the representative of the coset $\alpha_e(G_e)$. There is a natural action of $G_{r(e)}$ on $\Sigma_e$: we define $g \cdot \mu$ to be the coset representative of $g \mu$ for all $g \in G_{r(e)}$ and $\mu \in \Sigma_e$.

Consider the groupoid $\Gg = \bigsqcup_{e \in \Gamma^1} G_e$, a bundle of groups over $\Gamma^1$. Define a directed graph $E$ by $E^0 \coloneqq \Gamma^1$,
\[
E^1 \coloneqq \{e \mu f \mid \mu\in \Sigma_{e},\,  ef \in \Gamma^2,\, e = \ol{f} \implies \mu \ne 1_{G_{r(e)}}\},
\]
$r(e \mu f) = e$, and $s(e \mu f) = f$. We identify each $E^n$ with
\[
\{e_0 \mu_1 e_1 \mu_2 e_2 \cdots \mu_n e_n \mid \mu_i \in \Sigma_{e_i} ,\, e_0 \ldots e_n \in \Gamma^{n+1},\, e_{i-1}= \ol{e_{i}} \implies \mu_i \ne 1_{r(e_i)}\}.
\]
Consider the path category $E^*$ of $E$. We show that $(\Gg,E^*)$ can be made into matched pair (indeed, a self-similar groupoid action as in Example~\ref{ex:ssa on graph}).

Fix $ef \in \Gamma^2$, $g \in G_{e}$ and $\mu \in \Sigma_{f}$. Then $g \bla \mu \coloneqq \alpha_{\ol{e}}(g) \cdot \mu \in \Sigma_f$ and $g \bra \mu \coloneqq \alpha_{f}^{-1}((g \bla \mu)^{-1} \alpha_{\ol{e}}(g)\mu) \in G_f$ are the unique elements such that $\alpha_{\ol{e}}(g) \mu = (g \bla \mu) \alpha_{f}(g \bra \mu)$.

As in the proof of Proposition~\ref{prop:factorisation_system_is_matched_pair},  for $g_1,g_2 \in G_{e}$, we have $(g_1g_2) \bla \mu = g_1 \cdot (g_2 \bla \mu)$ and $(g_1g_2) \bra \mu = (g_1 \bra (g_2 \bla \mu))(g_2 \bra \mu)$. We define a left action $\la \colon \Gg * E^* \to E^*$ inductively by
\[
g \la e_0 \mu_1 e_1 \mu_2 e_2 \cdots \mu_n e_n = (e_0 (g \bla \mu_1) e_1) ((g \bra \mu_1) \la e_1 \mu_2 e_2 \cdots \mu_n e_n),
\]
and a right action $\ra \colon \Gg * E^* \to \Gg$ inductively by
\[
g \ra e_0 \mu_1 e_1 \mu_2 e_2 \cdots \mu_n e_n = (g \bra \mu_1) \ra e_1 \mu_2 e_2 \cdots \mu_n e_n.
\]
It is straightforward to verify that these actions turn $(\Gg,E^*)$ into a matched pair. Moreover, $\la,\ra$ satisfy \ref{itm:ssa-length}--\ref{itm:ssa-ss} so $\Gg$ acts self-similarly on $E^*$.

\section{Three homology theories for matched pairs}
\label{sec:homology}

We describe three homology theories associated to a matched pair (the last two via a double complex). We show in Section~\ref{sec:main_theorem} that they all coincide up to natural isomorphism.

\subsection{The categorical complex and categorical homology}
\label{subsec:categorical_homology}

\begin{dfn}\label{dfn:categorical_homology}
	Let $\Cc$ be a small category. For each $k \ge 0$ let $C_k (\Cc) \coloneqq \ZZ \Cc^k$ be the free abelian group generated by composable $k$-tuples. We write $[c_0,\ldots,c_k] \in C_{k+1}(\Cc)$ for the generator corresponding to $(c_0,\ldots,c_k) \in
	\Cc^{k+1}$.  For $k \ge 1$ define $d_k \colon C_{k+1}(\Cc) \to C_k(\Cc)$ by
	\[
	d_k[c_0,\ldots,c_k] = [c_1,\ldots,c_k] + \Big(\sum_{i=1}^{k}(-1)^i [c_0,\ldots,c_{i-1}c_{i},\ldots,c_k]\Big) + (-1)^{k+1}[c_0,\ldots,c_{k-1}]
	\]
	and define $d_0 \colon C_1(\Cc) \to C_0(\Cc)$ by $d_0[c] = [s(c)] - [r(c)]$.
	Then $(C_{\bullet}(\Cc),d_{\bullet})$ is a chain complex, and its homology, $H_{\bullet} (\Cc)$, is called the \emph{categorical homology of $\Cc$}.
	
	For an abelian group $A$, let $C^k( \Cc;A)  \coloneqq \Hom(C_k(\Cc),A)$. Define $d^k \colon C^k( \Cc;A) \to C^{k+1}( \Cc;A) $ by $d^k(f) = f \circ d_k$. Then $(C^{\bullet}( \Cc;A) ,d^\bullet)$ is a cochain complex, and its cohomology, $H^{\bullet}( \Cc;A)$, is called the
	\emph{categorical cohomology of $\Cc$ with coefficients in $A$}.
\end{dfn}

There are more-sophisticated definitions of categorical cohomology in terms of projective resolutions of $\Cc$-modules (cf. \cite{GK18}). Our definition amounts to fixing a resolution, analogous to the \emph{bar resolution} for group homology (cf. \cite[\S 6.5]{Wei94}), of a constant functor $\Cc \to A$ (see \cite[Proposition~2.4]{GK18})

\begin{dfn}
The \emph{categorical homology}, denoted $H^{\bowtie}_{\bullet}(\Cc,\Dd)$, of a matched pair $(\Cc,\Dd)$ is the categorical homology of the Zappa--Sz\'ep product category $\Cc \bowtie \Dd$. For each $k \ge 0$, we define $H_{k}^{\bowtie} \colon \MP \to \Ab$ as the functor defined by the composition $(\Cc,\Dd) \mapsto \Cc \bowtie \Dd \mapsto H_k(\Cc \bowtie \Dd)$.
	
For an abelian group $A$, the \emph{categorical cohomology of $(\Cc,\Dd)$ with coefficients in $A$}, denoted $H_{\bowtie}^{\bullet}(\Cc,\Dd; A)$, is the categorical cohomology of $\Cc \bowtie \Dd$ with coefficients in $A$.
\end{dfn}

We work with simplicial groups rather than chain complexes (see~\cite[Ch.8]{Wei94}) to simplify calculations. The Dold--Kan Theorem \cite[Theorem 8.4.1]{Wei94} gives an equivalence of categories between simplicial abelian groups and chain complexes of abelian groups.

For each $k \ge 1$ and $0 \le i \le k+1$ we define the \emph{face map} $\partial_k^i \colon C_{k+1}(\Cc) \to C_{k}(\Cc)$ by
\begin{equation}\label{eq:face_map}
	\partial_k^i [c_0,\ldots,c_{i},\ldots,c_k]
	=\begin{cases*}
		[c_1,\ldots,c_k] & if $i = 0$\\
		[c_0,\ldots,c_{i-1}c_i,\ldots, c_k] & if $1 \le i \le k$\\
		[c_0,\ldots,c_{k-1}] & if $i=k+1$.
	\end{cases*}
\end{equation}
We also define $\partial_0^0 [c] = [s(c)]$ and $\partial_0^1[c] = [r(c)]$. In particular, $d_k = \sum_{i=0}^{k+1} (-1)^k \partial_k^i$.

To work with degeneracy maps, we use the following---slightly non-standard---notation.
\begin{ntn}
	If $(c_1,\ldots,c_k) \in \Cc^k$ is a composable $k$-tuple, and $0 \le i \le k$, then we define
	\begin{align*}
		(c_1,\ldots,c_{i-1}, \_, c_i, \ldots, c_k)
		&\coloneqq (c_1,\ldots,c_{i-1}, s(c_{i-1}), c_i, \ldots, c_k)\\
		&= (c_1,\ldots,c_{i-1}, r(c_i), c_i, \ldots, c_k) \in \Cc^{k+1}.
	\end{align*}
The identity morphism represented by any given instance of $\_$ is determined by either of the neighbouring entries.
\end{ntn}
For each $k \ge 1$ and $0 \le i \le k$ we define the \emph{degeneracy map} $\sigma_k^i \colon C_k(\Cc) \to C_{k+1} (\Cc)$ by
\[
\sigma_k^i [c_0,\ldots,c_{k-1}] = \begin{cases*}
	[\_, c_0,\ldots, c_{k-1}] & if $i = 0$,\\
	[c_0,\ldots, c_{i-1},\_,c_i,\ldots,c_{k-1}] & if $0 < i <k$,\\
	[c_0,\ldots,c_{k-1},\_] & if $i = k$,
\end{cases*}
\]
with $\sigma_0^0[x] = [x]$ for $x \in \Cc^0$. These and the $\partial_j^i$ satisfy the \emph{simplicial identities}:
\begin{align*}
	\partial_{k-1}^i \partial_k^j &= \partial^{j-1}_{k-1} \partial_k^i \quad \text{if } i < j,\\
	\sigma^i_{k+1} \sigma^j_k &= \sigma^{j+1}_{k+1} \sigma^i_k \quad \text{if } i \le j, \text{ and}\\
	\partial^i_k \sigma^j_k	&= \begin{cases}
		\sigma^{j-1}_k \partial_k^i & \text{if } i < j\\
		\id_{C_k(\Cc)} & \text{if } i = j \text{ or } i = j+1\\
		\sigma^j_k \partial^{i-1}_k & \text{if } i > j+1,
	\end{cases}
\end{align*}
so $(C_{\bullet}(\Cc), \partial, \sigma)$ is a simplicial abelian group.

If $(\Cc,\Dd)$ is a matched pair, then Proposition~\ref{prop:bowtie_**} gives an action of $\Cc$ on each $\Dd^{k+1}$. For $(c,d) \in \Cc * \Dd^{k+1}$, we write $c \la [d]$ for the generator $[c \la d]$ of $C_{k+1}(\Dd)$. Similarly if $(c,d) \in \Cc^{k+1} * \Dd$ we write $[c] \ra d$ for the generator  $[c \ra d]$.

\begin{lem}\label{lem:actions faces commute}
	Let $(\Cc, \Dd)$ be a matched pair and take $k  \in \NN$. For $0 < i \le k +1$  and for
	$(c, d) \in \Cc * \Dd^{k+1}$, we have $\partial^i_k(c \la [d]) = c \la
	\partial^i_k[d]$ in $C_{k}(\Dd)$. Similarly, for $0 \le j < k+1 \in \NN$ and $(c,d) \in \Cc^{k+1} *
	\Dd$, we have $\partial^j_k([c]\ra d) = \partial^j_k[c]\ra d$ in $C_k(\Cc)$.
\end{lem}
\begin{proof}
	We prove the first statement; the second is proved symmetrically. Since $i > 0$, we have
	\begin{align*}
		c \la \partial_k^i[d]
		&= c\la[d_0, \dots, d_id_{i+1},\dots, d_k]\\
		&= \big[c\la d_0, (c\ra d_0)\la d_1, \dots, (c\ra d_0\dots d_{i-1})\la d_id_{i+1},\dots,(c\ra (d_0\dots d_{k-1}))\la d_k\big]\\
		&= \big[c\la d_0, (c\ra d_0)\la d_1, \dots, \big((c\ra d_0\dots d_{i-1})\la d_i\big)\big((c\ra d_0\dots d_{i})\la d_{i+1}\big),\\
		&\hskip12em \dots,(c\ra (d_0\dots d_{k-1}))\la d_k\big]\\
		&= \partial_k^i\big[c\la d_0, (c\ra d_0)\la d_1, \dots, \big((c\ra d_0\dots d_{i-1})\la d_i), \big((c\ra d_0\dots d_{i})\la d_{i+1}\big),\\
		&\hskip12em \dots,(c\ra (d_0\dots d_{k-1}))\la d_k\big]\\
		&= \partial_k^i(c \la [d]).\qedhere
	\end{align*}
\end{proof}

\begin{rmk}
Lemma~\ref{lem:actions faces commute} is \emph{only} valid for $i > 0$ and $j < k+1$. The left action of $\Cc$ on $\Dd^{k+1}$ does not commute with $\partial^0_k$, and the right action of $\Dd$ on $\Cc^{k+1}$ does not commute with $\partial^{k+1}_{k}$.
\end{rmk}

\subsection{The matched complex}

We associate a double complex to each matched pair $(\Cc,\Dd)$. For $p,\,q \ge 0$, regard elements of $\Cc^p * \Dd^q$ as composable tuples in $(\Cc \bowtie \Dd)^{p+q}$, whose first $p$ terms belong to $\Cc \subseteq \Cc \bowtie \Dd$ and whose remaining $q$ terms belong to $\Dd \subseteq \Cc \bowtie \Dd$.

Let $C_{p,q} (\Cc,\Dd) \coloneqq \ZZ (\Cc^p * \Dd^q)$, the free abelian group generated by $\Cc^p * \Dd^q$. Let $\la$ be the action of $\Cc$ on $\Dd^q$ of Lemma~\ref{lem:bowtie_to_bowtie*}. Define \emph{horizontal face maps} $\partial_{p,q}^{h,i} \colon C_{p+1,q}(\Cc,\Dd) \to C_{p,q}(\Cc,\Dd)$ as follows. For $q \ge 1$,
\[
\partial_{p,q}^{h,i} [c_0,\ldots,c_{p},d_{0},\ldots,d_{q-1}]
= \begin{cases}
	[c_1,\ldots, c_p, d_0,\ldots, d_{q-1}] & \text{if } i = 0\\
	[c_0,\ldots, c_{i-1}c_i,\ldots, c_p, d_0,\ldots, d_{q-1}] & \text{if } 1 \le i \le p\\
	[c_0,\ldots,c_{p-1}, c_p \la (d_0,\ldots,d_{q-1}) ]& \text{if } i = p+1,
\end{cases}
\]
while $\partial_{p,0}^{h,i} \coloneqq \partial_p^i \colon C_{p+1}(\Cc) \to C_{p}(\Cc) $ as in \eqref{eq:face_map}. For $0 \le i \le p$ we define the
\emph{horizontal degeneracy maps}
$\sigma_{p,q}^{h,i} \colon C_{p,q}(\Cc,\Dd) \to C_{p+1,q}(\Cc,\Dd)$ by
\[
\sigma^{h,i}_{p,q}[c_0,\ldots,c_{p-1},d_0,\ldots,d_{q-1}] = [c_0, \ldots,c_i,\_,c_{i+1}, \ldots, c_{p-1}, d_0,\ldots,d_{q-1}].
\]
For each $q \ge 0$, the tuple $(C_{\bullet,q}(\Cc,\Dd), \partial^h_{\bullet,q}, \sigma^h_{\bullet,q})$ is a simplicial abelian group.

Let $\ra$ be the action of $\Dd$ on $\Cc^*$ of Lemma~\ref{lem:bowtie_to_bowtie*}. Define \emph{vertical face maps} $\partial_{p,q+1}^{v,j} \colon C_{p,q+1}(\Cc,\Dd) \to C_{p,q}(\Cc,\Dd)$ as follows. For $p > 0$,
\begin{equation}
	\label{eq:vertical_face_maps}
	\partial_{p,q}^{v,j} [c_0,\ldots,c_{p-1},d_{0},\ldots,d_{q}]
	= (-1)^p \begin{cases}
		[(c_0,\ldots,c_{p-1}) \ra d_0, d_1,\ldots,d_q]  & \text{if } j = 0\\
		[c_0,\ldots, c_{p-1}, d_0,\ldots, d_{j-1}d_j, \ldots, d_{q}] & \text{if } 1 \le j \le q\\
		[c_0,\ldots,c_{p-1}, d_0,\ldots,d_{q-1}] & \text{if } j = q+1,
	\end{cases}
\end{equation}
while $\partial_{0,q}^{v,i} \coloneqq \partial_q^i \colon C_{q+1} (\Dd) \to C_{q}(\Dd)$ as in \eqref{eq:face_map}. For $0 \le j \le q$ we define \emph{vertical degeneracy maps} $\sigma_{p,q}^{v,j} \colon C_{p,q} \to C_{p,q+1}$ by
\[
\sigma^{v,j}_{p,q}[c_0,\ldots,c_{p-1},d_0,\ldots,d_{q-1}] = (-1)^p [c_0, \ldots , c_{p-1}, d_0,\ldots, d_{j},\_,d_{j+1},d_{q-1}].
\]
Then $(C_{p,\bullet}(\Cc,\Dd), \partial^v_{p,\bullet}, \sigma^v_{p,\bullet})$ is also a simplicial abelian group.

For the next result, recall from \cite[\S8.5]{Wei94} that a bisimplicial abelian group is a quintuple $(C_{\bullet,\bullet}, \partial^h, \partial^v, \sigma^h, \sigma^v)$ consisting of abelian groups $C_{p,q}$ and homomorphisms $\partial^{h,i}_{p,q}\colon C_{p+1,q} \to C_{p,q}$, $\sigma^{h,i}_{p,q}\colon C_{p,q} \to C_{p+1,q}$, $\partial^{v,j}_{p,q}\colon C_{p,q+1} \to C_{p,q}$, and $\sigma^{v,j}_{p,q}\colon C_{p,q} \to C_{p,q+1}$ such that each $(C_{p, \bullet}, \partial^v_{p, \bullet}, \sigma^v_{p,\bullet})$ and each  $(C_{\bullet, q}, \partial^h_{\bullet, q}, \sigma^h_{\bullet, q})$ is a simplicial group, and
\[
\partial_{p,q}^{v,i}\partial_{p,q+1}^{h,j} = -\partial_{p,q}^{h,j} \partial_{p+1,q}^{v,i}
\quad\text{and}\quad
\sigma_{p+1,q}^{v,i}\sigma_{p,q}^{h,j} = -\sigma_{p,q+1}^{h,j} \sigma_{p,q}^{v,i}.
\]

\begin{prop}
	The quintuple $(C_{\bullet,\bullet}, \partial^h, \partial^v, \sigma^h, \sigma^v)$ is a bisimplicial group. Define
	\[
	d_{p,q}^h = \sum_{i=0}^{p+1} (-1)^i \partial_{p,q}^{h,i} \colon C_{p+1,q} \to C_{p,q} \quad \text{and} \quad 	d_{p,q}^v = \sum_{i=0}^{q+1} (-1)^i \partial_{p,q}^{v,i} \colon C_{p,q+1} \to C_{p,q}.
	\]
	Then
	\begin{equation}\label{eq:double_complex}
		%
		\begin{tikzcd}[column sep=17pt, row sep=17pt]
			{} \arrow[dotted,d] & {} \arrow[dotted,d] & {} \arrow[dotted,d] & \\
			{C_{0,2}} \arrow[d,"d^v_{0,1}"] & {C_{1,2}}  \arrow[d,"d^v_{1,1}"] \arrow[l,"d^h_{0,2}"]& {C_{2,2}} \arrow[d,"d^v_{2,1}"]  \arrow[l,"d^h_{1,2}"]& {} \arrow[dotted,l]\\
			{C_{0,1}} \arrow[d,"d^v_{0,0}"] & {C_{1,1}} \arrow[d,"d^v_{1,0}"]  \arrow[l,"d^h_{0,1}"]& {C_{2,1}} \arrow[d,"d^v_{2,0}"] \arrow[l,"d^h_{1,1}"]& {} \arrow[dotted,l] \\
			{C_{0,0}}  & {C_{1,0}} \arrow[l,"d^h_{0,0}"] & {C_{2,0}} \arrow[l,"d^h_{1,0}"] & {} \arrow[dotted,l]
		\end{tikzcd}
	\end{equation}
	is a first-quadrant double chain complex satisfying $d^hd^v = - d^vd^h$.
\end{prop}
\begin{proof}
	Fix $p, q \ge 0$ and fix $i \le p+1$ and $j \le q+1$. We must show that $\partial_{p,q}^{v,i}\partial_{p,q+1}^{h,j} = - \partial_{p,q}^{h,j} \partial_{p+1,q}^{v,i}$. Fix $[c_0, \dots, c_p; d_0, \dots, d_q] \in C_{p+1, q+1}$. If $i \not= p+1$ or $j \not= 0$, then $\partial_{p,q}^{v,i}$ and $\partial_{p,q+1}^{h,j}$ compose or delete nonadjacent coordinates, as do $\partial_{p,q}^{h,j}$ and $\partial_{p+1,q}^{v,i}$, and so the factors of $(-1)^p$ and $(-1)^{p+1}$ in $\partial_{p,q}^{v,i}$ and $\partial_{p+1,q}^{v,i}$ give the desired anticommutation relation. If $i = p+1$ and $j = 0$, then
	\begin{align*}
		\partial_{p,q}^{v,j}\partial_{p,q+1}^{h,i}([c_0, \dots, c_p; d_0, \dots, d_q])
		&= \partial_{p,q}^{v,0}([c_0, \dots, c_{p-1}; c_p\la (d_0, \dots, d_q)])\\
		&= \partial_{p,q}^{v,0}([c_0, \dots, c_{p-1}; c_p \la d_0, (c_p \ra d_0) \la (d_1, \dots, d_q)])\\
		&= (-1)^p [(c_0, \dots, c_{p-1})\ra (c_p \la d_0); (c_p \ra d_0) \la (d_1, \dots, d_q)],
	\end{align*}
	and
	\begin{align*}
		\partial_{p,q}^{h,i}\partial_{p+1,q}^{v,j}([c_0, \dots, c_p; d_0, \dots, d_q])
		&= (-1)^{p+1} \partial_{p,q}^{h,p+1}([(c_0, \dots, c_p)\ra d_0; d_1, \dots, d_q])\\
		&= (-1)^{p+1} \partial_{p,q}^{h,p+1}([(c_0, \dots, c_{p-1})\ra (c_p \la d_0), c_p \ra d_0; d_1, \dots, d_q])\\
		&= (-1)^{p+1} [(c_0, \dots, c_{p-1})\ra (c_p \la d_0); (c_p \ra d_0) \la (d_1, \dots, d_q)],
	\end{align*}
	which gives the desired relation. It follows that $d^vd^h = -d^hd^v$.
	
	The anticommutation relation $\sigma_{p+1,q}^{v,i}  \sigma_{p,q}^{h,j} = - \sigma_{p,q+1}^{h,j}  \sigma_{p,q}^{v,i}$ also
	follows from direct computation. Routine calculation shows that~\eqref{eq:double_complex} is a first-quadrant double chain complex.
\end{proof}

\begin{dfn}
	We call the double chain complex $(C_{\bullet,\bullet}(\Cc,\Dd),d^h,d^v)$ the \emph{matched complex} of the matched pair $(\Cc,\Dd)$.
\end{dfn}

\begin{lem}
	The assignment $C_{\bullet,\bullet}$ of a matched complex to each matched pair is a functor from the category $\MP$ of matched pairs to the category of double complexes of abelian groups.
\end{lem}
\begin{proof}
	Matched-pair morphisms intertwine the face and degeneracy maps $\partial_{p,q}^{\bullet,i}$ and $\sigma_{p,q}^{\bullet,i}$.
\end{proof}

\begin{ntn}
For the remainder of the paper we frequently omit the subscripts on face maps, degeneracy maps and boundary maps. For example, $\partial^{h, i}$ denotes any of the maps $\partial^{h, i}_{p,q}$; the values of $p$ and $q$ should be clear from context.
\end{ntn}

There are two chain complexes associated to each double complex: the \emph{diagonal complex} and the \emph{total complex} \cite[\S 8.5]{Wei94}.

\subsection{The diagonal complex and diagonal homology}
\label{subsec:diagonal_complex}

Let $(C_{\bullet,\bullet}(\Cc,\Dd),d^h,d^v)$ be the matched complex of a matched pair $(\Cc,\Dd)$. 	For each $k \ge 0$, let
\[
C_{k}^{\Delta}(\Cc,\Dd) \coloneqq C_{k,k} (\Cc,\Dd)
\]
and define $\partial^{\Delta,i}_{k} \colon C_{k+1}^{\Delta}(\Cc,\Dd) \to C_k^{\Delta} (\Cc,\Dd)$ and $\sigma^{\Delta,i}_k \colon C_{k}^{\Delta}(\Cc,\Dd) \to C_{k+1}^{\Delta}
(\Cc,\Dd)$ by
\[
\partial^{\Delta, i}_k \coloneqq \partial^{h,i}_{k,k} \partial^{v,i}_{k+1,k}
\quad \text{and} \quad
\sigma^{\Delta,i}_{k} \coloneqq \sigma^{v,i}_{k+1,k} \sigma^{h,i}_{k,k}.
\]
Then $(C_{\bullet}^{\Delta}(\Cc,\Dd), \partial^{\Delta}, \sigma^{\Delta})$ is a simplicial group \cite[\S 8.5]{Wei94}. Let $d^{\Delta}_k \coloneqq \sum_{i=0}^{k+1} (-1)^{i}
\partial^{\Delta,i}_k$.

\begin{dfn}
	The \emph{diagonal complex} of $(\Cc,\Dd)$ is the chain complex $(C_{\bullet}^{\Delta}(\Cc,\Dd),d^{\Delta}_{\bullet})$. We denote the homology of this chain complex by $H^{\Delta}_{\bullet}(\Cc,\Dd)$.
\end{dfn}

\subsection{The total complex and total homology}
\label{subsec:total_complex}

Let $(C_{\bullet,\bullet}(\Cc,\Dd),d^h,d^v)$ be the matched complex of a matched pair $(\Cc,\Dd)$. 	For each $k \ge 0$, let
\[
C_{k}^{\Tot} (\Cc,\Dd) \coloneqq \bigoplus_{p+q = k} C_{p,q}(\Cc,\Dd).
\]
Define $d^{\Tot}_k \colon C_{k+1}^{\Tot}(\Cc,\Dd) \to C_{k}^{\Tot} (\Cc,\Dd)$ by $d^{\Tot}_k \coloneqq \sum_{p+q = k} d_{p,q}^{h} + d_{p,q}^v$.

\begin{dfn}
	The \emph{total complex} of $(\Cc,\Dd)$ is the chain complex $(C_{\bullet}^{\Tot}(\Cc,\Dd),d^{\Tot}_{\bullet})$. We denote the homology of this complex by $H^{\Tot}_{\bullet}(\Cc,\Dd)$.
\end{dfn}

\section{Equivalence of homology theories}
\label{sec:main_theorem}

In this section we prove that the homology theories for matched pairs introduced in Subsections~\ref{subsec:categorical_homology},~\ref{subsec:diagonal_complex},~and~\ref{subsec:total_complex} coincide. Specifically, we describe natural chain maps that induce isomorphisms between them and between the dual cohomology theories. We also give formulae for their inverses. The main result is Theorem~\ref{thm:cohomologies_are_the_same}. We start by defining the maps involved.

\subsection{The natural chain maps}
\label{subsec:the_chain_maps}

We begin by describing explicit formulae for natural chain maps $\nabla \colon C_{\bullet}^{\Tot} \to C_{\bullet}^{\Delta}$, $\Pi \colon C_{\bullet}^{\Delta} \to C_{\bullet}^{\bowtie}$, and $\Psi \colon C_{\bullet}^{\bowtie} \to C_{\bullet}^{\Tot}$.

\subsubsection{The map $\nabla$}

The map $\nabla \colon C_{\bullet}^{\Tot} \to C_{\bullet}^{\Delta}$ is the \emph{Eilenberg--Zilber map} \cite[\S~8.5.4]{Wei94}.
For $p,q \in \NN$ a \emph{$(p,q)$-shuffle} is a permutation $\beta$ of $\{1,\ldots,p+q\}$ such that
\[
\beta(1) < \beta(2) < \cdots < \beta(p) \quad \text{ and } \quad \beta(p+1) < \beta(p+2) < \cdots < \beta (p+q).
\]
We write $\Sh(p,q)$ for the collection of all $(p,q)$-shuffles,
and $\sgn(\beta) \in \{-1,1\}$ for the sign of a permutation $\beta$.
The $(p,q)$-component $\nabla_{p,q} \colon C_{p,q} \to C_{p+q}^{\Delta}$ of the Eilenberg--Zilber map is
\begin{equation} \label{eq:eilenberg-zilber}
	\nabla_{p,q} = \sum_{\beta \in \Sh(p,q)} \sgn(\beta) \,\sigma^{h,\beta(p+q)}_{p+q-1,p+q} \circ \cdots \circ \sigma^{h,\beta(p+1)}_{p,p+q} \circ \sigma^{v,\beta(p)}_{p,p+q-1} \circ \cdots \circ \sigma^{v,\beta(1)}_{p,q}.
\end{equation}

\subsubsection{The map $\Pi$} We describe $\Pi \colon C_{\bullet}^{\Delta} \to C_{\bullet}^{\bowtie}$. For $k \ge 1$ define $\bowtie^k \colon (\Cc * \Dd)^k \to (\Dd * \Cc)^k$ by
\[
\bowtie^k(c_1,d_1,\ldots,c_k,d_k) \coloneqq (c_1 \bowtie d_1,\ldots, c_k \bowtie d_k).
\]
Set  $\Pi_0 = \id_{\Cc^{0}}$ and inductively define $\Pi_k \colon \Cc^k * \Dd^k \to (\Cc \bowtie \Dd)^k = (\Dd * \Cc)^k$ for $k \ge 1$ by
\begin{equation}\label{eq:pi}
	\Pi_k \coloneqq \bowtiemap^k \circ (1_\Cc * \Pi_{k-1} * 1_\Dd).
\end{equation}
These extend to homomorphisms $\Pi_k \colon C^{\Delta}_k (\Cc,\Dd) \to C^{\bowtie}_k(\Cc,\Dd)$. For example,
\begin{align*}
	\Pi_1[c_1,d_1] &= [c_1 \la d_1, c_1 \ra d_1],\\
	\Pi_2[c_1,c_2,d_1,d_2] &=
	[c_1c_2 \la d_1, c_1 \ra (c_2 \la d_1), (c_2 \ra d_1) \la d_2, c_2\ra d_1d_2], \text{ and}\\
	\Pi_3[c_1,c_2,c_3,d_1,d_2,d_3] &= [c_1c_2c_3 \la d_1, c_1 \ra (c_2c_3 \la d_1), (c_2c_3 \ra d_1 )\la d_2,\\
	&\qquad \qquad c_2 \ra (c_3 \la d_1d_2), (c_3 \ra d_1d_2) \la d_3, c_3 \ra d_1d_2d_3].
\end{align*}
An induction on $k$, using that matched-pair morphisms respect left and right actions, shows that the $\Pi_k$ extend to natural transformations $\Pi_k
\colon C_k^{\Delta} \to C_k^{\bowtie}$.

\begin{rmk}\label{rmk:diagram_pi_k}
	The map $\Pi_k$ can be described diagrammatically. We represent elements of $\Cc$ by blue vertices, and elements of $\Dd$ by red vertices; vertical lines are identity morphisms; and crossings are applications of $\bowtie$:
	
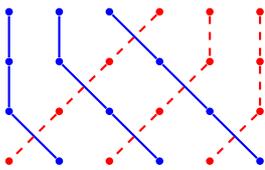
	\begin{minipage}{0.35\textwidth}
		\def\wid{0.55}
		\def\hei{0.55}
		\begin{center}
			\vspace{4pt}
			
			\begin{tikzpicture}
				[vertex/.style={circle, fill=black, inner sep=1pt},
				edge/.style={thick,blue},
				edgev/.style={thick,dashed,red},
				scale = 1.2]
				
				\node[vertex,blue] (0-0) at (0*\wid,0) {};
				\node[vertex,blue] (1-0) at (1*\wid,0) {};
				\node[vertex,blue] (2-0) at (2*\wid,0) {};
				\node[vertex,red] (3-0) at (3*\wid,0) {};
				\node[vertex,red] (4-0) at (4*\wid,0) {};
				\node[vertex,red] (5-0) at (5*\wid,0) {};
				
				\node[vertex,blue] (0-1) at (0*\wid,-1*\hei) {};
				\node[vertex,blue] (1-1) at (1*\wid,-1*\hei) {};
				\node[vertex,red] (2-1) at (2*\wid,-1*\hei) {};
				\node[vertex,blue] (3-1) at (3*\wid,-1*\hei) {};
				\node[vertex,red] (4-1) at (4*\wid,-1*\hei) {};
				\node[vertex,red] (5-1) at (5*\wid,-1*\hei) {};
				
				\node[vertex,blue] (0-2) at (0*\wid,-2*\hei) {};
				\node[vertex,red] (1-2) at (1*\wid,-2*\hei) {};
				\node[vertex,blue] (2-2) at (2*\wid,-2*\hei) {};
				\node[vertex,red] (3-2) at (3*\wid,-2*\hei) {};
				\node[vertex,blue] (4-2) at (4*\wid,-2*\hei) {};
				\node[vertex,red] (5-2) at (5*\wid,-2*\hei) {};
				
				\node[vertex,red] (0-3) at (0*\wid,-3*\hei) {};
				\node[vertex,blue] (1-3) at (1*\wid,-3*\hei) {};
				\node[vertex,red] (2-3) at (2*\wid,-3*\hei) {};
				\node[vertex,blue] (3-3) at (3*\wid,-3*\hei) {};
				\node[vertex,red] (4-3) at (4*\wid,-3*\hei) {};
				\node[vertex,blue] (5-3) at (5*\wid,-3*\hei) {};

				\draw[edge] (0-0) -- (0-1) -- (0-2) -- (1-3);
				\draw[edge] (1-0) -- (1-1) -- (2-2) -- (3-3);
				\draw[edge] (2-0) -- (3-1) -- (4-2) -- (5-3);
				
				\draw[edgev] (3-0) -- (2-1) -- (1-2) -- (0-3);
				\draw[edgev] (4-0) -- (4-1) -- (3-2) -- (2-3);
				\draw[edgev] (5-0) -- (5-1) -- (5-2) -- (4-3);
			\end{tikzpicture}
		\end{center}
	\end{minipage}
	\begin{minipage}{0.6\textwidth}
		\begin{align*}
			&[{\cb c_1},{ \cb c_2},{\cb c_3},{\cre d_1},{\cre d_2},{\cre d_3}]\\
			&[{\cb c_1},{\cb c_2}, {\cre c_3 \la d_1},{\cb c_3 \ra d_1}, {\cre d_2} ,{\cre d_3}]\\
			&[{\cb c_1}, {\cre c_2c_3 \la d_1}, {\cb c_2 \ra (c_3 \la d_1)}, {\cre (c_3 \ra d_1) \la d_2}, {\cb c_3 \ra d_1d_2}, {\cre d_3}]\\
			&\Pi_3[c_1,c_2,c_3,d_1,d_2,d_3].
		\end{align*}
	\end{minipage}
	
	So starting with an element of $\Cc^k * \Dd^k$, we apply $\bowtie$ to pairs of adjacent terms wherever possible until we obtain an element of $(\Dd * \Cc)^k$.
\end{rmk}

\subsubsection{The map $\Psi$}
\label{sec:Psi}
We now define $\Psi \colon C_{\bullet}^{\bowtie} \to C_{\bullet}^{\Tot}$.
For $q \ge 0$ define
$
\tau^{q} \colon (\Dd\bowtie\Cc)^q \to \Dd^q*\Cc^q
$
as follows: regard $(d_1c_1,\ldots,d_qc_q) \in (\Dd\bowtie\Cc)^q$ as a composable $q$-tuple in $\Cc^* \bowtie \Dd^*$. By Proposition~\ref{prop:bowtie_**} there exist unique
$d' \in \Dd^q$ and $c' \in \Cc^q$ such that $(d_1c_1)\cdots(d_qc_q) = d'c' \in \Cc^p \bowtie \Dd^q$. For instance,
\[
(d_1c_1)(d_2c_2)(d_3c_3) = (d_1,c_1\la d_2, ((c_1 \ra d_2)c_2) \la d_3 ,c_1 \ra (d_2(c_2 \la d_3)), c_2 \ra d_3, c_3) \in \Dd^3 * \Cc^3.
\]
We define $\tau^q (d_1c_1,\ldots,d_qc_q) \coloneqq (d',c')$.
\begin{rmk}\label{rmk:diagram_tau_q}
	We can describe $\tau^q$ via a diagram using the same conventions as in Remark~\ref{rmk:diagram_pi_k}. For example $\tau^3$ is represented by the diagram
	
	\begin{minipage}{0.25\textwidth}
		\def\wid{0.55}
		\def\hei{0.55}
		\begin{center}
			\vspace{4pt}
			
			\begin{tikzpicture}
				[vertex/.style={circle, fill=black, inner sep=1pt},
				edge/.style={thick,blue},
				edgev/.style={thick,dashed,red},
				scale = 1.2]
				
				\node[vertex,red] (0-0) at (0*\wid,0) {};
				\node[vertex,blue] (1-0) at (1*\wid,0) {};
				\node[vertex,red] (2-0) at (2*\wid,0) {};
				\node[vertex,blue] (3-0) at (3*\wid,0) {};
				\node[vertex,red] (4-0) at (4*\wid,0) {};
				\node[vertex,blue] (5-0) at (5*\wid,0) {};
				
				\node[vertex,red] (0-1) at (0*\wid,-1*\hei) {};
				\node[vertex,red] (1-1) at (1*\wid,-1*\hei) {};
				\node[vertex,blue] (2-1) at (2*\wid,-1*\hei) {};
				\node[vertex,red] (3-1) at (3*\wid,-1*\hei) {};
				\node[vertex,blue] (4-1) at (4*\wid,-1*\hei) {};
				\node[vertex,blue] (5-1) at (5*\wid,-1*\hei) {};
				
				\node[vertex,red] (0-2) at (0*\wid,-2*\hei) {};
				\node[vertex,red] (1-2) at (1*\wid,-2*\hei) {};
				\node[vertex,red] (2-2) at (2*\wid,-2*\hei) {};
				\node[vertex,blue] (3-2) at (3*\wid,-2*\hei) {};
				\node[vertex,blue] (4-2) at (4*\wid,-2*\hei) {};
				\node[vertex,blue] (5-2) at (5*\wid,-2*\hei) {};
				
				\draw[edgev] (0-0) -- (0-1) -- (0-2);
				\draw[edge] (1-0) -- (2-1) -- (3-2);
				\draw[edgev] (2-0) -- (1-1) -- (1-2);
				\draw[edge] (3-0) -- (4-1) -- (4-2);
				\draw[edgev] (4-0) -- (3-1) -- (2-2);			
				\draw[edge] (5-0) -- (5-1) -- (5-2);
			\end{tikzpicture}
		\end{center}
	\end{minipage}
	\begin{minipage}{0.7\textwidth}
		\begin{align*}
			&({\cre d_1}, {\cb c_1}, {\cre d_2}, {\cb c_2}, {\cre d_3}, {\cb c_3})\\
			&({\cre d_1}, {\cre c_1 \la d_2}, {\cb c_1 \ra d_2}, {\cre c_2 \la d_3}, {\cb c_2 \ra d_3}, {\cb c_3})\\
			&({\cre d_1},{\cre c_1\la d_2}, {\cre ((c_1 \ra d_2)c_2) \la d_3 },{\cb c_1 \ra (d_2(c_2 \la d_3))}, {\cb c_2 \ra d_3}, {\cb c_3}).
		\end{align*}
	\end{minipage}
	
	\noindent The maps $\tau^q$ for $q \ge 3$ can be visualised similarly.
\end{rmk}

For $p,q \ge 0$ let $\rho_{p,q} \colon \Dd^p * \Cc^p * \Dd^q * \Cc^q \to \Cc^p * \Dd^q$ denote the projection onto the middle two factors. Define $\Psi_{p,q} \colon (\Cc
\bowtie \Dd)^{p+q} \to \Cc^p*\Dd^q$ by
\[
\Psi_{p,q} = \rho_{p,q} \circ (\tau_{p} * \tau_{q})
\]
and extend it to a homomorphism $\Psi_{p,q} \colon C_{p+q}^{\bowtie} (\Cc,\Dd) \to C_{p,q}(\Cc,\Dd)$. For example, we can represent $\Psi_{3,2}[d_1, c_1, \ldots, d_5, c_5]$ diagrammatically by
\begin{center}
	\def\wid{0.55}
	\def\hei{0.55}
	\begin{tikzpicture}
		[vertex/.style={circle, fill=black, inner sep=1pt},
		edge/.style={thick,blue},
		edgev/.style={thick,dashed,red},
		cross/.style={cross out, draw,
			minimum size=2*(#1-\pgflinewidth),
			inner sep=2pt, outer sep=0pt,thick},
		scale = 1.2]
		
		\node[vertex,red] (0-0) at (0*\wid,0) {};
		\node[vertex,blue] (1-0) at (1*\wid,0) {};
		\node[vertex,red] (2-0) at (2*\wid,0) {};
		\node[vertex,blue] (3-0) at (3*\wid,0) {};
		\node[vertex,red] (4-0) at (4*\wid,0) {};
		\node[vertex,blue] (5-0) at (5*\wid,0) {};
		\node[vertex,red] (6-0) at (6*\wid,0) {};
		\node[vertex,blue] (7-0) at (7*\wid,0) {};
		\node[vertex,red] (8-0) at (8*\wid,0) {};
		\node[vertex,blue] (9-0) at (9*\wid,0) {};
		
		\node[vertex,red] (0-1) at (0*\wid,-1*\hei) {};
		\node[vertex,red] (1-1) at (1*\wid,-1*\hei) {};
		\node[vertex,blue] (2-1) at (2*\wid,-1*\hei) {};
		\node[vertex,red] (3-1) at (3*\wid,-1*\hei) {};
		\node[vertex,blue] (4-1) at (4*\wid,-1*\hei) {};
		\node[vertex,blue] (5-1) at (5*\wid,-1*\hei) {};
		\node[vertex,red] (6-1) at (6*\wid,-1*\hei) {};
		\node[vertex,red] (7-1) at (7*\wid,-1*\hei) {};
		\node[vertex,blue] (8-1) at (8*\wid,-1*\hei) {};
		\node[vertex,blue] (9-1) at (9*\wid,-1*\hei) {};
		
		\node[cross,red] (0-2) at (0*\wid,-2*\hei) {};
		\node[cross,red] (1-2) at (1*\wid,-2*\hei) {};
		\node[cross,red] (2-2) at (2*\wid,-2*\hei) {};
		\node[vertex,blue] (3-2) at (3*\wid,-2*\hei) {};
		\node[vertex,blue] (4-2) at (4*\wid,-2*\hei) {};
		\node[vertex,blue] (5-2) at (5*\wid,-2*\hei) {};
		\node[vertex,red] (6-2) at (6*\wid,-2*\hei) {};
		\node[vertex,red] (7-2) at (7*\wid,-2*\hei) {};
		\node[cross,blue] (8-2) at (8*\wid,-2*\hei) {};
		\node[cross,blue] (9-2) at (9*\wid,-2*\hei) {};
		
		\node[vertex,blue] (3-3) at (3*\wid,-3*\hei) {};
		\node[vertex,blue] (4-3) at (4*\wid,-3*\hei) {};
		\node[vertex,blue] (5-3) at (5*\wid,-3*\hei) {};
		\node[vertex,red] (6-3) at (6*\wid,-3*\hei) {};
		\node[vertex,red] (7-3) at (7*\wid,-3*\hei) {};
		
		
		\draw[edgev] (0-0) -- (0-1) -- (0-2);
		\draw[edge] (1-0) -- (2-1) -- (3-2) -- (3-3);
		\draw[edgev] (2-0) -- (1-1) -- (1-2);
		\draw[edge] (3-0) -- (4-1) -- (4-2) -- (4-3);
		\draw[edgev] (4-0) -- (3-1) -- (2-2);			
		\draw[edge] (5-0) -- (5-1) -- (5-2) -- (5-3);
		
		\draw[edgev] (6-0) -- (6-1) -- (6-2) -- (6-3);
		\draw[edge] (7-0) -- (8-1) -- (8-2);
		\draw[edgev] (8-0) -- (7-1) -- (7-2) -- (7-3);
		\draw[edge] (9-0) -- (9-1) -- (9-2);
        \node[anchor=west] at (11*\wid, 0) {$[{\cre d_1}, {\cb c_1}, \dots, {\cre d_5}, {\cb c_5}]$};
        \node[anchor=west] at (11*\wid, -2*\hei) {$(\tau^3 * \tau^2)([{\cre d_1}, {\cb c_1}, \dots, {\cre d_5}, {\cb c_5}])$};
        \node[anchor=west] at (11*\wid, -3*\hei) {$\Psi_{3,2}([{\cre d_1}, {\cb c_1}, \dots, {\cre d_5}, {\cb c_5}])$};
	\end{tikzpicture}
\end{center}
(crossed vertices like
$\begin{tikzpicture}
	\node[cross out, draw, inner sep=2pt, outer sep=0pt,thick, red] (a) at (0,0){};
\end{tikzpicture}$
indicate omission of the corresponding entries).

We now define
$\Psi_{k} \colon C_{k}^{\bowtie}(\Cc,\Dd) \to C_{k}^{\Tot}(\Cc,\Dd)$ by
\begin{equation}\label{eq:psi}
	\Psi_k = \sum_{p+q = k} \Psi_{p,q}.
\end{equation}
Explicit formulae for low-degree terms are given by
\begin{align*}
	\Psi_{1} [d_1c_1] &=  \Psi_{1,0}[d_1c_1] + \Psi_{0,1}[d_1c_1] = [c_1] + [d_1] \\
	\Psi_{2} [d_1c_1,d_2c_2] &= [c_1 \ra d_2, c_2] + [c_1;d_2] + [d_1, c_1 \la d_2]  \\
	\Psi_{3} [d_1c_1,d_2c_2,d_3c_3] &= [d_1,c_1 \la d_2, ((c_1 \ra d_2)c_2) \la d_3] + [c_1;d_2,c_2 \la d_3]\\
	&\qquad +  [c_1 \ra d_2, c_2;d_3] + [c_1 \ra (d_2(c_2 \la d_3)),c_2 \ra d_3,c_3].
\end{align*}
It is routine to verify that the $\Psi_k$ extend to natural transformations $\Psi_k \colon C_k^{\bowtie} \to C_k^{\Tot}$.

\subsection{The statement of the main theorem}
\label{subsec:the actual theorem}

We state our main homology theorem and outline the proof.
We write $\Ch$ for the category of abelian chain complexes and chain maps.

\begin{thm}\label{thm:cohomologies_are_the_same}
	The formulae \eqref{eq:eilenberg-zilber}, \eqref{eq:pi}, and \eqref{eq:psi} determine natural chain equivalences such that the diagram	
	\begin{equation}\label{eq:iso_theorem_diagram_half}
		\begin{tikzcd}[ampersand replacement=\&, column sep=30pt, row sep=25pt]
			{C_{\bullet}^{\Delta}} \&\& {C_{\bullet}^{\Tot}} \\
			\& {C_\bullet^{\bowtie}}
			\arrow["\nabla"', from=1-3, to=1-1]
			\arrow["\Pi", from=1-1, to=2-2]
			\arrow["\Psi", from=2-2, to=1-3]
		\end{tikzcd}
	\end{equation}
	commutes up to natural chain homotopy. They induce natural isomorphisms
	\[
	H^{\bowtie}_{\bullet} \cong H^{\Delta}_{\bullet} \cong H^{\Tot}_{\bullet}
	\]
	 of functors from $\MP$ to $\Ch$.
	In particular, for any matched pair $(\Cc,\Dd)$,
	\[
	H^{\bowtie}_k(\Cc,\Dd) \cong H^{\Delta}_k(\Cc,\Dd) \cong H^{\Tot}_k(\Cc,\Dd).
	\]
\end{thm}

Before commencing the proof of this theorem, we record a corollary.
Recall that for us, given a chain complex $(C_\bullet,d_\bullet)$ and an abelian group $A$, the cohomology with coefficients in $A$ is the cohomology of the dual cochain complex $(\Hom(C_{\bullet},A),d_{\bullet}^{*})$.
\begin{cor}\label{cor:main_theorem_cohomology}
	For any fixed abelian group $A$, the duals of the natural chain maps $\nabla$, $\Pi$, and $\Psi$ induce natural isomorphisms
	\[
	H^{\bullet}_{\bowtie}(\,\cdot\, ; A) \cong 	H^{\bullet}_{\Delta}(\,\cdot\, ; A) \cong
	H^{\bullet}_{\Tot}(\,\cdot\, ; A)
	\]
	of cohomology functors with coefficients in $A$.
\end{cor}
\begin{proof}
	Dualising all the maps in a chain homotopy diagram yields a cochain homotopy.
\end{proof}

That $\nabla$ induces a natural isomorphism is the content of a general form of the Eilenberg--Zilber Theorem. Following \cite[\S 8.5.4]{Wei94} the \emph{Alexander--Whitney map} $\Delta \colon C_\bullet^{\Delta} \to C_{\bullet}^{\Tot}$ is defined as follows: for $p,q$ such that $p + q = n$, define $\Delta_{p,q} \colon C_{n}^{\Delta} \to C_{p,q}$ by
\[
\Delta_{p,q} \coloneqq \underbrace{\partial_{p+1}^h \cdots \partial_n^h}_{q \text{ terms}} \circ \underbrace{\partial_0^v \cdots \partial_0^v}_{p \text{ terms}}.
\]
Then the map $\Delta$ is defined by
\begin{equation}\label{eq:alexander-whitney}
	\Delta \coloneqq \bigoplus_{p+q = n} \Delta_{p,q}.
\end{equation}
\begin{thm}[{\cite[Theorem 8.5.1]{Wei94}}]\label{thm:Eilenberg--Zilber}
The map $\nabla \colon C_{\bullet}^{\Tot} \to C_{\bullet}^\Delta$ of~\eqref{eq:eilenberg-zilber} induces a natural isomorphism $H^{\Delta}_{\bullet} \cong H^{\Tot}_{\bullet}$, with inverse induced by the map $\Delta
\colon C_{\bullet}^{\Delta} \to C_{\bullet}^{\Tot}$ of~\eqref{eq:alexander-whitney}.
\end{thm}

Given Theorem~\ref{thm:Eilenberg--Zilber}, to prove Theorem~\ref{thm:cohomologies_are_the_same} it suffices to establish the natural isomorphism $H^{\bowtie}_{\bullet} \cong H^{\Delta}_{\bullet}$. To do this we fill out a diagram
\begin{equation}\label{eq:iso_theorem_diagram}
	\begin{tikzcd}[ampersand replacement=\&, column sep=40pt, row sep=30pt]
		{C_{\bullet}^{\Delta}} \&\& {C_{\bullet}^{\Tot}} \\
		\& {C_\bullet^{\bowtie}}
		\arrow["\Delta", shift left=1, from=1-1, to=1-3]
		\arrow["\nabla", shift left=1, from=1-3, to=1-1]
		\arrow["\Pi", shift left=1, from=1-1, to=2-2]
		\arrow["\RPi", shift left=1, from=2-2, to=1-1]
		\arrow["\RPsi", shift left=1, from=1-3, to=2-2]
		\arrow["\Psi", shift left=1, from=2-2, to=1-3]
	\end{tikzcd}
\end{equation}
of natural chain equivalences that commutes up to natural chain homotopy. We use the \emph{method of acyclic models} (see \cite{Rot88} for instance).
The details occupy Subsections \ref{sec:homology_of_models}~and~\ref{subsec:proof_of_theorem}.

We show in Section~\ref{sec:homology_of_models} that the model matched pairs $(\Ee_{2k},\Ff_{2k})$
satisfy $H_{p}^{\bowtie}(\Ee_{2k},\Ff_{2k}) = 0 = H_{p}^{\Delta}(\Ee_{2k},\Ff_{2k})$ for all $p \ge 1$. So we can use these as the models in the method of acyclic models. We deduce that there exist natural chain equivalences between $C_{\bullet}^{\bowtie}$ and $C_{\bullet}^{\Delta}$ that induce natural isomorphisms on homology, and show how to recognise when given chain maps do the job.

In Subsection~\ref{subsec:PiPsidef}, we show that \eqref{eq:pi}~and~\eqref{eq:psi} are such chain maps. We also
 give explicit formulae for the remaining maps
$\RPi \coloneqq\nabla  \circ \Psi $ and $\RPsi \coloneqq   \Pi \circ \nabla$.

\subsection{Homological acyclicity of model matched pairs}\label{sec:homology_of_models}
The proof of Theorem~\ref{thm:cohomologies_are_the_same} hinges on homological properties of the model matched pairs $(\Ee_n,\Ff_n)$. Recall that the object set of $\Gamma_n \cong \Ee_n *\Ff_n$ is $X_n = \{\mathbf{a} =(a_L,a_R) \in \NN \times \NN  \mid 0 \le a_L + a_R \le n\}$. Each morphism of $\Gamma_n$ is a pair $(\mathbf{a},\mathbf{b}) \in X_n \times X_n$ such that $a_L \le b_L$ and $a_R \ge b_R$.

The map
$r \times s \colon \Gamma_n \to X_n \times X_n$ is injective. Hence,
\begin{equation}\label{eq:bijection_gamma_n_k}
	\Gamma_n^k \ni (\gamma_1,\ldots,\gamma_k) \mapsto (r(\gamma_1),s(\gamma_1),s(\gamma_2),\ldots,s(\gamma_k)) \in X_n^{(k)}.
\end{equation}
is a bijective correspondence between $\Gamma_n^k$ and
\[
X_n^{(k)} \coloneqq \{(\mathbf{a}_0,\mathbf{a}_1,\ldots,\mathbf{a}_k) \mid \mathbf{a}_i \in X_n,\, a_{i,L} \le a_{i+1,L}, \text{ and } a_{i,R} \ge a_{i+1,R} \text{ for all } i\}.
\]
Since  $C^{\bowtie}_k(\Ee_n,\Ff_n)$ is the free abelian group $\ZZ \Gamma_n^k$, we have $
C^{\bowtie}_k(\Ee_n,\Ff_n) \cong \ZZ X_n^{(k)}.
$

Let $\langle \mathbf{a}_0,\ldots,\mathbf{a}_k \rangle$ denote the generator of $\ZZ X_n^{(k)}$ corresponding to $(\mathbf{a}_0,\ldots,\mathbf{a}_k) \in X_n^{(k)}$. Using carets to denote elision of coordinates, the face and degeneracy maps on $C_k^{\bowtie}(\Ee_n,\Ff_n)$ are
\[
\partial^{\bowtie, i}_{k}\langle \mathbf{a}_0,\ldots,\mathbf{a}_k\rangle
= \langle \mathbf{a}_0,\ldots,\widehat{\mathbf{a}}_i,\ldots,\mathbf{a}_k\rangle
\quad
\text{and}
\quad
\sigma^{\bowtie, i}_{k}\langle \mathbf{a}_0,\ldots,\mathbf{a}_k\rangle
=\langle   \mathbf{a}_0,\ldots,\mathbf{a}_i,\mathbf{a}_i,\ldots,\mathbf{a}_k \rangle.
\]
A chain complex $(C_{\bullet},d_{\bullet})$ is \emph{acyclic} if $H_0(C_\bullet,d_{\bullet}) \cong \ZZ$ and $H_k(C_{\bullet}, d_{\bullet}) = 0$ for $k \ge 1$.

Recall that an \emph{initial object} in a category $\Cc$ is an object $v \in \Cc^0$ such that $w \Cc v$ has precisely one element for each $w \in \Cc^0$.

The following is well-known, but we could not find an explicit reference.

\begin{lem}
	\label{lem:initial_ob_contratible}
	Let $\Cc$ be a small category with an initial object $v$. Let $\bone$ be the category with a single morphism $1$. Let $\iota \colon \bone \to \Cc$ be the functor such that $\iota(1) = v$. Let $\rho$ be the unique functor from $\Cc$ to $\bone$. Then $\rho \circ \iota = \id_{\bone}$, and $(\iota \circ \rho)_{\bullet}\colon C_{\bullet}(\Cc) \to C_{\bullet}(\Cc)$ is chain-homotopic to $\id_{C_{\bullet}(\Cc)}$. In particular, $(C_{\bullet}(\Cc), d_{\bullet})$ is acyclic.
\end{lem}
\begin{proof}
	Clearly, $\rho \circ \iota = \id_{\bone}$.
	For each $w \in \Cc^0$ let $\tau_{w} \in \Cc$ be the unique morphism from $v$ to $w$.
	Fix $k \ge 0$. For $0 \le i \le k$ define $h^i \colon C_k(\Cc) \to C_{k+1}(\Cc)$ by
	\[
	h^i [ c_0,\ldots,c_{k-1} ] =
	\begin{cases*}
	[c_0,\ldots,c_i,\tau_{s(c_i)}, v, \ldots, v] & if $i > 0$\\
	[\tau_{r(c_0)}, v, \ldots, v] & if $i = 0$.
	\end{cases*}
	\]
	To see that $h$ is a simplicial homotopy we need to check that $\partial^0 h^0 = (\iota \circ \rho)_k$, that $\partial^{k+1}h^k = \id_{C_k(\Cc)}$, that $\partial^i h^j = h^{j-1} \partial^i$ for $i < j$, that $\partial^i h^j = h^j \partial^{i-1}$ for $i > j+1$, that $\sigma^i h^j = h^{j+1}\sigma^i$ for $i \le j$, and that $\sigma^i h^j = h^j \sigma^{i-1}$ for $i > j$. For the first two identities, we calculate	
	\begin{align*}
		\partial^0 h^0 [c_0,\ldots,c_{k-1}]
		&= \partial^0[\tau_{r(c_0)}, v, \ldots, v]
		= [v,\ldots,v], \text{ and}\\
		\partial^{k+1} h^k [c_0,\ldots,c_{k-1}]
		&= \partial^{k+1} [c_0,\ldots,c_{k-1},\tau_{s(c_{k-1})}] = [c_0,\ldots,c_{k-1}].
	\end{align*}
	The remaining four conditions follow from similar calculations. For example, if $0< i < j$, then
	\begin{align*}
	\partial^i h^j [c_0,\ldots,c_{k-1}]
	&= \partial^i [c_0,\ldots, c_j, \tau_{s(c_j)},v, \ldots, v]
	= [c_0,\ldots, c_{i-1} c_i,\ldots, c_j,\tau_{s(c_j)},v, \ldots, v]\\
	&= h^{j-1} [c_0,\ldots, c_{i-1} c_i,\ldots, c_j,\ldots, c_{k-1}]
	= h^{j-1}\partial^i[c_0,\ldots,c_{k-1}].
	\end{align*}
	Hence, the simplicial maps $(\iota \circ p)_{\bullet}$ and $\id_{C_{\bullet}(\Cc)}$ are simplicially homotopic. So
	$
	s_k \coloneqq \sum_{i=0}^k (-1)^i h_k^i
	$
	defines a chain homotopy $s$ between $(\iota \circ p)_{\bullet}$ and $\id_{C_{\bullet}(\Cc)}$ \cite[Lemma 8.3.13]{Wei94}.
	
	The final statement follows from acyclicity of $(C_{\bullet}(\bone),d_{\bullet})$.
\end{proof}

\begin{lem}\label{lem:bowtie_models_acyclic}
	For each $n \ge 0$ the chain complex $(C_{\bullet}^{\bowtie}(\Ee_n,\Ff_n), d)$ is acyclic.
\end{lem}

\begin{proof}
As $(n,0)$ is an initial object in $\Gamma_n = \Ee_n \bowtie \Ff_n$, the result follows from Lemma~\ref{lem:initial_ob_contratible}.
\end{proof}
Like the description coming from \eqref{eq:bijection_gamma_n_k} for $C_k^{\bowtie}(\Ee_n,\Ff_n)$, there is a tractible description for $C_{k,l}(\Ee_n,\Ff_n)$. The formula
\begin{align*}
	\Ee_n^k * \Ff_n^l \ni& ((p_0,q_0),(p_1,q_0),\ldots,(p_k,q_0),\ldots, (p_k,q_{l-1}),(p_k,q_l))\\
	&\quad  \mapsto (p_0,p_1,\ldots,p_k;q_0,\ldots, q_{l-1},q_l) \in Y_n^{(k,l)}
\end{align*}
is a bijection between $\Ee_n^k * \Ff_n^l$ and
\[
Y_n^{(k,l)} \coloneqq \Big\{ (p_0,\ldots,p_k;q_0,\ldots,q_l) \in \NN^{2k}\mid p_i \le p_{i+1}, \, q_i \ge q_{i+1}, \text{ and } p_k + q_0 \le n  \Big\},
\]
and induces an isomorphism
$
C_{k,l}(\Ee_{n},\Ff_n) \cong \ZZ Y_n^{(k,l)}.
$

We write $\langle p_0,\ldots,p_k;q_0,\ldots,q_l \rangle$ for the generator of $\ZZ Y_n^{(k,l)}$ that corresponds to the tuple $(p_0,\ldots,p_k;q_0,\ldots,q_l) \in Y_n^{(k,l)}$. The face maps in the double complex become
\begin{align*}
	\partial^{h,i} \langle p_0,\ldots,p_k;q_0,\ldots,q_l \rangle
	&= \langle p_0,\ldots,\widehat{p_i},\ldots,p_k;q_0,\ldots, q_l \rangle \quad \text{and}\\
	\partial^{v,i} \langle p_0,\ldots,p_k;q_0,\ldots,q_l \rangle
	&= \langle p_0,\ldots,p_k;q_0,\ldots,\widehat{q_i},\ldots, q_l \rangle.
 \end{align*}
The degeneracy maps become
\begin{align*}
	\sigma^{h,i} \langle p_0,\ldots,p_k;q_0,\ldots,q_l \rangle
	&= \langle p_0,\ldots,p_{i-1},p_i,p_i,p_{i+1},\ldots, p_k;q_0,\ldots, q_l \rangle \quad \text{and}\\
	\sigma^{v,i} \langle p_0,\ldots,p_k;q_0,\ldots,q_l \rangle
	&= \langle p_0,\ldots, p_k;q_0,\ldots,q_{i-1},q_i,q_i,q_{i+1},\ldots,q_l \rangle.
\end{align*}
In particular, for the diagonal complex $C_{\bullet}^{\Delta}(\Ee_n,\Ff_n)$ the face and degeneracy maps are
\begin{align*}
	\partial^{\Delta, i} \langle p_0,\ldots,p_k;q_0,\ldots,q_k \rangle &= \langle p_0,\ldots,\widehat{p_i},\ldots,p_k;q_0,\ldots,\widehat{q_i},\ldots, q_k \rangle \quad \text{and}\\
	\sigma^{\Delta, i} \langle p_0,\ldots,p_k;q_0,\ldots,q_k \rangle &= \langle p_0,\ldots, p_{i-1},p_i,p_i,p_{i+1},\ldots,p_k;q_0,\ldots,q_{i-1},q_i,q_i,q_{i+1},\ldots,q_k \rangle.
\end{align*}

\begin{lem}\label{lem:diagonal_models_acyclic}
	The diagonal complex $(C_\bullet^\Delta(\Ee_{n},\Ff_n), d^{\Delta})$ is acyclic.
\end{lem}

\begin{proof}
	Consider the directed graph
	$G_n = 0 \overset{e_0}{\leftarrow} 1 \overset{e_1}{\leftarrow} \cdots \overset{e_{n-1}}{\leftarrow} n$. Since $n$ is an initial object for $G_n^*$, Lemma~\ref{lem:initial_ob_contratible} implies that $(C_\bullet(G_n^*), d_{\bullet})$ is acyclic. 	
	So it suffices to show that $(C_\bullet^\Delta(\Ee_{n},\Ff_n), d^{\Delta})$ is chain-homotopic to $(C_\bullet(G_n^*),d)$.
	
	The group $C_k(G_n^*)$ is freely generated by $k$-tuples $\langle p_0,\ldots,p_k \rangle$ where $0 \le p_i \le p_{i+1} \le n$ for each $0 \le i < k$.
	The functor $\iota \colon G_n^* \hookrightarrow \Ee_n$ given by $\iota(e_p) = e_{p,0}$ induces a chain map $\iota \colon C_{\bullet}(G_n^*) \to C_{\bullet}^\Delta(\Ee_{n},\Ff_n)$ satisfying
	$
	\iota_k \langle p_0,\ldots,p_k\rangle  = \langle p_0,\ldots,p_k;0,\ldots,0 \rangle.
	$
	The functor $\rho \colon \Ee_n \to G_n^*$ defined by $\rho (e_{p,q}) = e_p$,  induces a chain map $\rho \colon  C_{\bullet}^\Delta(\Ee_{n},\Ff_n) \to C_{\bullet}(G_n^*)$ satisfying
	$
	\rho_k \langle p_0,\ldots,p_k;q_0,\ldots, q_k\rangle = \langle p_0,\ldots,p_k \rangle.
	$
	We have $\rho_k \circ \iota_k = \id_{C_k(G_n^*)}$.	
	For $0 \le i \le k$ define $h^i \colon C_k^{\Delta}(\Ee_{n},\Ff_n) \to C_{k+1}^\Delta (\Ee_n,\Ff_n)$ by
	\[
	h^i \langle p_0,\ldots,p_k;q_0,\ldots,q_k \rangle =
	\langle
	p_0, \ldots, p_{i-1}, p_i, p_i, p_{i+1} , \ldots, p_k; 0, \ldots, 0, q_i, q_{i+1}, \ldots, q_k
	\rangle.
	\]
	Direct calculation shows that $\partial^0 h^0 = \id_{C_k^{\Delta}(\Ee_n,\Ff_n)}$ and $\partial^{k+1} h^k = (\iota \circ \rho)_k$.
	
		It is routine to check that $\partial^i h^j = h^{j-1} \partial^i$ for $i < j$ and $\partial^i h^j = h^j \partial^{i-1}$ for $i >
j+1$. Similarly, $\sigma^i h^j = h^{j+1}\sigma^i$ for $i \le j$ and $\sigma^i h^j = h^j \sigma^{i-1}$ for $i > j$. It follows that the simplicial maps $(\iota \circ p)_{\bullet}$ and $\id_{C_{\bullet}^{\Delta}(\Ee_n,\Ff_n)}$ are
simplicially homotopic.
\end{proof}

We identify some particularly useful chains in the categorical and diagonal homology of $(\Ee_{2k},\Ff_{2k})$.
For each $k \ge 0$ define $x_k \in C_k^{\bowtie}(\Ee_{2k},\Ff_{2k})$ and $y_k \in C_k^{\Delta}(\Ee_{2k},\Ff_{2k})$ by
\begin{align}
	\begin{split}\label{eq:x_k}
			x_{k} &\coloneqq
		[f_{0,2k}e_{0,2k},f_{1,2k-1}e_{1,2k-1},\ldots,f_{k-1,k+1}e_{k-1,k+1}]
		= \langle (0,2k),(1,2k-1),\ldots,(k,k)\rangle
	\end{split}
\end{align}
and
\begin{align}
	\begin{split}\label{eq:y_k}
		y_k &\coloneqq  [e_{0,k},e_{1,k},\ldots,e_{k-1,k};f_{k,k-1},\ldots,f_{k,1}, f_{k,0}]
		= \langle 0,1,\ldots,k-1,k;k,k-1,\ldots,1,0\rangle.
	\end{split}
\end{align}
Pictorially, $x_k$ and $y_k$ correspond to the following composable tuples  in $\Ee_{2k} \bowtie \Ff_{2k}$:
\def\sf {2.3}
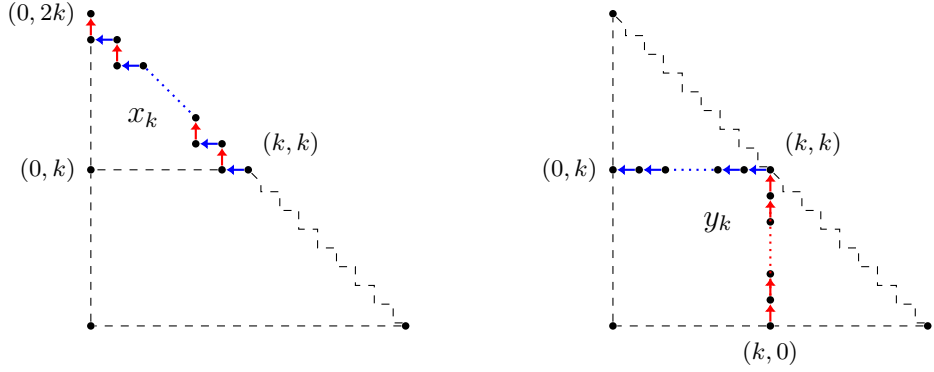
\begin{equation*}\label{eq:x_and_y_pics}
\begin{tikzpicture}
	[vertex/.style={
		circle,
		fill=black,
		inner sep=1pt},
	edge/.style={
		thick,
		-{Latex[length=1mm, width=1.5mm]}
		},
	edged/.style={
		dashed
		},
	scale = 1.5]
	
	\clip (-0.45*\sf,-0.15*\sf) rectangle (1.3*\sf,1.3*\sf);
	
	\node at (0.2*\sf,0.8*\sf) {$x_k$};%
	
	
	\node[vertex] (0-0) at (0,0) {};%
	\node[label=below:{\scriptsize{\phantom{$(k,0)$}}}] (1-0) at (0.6*\sf,0) {};%
	\node[vertex] (2-0) at (1.2*\sf,0) {};%
	
	\node[vertex,label=left:{\scriptsize$(0,k)$}] (0-1) at (0,0.6*\sf) {};%
	\node[vertex,label=above right:{\scriptsize$(k,k)$}] (1-1) at (0.6*\sf,0.6*\sf) {};%
	
	\node[vertex,label=left:{\scriptsize$(0,2k)$}] (0-2) at (0,1.2*\sf) {};%

	
	\node[vertex] (1a) at (0,1.1*\sf) {};
	\node[vertex] (2a) at (0.1*\sf,1.1*\sf) {};
	\node[vertex] (3a) at (0.1*\sf,1.0*\sf) {};
	\node[vertex] (4a) at (0.2*\sf,1.0*\sf) {};
	
	\node[vertex] (1b) at (0.5*\sf,0.6*\sf) {};
	\node[vertex] (2b) at (0.5*\sf,0.7*\sf) {};
	\node[vertex] (3b) at (0.4*\sf,0.7*\sf) {};
	\node[vertex] (4b) at (0.4*\sf,0.8*\sf) {};

	\draw[edge,red] (1a) to (0-2);
	\draw[edge,blue] (2a) to (1a);
	\draw[edge,red] (3a) to (2a);
	\draw[edge,blue] (4a) to (3a);
	
	\draw[edge,blue] (1-1) to (1b);
	\draw[edge,red] (1b) to (2b);
	\draw[edge,blue] (2b) to (3b);
	\draw[edge,red] (3b) to (4b);

	\draw[edged,dashed] (2-0) to node[midway, anchor=south,inner sep=2.5pt] {} (0-0);
	
	\draw[edged,dashed] (1b) to node[midway, anchor=south,inner sep=2.5pt] {} (0-1);

	\draw[edged,dashed] (0-0) to node[midway, anchor=east,inner sep=2.5pt] {} (0-1);
	\draw[edged,dashed] (0-1) to node[midway, anchor=east,inner sep=2.5pt] {} (1a);
	
	
	\draw[edged,decoration={zigzag},decorate] (2-0) to (1-1);
	\draw[dotted, thick, blue] (4a) to (4b);

\end{tikzpicture}
\qquad
\begin{tikzpicture}
	[vertex/.style={circle, fill=black, inner sep=1pt},
	edge/.style={
		thick,
		-{Latex[length=1mm, width=1.5mm]}
	},
	edged/.style={dashed},
	scale = 1.5]
	
	\clip (-0.45*\sf,-0.15*\sf) rectangle (1.3*\sf,1.3*\sf);
	
	\node at (0.4*\sf,0.4*\sf) {$y_k$};%
	
	
	\node[vertex] (0-0) at (0,0) {};%
	\node[vertex,label=below:{\scriptsize$(k,0)$}] (1-0) at (0.6*\sf,0) {};%
	\node[vertex] (2-0) at (1.2*\sf,0) {};%
	
	\node[vertex,label=left:{\scriptsize$(0,k)$}] (0-1) at (0,0.6*\sf) {};%
	\node[vertex,label=above right:{\scriptsize$(k,k)$}] (1-1) at (0.6*\sf,0.6*\sf) {};%
	
	\node[vertex,label=left:\phantom{\scriptsize$(0,2k)$}] (0-2) at (0,1.2*\sf) {};%

	\node[vertex] (a1) at (0.1*\sf, 0.6*\sf) {};
	\node[vertex] (a2) at (0.2*\sf, 0.6*\sf) {};
	
	\node[vertex] (b1) at (0.4*\sf, 0.6*\sf) {};
	\node[vertex] (b2) at (0.5*\sf, 0.6*\sf) {};
	
	\node[vertex] (c1) at (0.6*\sf, 0.1*\sf) {};
	\node[vertex] (c2) at (0.6*\sf, 0.2*\sf) {};
	
	\node[vertex] (d1) at (0.6*\sf, 0.4*\sf) {};
	\node[vertex] (d2) at (0.6*\sf, 0.5*\sf) {};

	\draw[edged,dashed] (1-0) to node[midway, anchor=south,inner sep=2.5pt] {} (0-0);
	\draw[edged,dashed] (2-0) to node[midway, anchor=south,inner sep=2.5pt] {} (1-0);
	
	\draw[edge,blue] (a2) to (a1);
	\draw[edge,blue] (a1) to (0-1);
	
	\draw[thick,blue, dotted] (b1) to (a2);
	
	\draw[edge,blue] (1-1) to (b2);
	\draw[edge,blue] (b2) to (b1);
	
	\draw[edged,dashed] (0-0) to node[midway, anchor=east,inner sep=2.5pt] {} (0-1);
	\draw[edged,dashed] (0-1) to node[midway, anchor=east,inner sep=2.5pt] {} (0-2);
	
	\draw[edge,red] (1-0) to (c1);
	\draw[edge,red] (c1) to (c2);
	
	\draw[thick, red, dotted] (c1) to (d2);
	
	\draw[edge,red] (d1) to (d2);
	\draw[edge,red] (d2) to (1-1);
	
	\draw[edged,decoration={zigzag},decorate] (2-0) to (1-1);
	\draw[edged,decoration={zigzag},decorate] (1-1) to (0-2);
	
\end{tikzpicture}
\end{equation*}

By Corollary~\ref{cor:matched_pair_morphisms}, a matched pair morphism $(\Ee_{n},\Ff_{n}) \to (\Cc,\Dd)$ corresponds to a functor $\Gamma_n \to \Cc \bowtie \Dd$ taking $\Ee_n$ to $\Cc$ and $\Ff_n$ to $\Dd$. 	
For each $\gamma =
(d_0c_0,\ldots,d_{k-1}c_{k-1})$ in $(\Cc \bowtie \Dd)^k$, Lemma~\ref{lem:matched_pair_universality} gives a morphism $h_\gamma^{\bowtie} \colon  \Gamma_{2k} \to \Cc \bowtie \Dd$ such that
\[
h^{\bowtie}_{\gamma} (f_{p,2k-1-p}e_{p,2k-1-p}) = \begin{cases}
	d_pc_p & \text{if } 0 \le p < k\\
	s(c_{k-1}) & \text{if } k \le p < 2k.
\end{cases}
\]
For $\lambda = (c_0,\ldots,c_{k-1},d_0,\ldots,d_{k-1}) \in \Cc^k * \Dd^k$, with $d'_i = r(c_i)$ and $c'_i = s(d_i)$,
Lemma~\ref{lem:matched_pair_universality} applied to $(d'_0c_0, \cdots, d'_{k-1}c_{k-1}, d_0c'_0, \dots, d_{k-1}c'_{k-1})$ yields a morphism
$h_{\lambda}^{\Delta} \colon \Gamma_{2k} \to \Cc \bowtie \Dd$ such that
\[
h_{\lambda}^{\Delta}(f_{p,2k-1-p}e_{p,2k-1-p}) = \begin{cases}
c_p & \text{if } 0 \le p < k\\
d_{p-k} & \text{if } k \le p < 2k.
\end{cases}
\]

\begin{lem}\label{lem:free-functors}Let $(\Cc,\Dd)$ be a matched pair. For $\gamma \in (\Cc \bowtie \Dd)^k$ and $\lambda \in \Cc^k * \Dd^k$ we have
		 $[\gamma] = C^{\bowtie}_k(h^{\bowtie}_{\gamma})(x_k)$ and  $[\lambda]=C^{\Delta}_k(h^{\Delta}_\lambda)(y_k)$.
		Moreover,
		\[
		\{ C^{\bowtie}_{k} (h)(x_{k}) \mid  h \colon (\Ee_{2k},\Ff_{2k}) \to (\Cc,\Dd)\}
		\quad
		\text{ and }
		\quad
		\{ C_k^{\Delta} (h) (y_k) \mid h \colon (\Ee_{2k}, \Ff_{2k}) \to (\Cc,\Dd) \}
		\]
		generate $C^{\bowtie}_k(\Cc,\Dd)$ and $C^{\Delta}_k(\Cc,\Dd)$ respectively.
\end{lem}
\begin{proof}
	That $[\gamma] = C^{\bowtie}_k(h^{\bowtie}_{\gamma})(x_k)$ and  $[\lambda]=C^{\Delta}_k(h^{\Delta}_\lambda)(y_k)$ follow immediately from the definitions of $h^{\bowtie}_{\gamma}$ and $h^{\Delta}_{\lambda}$.
	For the second statement,  let $h \colon (\Ee_{2k},\Ff_{2k}) \to (\Cc,\Dd)$ be a matched pair morphism. Then	$C^{\bowtie}_{k} (h)(x_{k}) = [h(f_{0,2k}e_{0,2k}),\ldots,h(f_{k-1,k+1}e_{k-1,k+1})]$.
So
	\begin{align*}
		\{ C^{\bowtie}_{k} (h)(x_{k}) \mid  h \colon (\Ee_{2k},\Ff_{2k}) \to (\Cc,\Dd)\}
		\supseteq	\{ C^{\bowtie}_{k} (h^{\bowtie}_{\gamma})(x_{k})  \mid \gamma \in (\Cc \bowtie \Dd)^{k}\}
		= \{ [\gamma] \mid \gamma \in (\Cc \bowtie \Dd)^{k}\},
	\end{align*}
which generates $C_k^{\bowtie}(\Cc,\Dd)$.
	Similarly,
	\[
	\{ C_k^{\Delta} (h) (y_k) \mid h \colon (\Ee_{2k}, \Ff_{2k}) \to (\Cc,\Dd) \} \supseteq \{[\lambda] \mid \lambda \in \Cc^k * \Dd^k\}
	\]
	generates $C_k^{\Delta}(\Cc, \Dd)$.
\end{proof}

In the terminology of {\cite[pp. 239--240]{Rot88}}, Lemma~\ref{lem:free-functors} says that the functors $C_{k}^{\bowtie}$ and $C_k^{\Delta}$ from $\MP$ to $\Ch$ are \emph{free} with bases $\{x_k\}$ and $\{y_k\}$, giving the following lemma.

\begin{lem}[{\cite[Lemma 9.10]{Rot88}}] \label{lem:all_you_need_is_models}
	If $G \colon \MP \to \Ab$ is a functor and $g \in G(\Ee_{2k},\Ff_{2k})$, then there is a unique natural transformation $\alpha \colon C_k^{\bowtie} \to G$ such that
	$
	\alpha_{(\Ee_{2k},\Ff_{2k})} (x_k) = g,$
	and a unique natural transformation $\beta \colon C_k^{\Delta} \to G$ such that
	$\beta_{(\Ee_{2k},\Ff_{2k})} (y_k) = g. $
\end{lem}

The proof of \cite[Lemma 9.10]{Rot88} describes the natural transformations of Lemma~\ref{lem:all_you_need_is_models}: for $[\gamma] \in C_k^{\bowtie}(\Cc,\Dd)$, Lemma~\ref{lem:free-functors} gives $h_{\gamma}^{\bowtie} \colon (\Ee_{2k},\Ff_{2k}) \to (\Cc,\Dd)$ such that $h_{\gamma}^{\bowtie}(x_k) = [\gamma]$, and then $\alpha_{(\Cc,\Dd)}([\gamma]) \coloneqq G(h_\gamma^{\bowtie})(g)$. Similarly, $\beta_{(\Cc,\Dd)}([\lambda]) \coloneqq G(h_\lambda^{\Delta})(g)$  for $[\lambda] \in C_k^{\Delta}(\Cc,\Dd)$.

\subsection{Proof of the main theorem}\label{subsec:proof_of_theorem}

To prove Theorem~\ref{thm:cohomologies_are_the_same} we construct a chain equivalence between $C_{\bullet}^{\bowtie}$ and $C_{\bullet}^{\Delta}$ inductively using~\cite[Theorem 9.12]{Rot88}.

\begin{lem}\label{lem:H_0_the_same}
 The identity map $C_0^{\Delta}(\Cc,\Dd) \coloneqq \ZZ X \overset{\id}{\to} \ZZ X  \eqqcolon  C_0^{\bowtie}(\Cc,\Dd)$ induces a natural isomorphism  $\widetilde{\operatorname{id}} : H_0^{\bowtie} \cong H_0^{\Delta}$.
\end{lem}
\begin{proof}
	Fix a matched pair $(\Cc,\Dd)$ with objects $X$. Identifying $C_0^{\Delta}(\Cc,\Dd)$ with $C_0^{\bowtie}(\Cc,\Dd)$ via the identity map on $\ZZ X$, it suffices to show that $\im(d^{\bowtie}) = \im(d^{\Delta})$ in $\ZZ X$. If $[d,c] \in C_1^{\bowtie} (\Cc,\Dd)$, then	
	$
	d^{\bowtie} [d,c] = [s(c)] - [r(d)] = d^{\Delta} [c,s(c)] +  d^{\Delta} [r(d),d] \in \im (d^{\Delta}).
	$
	If $[c,d] \in C_1^{\Delta} (\Cc,\Dd)$, then
	$
	d^{\Delta} [c,d] = [s(d)] - [r(c)] = d^{\bowtie}[c \bowtie d] \in \im(d^{\bowtie}).
	$
\end{proof}

\begin{prop}\label{prop:the_real_main_theorem}
There exist natural chain maps $\alpha \colon C_{\bullet}^{\bowtie} \to C_{\bullet}^{\Delta}$ and $\beta \colon C_{\bullet}^{\Delta} \to C_{\bullet}^{\bowtie}$ such that $\alpha \circ \beta$ is naturally chain-homotopic to $\id_{C_{\bullet}^{\Delta}}$ and $\beta \circ \alpha$ is naturally chain-homotopic to $ \id_{C_{\bullet}^{\bowtie}}$ such that $\alpha$ and $\beta$ lift the natural isomorphism $H_0^{\bowtie} \cong H_0^{\Delta}$, in the sense that the diagram
	\begin{equation}	\label{eq:ayclic_models_diagram}
		\begin{tikzcd}[ampersand replacement=\&]
			{} \& {C_2^{\bowtie}} \& {C_1^{\bowtie}} \& {C^{\bowtie}_0} \& {H^{\bowtie}_0} \& 0 \\
			{} \& {C_2^{\Delta}} \& {C_1^{\Delta}} \& {C_0^{\Delta}} \& {H^{\Delta}_0} \& 0\\
			{} \& {C_2^{\bowtie}} \& {C_1^{\bowtie}} \& {C^{\bowtie}_0} \& {H^{\bowtie}_0} \& 0
			\arrow[dashed, from=1-1, to=1-2]
			\arrow[dashed, from=2-1, to=2-2]
			\arrow[dashed, from=3-1, to=3-2]
			\arrow["{d_1^{\bowtie}}", from=1-2, to=1-3]
			\arrow["{d_0^{\bowtie}}", from=1-3, to=1-4]
			\arrow["{d_1^{\Delta}}", from=2-2, to=2-3]
			\arrow["{d_0^{\Delta}}", from=2-3, to=2-4]
			\arrow["{d_1^{\bowtie}}", from=3-2, to=3-3]
			\arrow["{d_0^{\bowtie}}", from=3-3, to=3-4]
			\arrow[from=1-4, to=1-5]
			\arrow[from=2-4, to=2-5]
			\arrow[from=3-4, to=3-5]
			\arrow[from=1-5, to=1-6]
			\arrow[from=2-5, to=2-6]
			\arrow[from=3-5, to=3-6]
			\arrow["\widetilde{\operatorname{id}}"', from=1-5, to=2-5]
			\arrow["\widetilde{\operatorname{id}}"', from=2-5, to=3-5]
			\arrow["{\alpha_0}"', shift right=1, from=1-4, to=2-4]
			\arrow["{\beta_0}"', shift right=1, from=2-4, to=3-4]
			\arrow["{\alpha_1}"', shift right=1, from=1-3, to=2-3]
			\arrow["{\beta_1}"', shift right=1, from=2-3, to=3-3]
			\arrow["{\alpha_2}"', shift right=1, from=1-2, to=2-2]
			\arrow["{\beta_2}"', shift right=1, from=2-2, to=3-2]
		\end{tikzcd}
	\end{equation}	
	of natural transformations commutes. If $\alpha' \colon C_{\bullet}^{\bowtie} \to C_{\bullet}^{\Delta}$ and $\beta' \colon C_{\bullet}^{\Delta} \to C_{\bullet}^{\bowtie}$ are chain maps that lift the natural isomorphism $H_0^{\bowtie} \cong H_0^{\Delta}$, then they are naturally chain-homotopic to $\alpha$ and $\beta$.

For $k \ge 0$, let $x_k, y_k$ be as in \eqref{eq:x_k}~and~\eqref{eq:y_k}. If for each $k \ge 0$, $\alpha_k \colon C_k^{\Delta} \to C_k^{\bowtie}$ is a natural transformation such that $d^\Delta \circ \alpha_k (x_k) = \alpha_{k-1} \circ d^{\Delta}(x_k)$, then $\alpha = (\alpha_k)$ is a natural chain equivalence from
$C_{\bullet}^{\bowtie}$ to $C_{\bullet}^{\Delta}$.
Similarly, if for each $k \ge 0$, $\beta_k \colon C_k^{\bowtie} \to C_k^{\Delta}$ is a natural transformations such that $d^{\bowtie} \circ \beta_k (y_k) = \beta_{k-1} \circ d^{\bowtie} (y_k)$, then $\beta = (\beta_k)$ is a natural chain equivalence
from $C_{\bullet}^{\Delta}$ to $C_{\bullet}^{\bowtie}$.
\end{prop}

The result is standard and follows from \cite[Theorem 9.12]{Rot88}, but we include some details to describe the resulting isomorphisms in homology explicitly.

\begin{proof}
	The morphisms $\alpha_0$ and $\beta_0$ are induced by the identity maps on objects.
	We start by constructing $\alpha$.
	Suppose that there exists maps $\alpha_n$, for $n < k$ such that the right-most $n+1$ squares of~\eqref{eq:ayclic_models_diagram} commute. Consider the matched pair $(\Ee_{2k},\Ff_{2k})$ and let $x_k \in C_k^{\bowtie}(\Ee_{2k},\Ff_{2k})$ be as in~\eqref{eq:x_k}. Commutativity of~\eqref{eq:ayclic_models_diagram} implies that $d^{\Delta}\alpha_{k-1}d^{\bowtie} = (d^{\Delta})^2 \alpha_{k-2} = 0$. Lemma~\ref{lem:bowtie_models_acyclic} implies that $(C_{\bullet}^{\bowtie}(\Ee_{2k},\Ff_{2k}),d^{\bowtie})$ is acyclic, so there exists $\ol{x}_k \in
C_k^{\Delta}(\Ee_{2k},\Ff_{2k})$ such that $d^{\Delta}(\ol{x}_k) = \alpha_{k-1} d^{\bowtie}(x_k)$. So Lemma~\ref{lem:all_you_need_is_models}
yields a unique natural transformation $\alpha_k \colon C_k^{\bowtie} \to C_k^{\Delta}$ such that $\alpha_k^{(\Ee_{2k},\Ff_{2k})} (x_k)  = \ol{x}_k$ and $d^{\Delta}\alpha_k = \alpha_{k-1} d^{\bowtie}$.
	
	A similar construction using $y_k$ and Lemma~\ref{lem:diagonal_models_acyclic} gives $\beta_k$. By \cite[Theorem~9.12]{Rot88} $\alpha_k$ and $\beta_k$ induce natural chain equivalences $\alpha$ and $\beta$ and these are, up to natural chain homotopy, the unique chain equivalences lifting the isomorphism of Lemma~\ref{lem:H_0_the_same}
\end{proof}

\begin{proof}[Proof of Theorem~\ref{thm:cohomologies_are_the_same}]
	The result follows from Proposition~\ref{prop:the_real_main_theorem} and Theorem~\ref{thm:Eilenberg--Zilber}.
\end{proof}

\subsection{Explicit formulas for the natural isomorphisms between homology theories}\label{subsec:PiPsidef}

Proposition~\ref{prop:the_real_main_theorem} yields a natural isomorphism $C_\bullet^{\Delta} \cong C_\bullet^{\bowtie}$, and its final statement says how to recognise chain maps $\alpha$, $\beta$ that induce such an isomorphism.
We show that the map $\Pi$ of \eqref{eq:pi} and $\RPi \coloneqq \nabla \circ \Psi$ are such chain maps, and describe chain maps inducing the remaining arrows in~\eqref{eq:iso_theorem_diagram}. We first examine how $\Pi_k$ behaves on the model matched pairs $(\Ee_m,\Ff_m)$.

\begin{lem}\label{lem:pi_y}
	Fix $m,\,k \ge 0$, and consider $\Pi_k \colon C_{k}^{\Delta}(\Ee_{m},\Ff_{m}) \to C_k^{\bowtie} (\Ee_{m},\Ff_{m})$. We have
	\[
	\Pi_k\langle p_0,\ldots,p_k;q_0,\ldots,q_k \rangle= \langle (p_0,q_0),\ldots,(p_k,q_k) \rangle.
	\]
	In particular, the element $y_k \in C_{k}^{\Delta}(\Ee_{2k},\Ff_{2k})$ from~\eqref{eq:y_k}, satisfies
	\begin{equation}\label{eq:pi_k_y_k}
		\Pi_k (y_k) = [f_{0,k-1}e_{0,k-1},f_{1,k-2}e_{1,k-2},\ldots, f_{k-2,1}e_{k-2,1},f_{k-1,0}e_{k-1,0}].
	\end{equation}
\end{lem}
\begin{proof}
	We begin by establishing \eqref{eq:pi_k_y_k}. We proceed by induction on $k$. The case $k = 0$ is trivial. Suppose inductively that the analogue of \eqref{eq:pi_k_y_k} describes $\Pi_{k-1}(y_{k-1})$. Recall that \[y_k=[e_{0,k},e_{1,k},\ldots,e_{k-1,k};f_{k,k-1},\ldots,f_{k,1}, f_{k,0}],
	\] and let
	\[
	\lambda_{k-1} \coloneqq [e_{1,k},\ldots,e_{k-1,k};f_{k,k-1},\ldots,f_{k,1}] \in C_k^{\Delta}(\Ee_{2(k-1)},\Ff_{2(k-1)}).
	\] 	
	By Lemma~\ref{lem:free-functors} there exists a morphism $h_{\lambda_{k-1}}^{\Delta} \colon (\Ee_{2(k-1)},\Ff_{2(k-1)}) \to (\Ee_{2k},\Ff_{2k})$ such that $\lambda_{k-1} = C_{k-1}^{\Delta}(h_{\lambda_{k-1}}^{\Delta})(y_{k-1})$.
	The inductive hypothesis gives
	\begin{align*}
		(1_{\Ee_{2k}} * \Pi_{k-1} * 1_{\Ff_{2k}}) (y_k) = [e_{0,k},f_{1,k-1},e_{1,k-1},f_{2,k-2},\ldots, e_{k-2,2} , f_{k-1,1}, e_{k-1,1} ,f_{k,0}],
	\end{align*}
	and applying $\bowtie^k$ yields \eqref{eq:pi_k_y_k}.
	
	For the second statement, let $\gamma = \langle p_0,\ldots, p_k;q_0,\ldots,q_k \rangle$. 	Lemma~\ref{lem:free-functors} gives a morphism $h_\gamma^{\Delta} \colon (\Ee_{2k},\Ff_{2k}) \to (\Ee_{m},\Ff_{m})$ such that
$C_k^{\Delta}(h_\gamma^{\Delta})(y_k) = \gamma$. Naturality of $\Pi_k$ implies that $\Pi_k(\gamma) = C_k^{\bowtie}(h_{\gamma}^{\Delta}) \circ \Pi_k (y_k)$. So the result follows from \eqref{eq:pi_k_y_k} and the definition of $h_{\gamma}^{\Delta}$.
\end{proof}

\begin{prop}
	For $k \ge 0$, we have $\Pi_{k-1} \circ d^{\Delta} (y_k) = d^{\bowtie} \circ \Pi_k(y_k)$. In particular,  $\Pi \colon C_{\bullet}^{\Delta} \to C_{\bullet}^{\bowtie}$  induces a natural isomorphism on homology.
\end{prop}
\begin{proof}
	Fix $k \ge 0$. By Lemma~\ref{lem:pi_y},	for each $0 \le i \le k$, we have
	\begin{align*}
		\Pi_{k-1} \circ \partial_i^{\Delta} (y_k)
		&= \Pi_{k-1} \langle 0,1,\ldots,\widehat{i},\ldots, k-1,k;k,k-1,\ldots,\widehat{k-i},\ldots,1,0 \rangle \\
		&= \langle (0,k),(1,k-1),\ldots, \widehat{(i,k-i)},\ldots, (k-1,1),(k,0) \rangle
		=\partial_i^{\bowtie} \circ \Pi_k(y_k).
	\end{align*}
	A similar calculation gives $\Pi_{k+1} \circ \sigma_i^{\Delta} (y_k) = \sigma_i^{\bowtie} \circ \Pi_k(y_k)$ and so $d^{\bowtie} \circ \Pi_k(y_k) = \Pi_{k-1} \circ d^{\Delta}(y_k)$. The final statement follows from
Proposition~\ref{prop:the_real_main_theorem}.
\end{proof}

We next examine how the map $\Psi$ of \eqref{eq:psi} behaves on the model matched pairs.

\begin{lem} Fix $k,\,m \ge 0$ and $a = \langle \mathbf{a}_0,\ldots \mathbf{a}_k \rangle \in C_k^{\bowtie}(\Ee_{m},\Ff_{m})$. We have
	\begin{equation}\label{eq:psi_models}
		\Psi_k \langle \mathbf{a}_0,\ldots,\mathbf{a}_k \rangle
		= \sum_{i=0}^k \langle a_0^L, \ldots, a_i^L ; a_i^R, \ldots, a_k^R \rangle.
	\end{equation}
	In particular, $x_k \in C_k^{\bowtie}(\Ee_{2k},\Ff_{2k})$ as in~\eqref{eq:x_k}, satisfies
	\begin{equation}\label{eq:psi_models_xk}
		\Psi_k(x_k)
		= \sum_{i=0}^k  \langle 0,1,\ldots, i-1,i ; 2k-i,2k - i -1, \ldots,k+1, k \rangle.
	\end{equation}
\end{lem}
\begin{proof}
The formula \eqref{eq:bijection_gamma_n_k} gives a bijection between composable $q$-tuples in $\Gamma_m = \Ee_m \bowtie \Ff_m$ and the set $X_m^{(q)}$. We claim that under this identification, if $(\mathbf{a}_0,\ldots,\mathbf{a}_q) \in X_m^{(q)}$, then
	\begin{equation} \label{eq:tau_explicit_models}
			\tau_q (\mathbf{a}_0,\ldots,\mathbf{a}_q) = ((a_0^L,a_0^R),\ldots,(a_0^L,a_{q-1}^R),(a_0^L,a_q^R),(a_1^L, a_q^R)\ldots,(a_q^L,a_q^R))
	\end{equation}
	in $\Gamma_m^q$.
	The tuples $((a_0^L,a_0^R),(a_0^L,a_1^R),\ldots,(a_0^L,a_q^R))$ and $((a_0^L,a_q^R),\ldots,(a_{q-1}^L,a_q^R),(a_q^L,a_q^R))$ belong to  $\Ee_m^q \subseteq \Gamma_m^q$ and $\Ff_m^q\subseteq \Gamma_m^q$. Since $r \times s \colon \Gamma_m \to X_{m} \times X_m$ is injective, and factorisation in $\Ee_m^* \bowtie \Ff_m^*$ is unique by
Proposition~\ref{prop:bowtie_**}, the formula \eqref{eq:tau_explicit_models} follows.

Since $\Psi_{p,q} = \rho_{p,q} \circ (\tau_p * \tau_q)$ by definition, and $\Psi_k = \sum_{i=0}^k \Psi_{i,k-i}$ the identity~\eqref{eq:psi_models} holds; and~\eqref{eq:psi_models_xk} follows from~\eqref{eq:psi_models} since $x_k = \langle (0,2k),(1,2k-1), \ldots, (k,k) \rangle$.
\end{proof}

\begin{prop}
	For $k \ge 0$, we have $\Psi_{k-1} \circ d^{\bowtie}  = d^{\Tot} \circ \Psi_k$. The chain maps $\Psi \colon C_{\bullet}^{\bowtie} \to C_{\bullet}^{\Tot}$ and $\RPi \colon C_{\bullet}^{\bowtie} \to C_{\bullet}^{\Delta}$ are natural chain equivalences inducing isomorphisms in homology.
\end{prop}

\begin{proof} Let
	$a = \langle \mathbf{a}_0,\ldots,\mathbf{a}_k \rangle \in C_k^{\bowtie}(\Ee_m,\Ff_m)$. For each $0 \le i \le k$,
	\begin{align*}
		\Psi_{k-1} \circ \partial_i^{\bowtie} (a)
		&= \sum_{0 \le p < i} \langle a_0^L,\ldots,a_{p}^L ; a_p^R, \ldots, \widehat{a_i^R},\ldots, a_k^R \rangle  +\sum_{i < p \le k} \langle a_0^L, \ldots, \widehat{a_i^L}, \ldots, a_p^L ; a_0^R, \ldots, a_k^R \rangle \\
		&= \sum_{0 \le p < i} \partial^v_{i-p} \circ \Psi_{p,k-p} (a) + \sum_{i < p \le k} \partial^h_{i} \circ \Psi_{p,k-p}(a).
	\end{align*}
Consequently,
	\begin{align*}
		\Psi_{k-1} \circ d^{\bowtie} =
		 \sum_{0 \le p < i \le k} (-1)^i \partial^v_{i-p} \circ \Psi_{p,k-p}+ \sum_{0 \le i < p \le k} (-1)^i \partial^h_{i} \circ \Psi_{p,k-p}.
	\end{align*}
Using at the third equality that, $\partial_0^v \circ \Psi_{0,k} = 0$ and $\partial_k^h \circ \Psi_{k,0} = 0$, we calculate:
	\begin{align*}
		d^{\Tot} \circ \Psi_k
		&= \sum_{p=0}^k \left(\sum_{i=0}^{k-p} (-1)^{i+p} \partial_i^v \circ \Psi_{p,k-p} +
		\sum_{i=0}^p (-1)^{i} \partial_i^h \circ \Psi_{p,k-p}\right)\\
		&=  \sum_{0\le p < i \le k} (-1)^{i} \partial_{i-p}^v  \circ \Psi_{p,k-p} +
		\sum_{0 \le i < p \le k} (-1)^{i} \partial_i^h  \circ \Psi_{p,k-p}\\
		&\quad + \sum_{p=0}^k (-1)^p  \partial_0^v \circ \Psi_{p,k-p} + (-1)^p \partial_p^h \circ \Psi_{p,k-p}\\
		&=  \Psi_{k-1} \circ d^{\bowtie} +
		\sum_{t = 0}^{k-1}
		 (-1)^t (\partial_0^v \circ \Psi_{t,k-t} - \partial_{t+1}^h \circ \Psi_{t+1,k-t-1}).
	\end{align*}
So $d^{\Tot} \circ \Psi_k = \Psi_{k-1} \circ d^{\bowtie}$ because
\begin{align*}
	\partial_0^v \langle a_0^L,\ldots,a_t^L; a_t^R,\ldots, a_k^R \rangle
	&= 	 \langle a_0^L,\ldots,a_t^L; a_{t+1}^R,\ldots, a_k^R \rangle
	= \partial_{t+1}^h \langle a_0^L,\ldots,a_{t+1}^L; a_{t+1}^R,\ldots, a_k^R \rangle.
\end{align*}

Since $\RPi = \nabla \circ \Psi$ and $\nabla$ is a chain map, $d^{\Delta}_{k-1} \circ \RPi_k (x_k) = \RPi_{k-1} \circ d^{\bowtie}_k(x_k)$. So Proposition~\ref{prop:the_real_main_theorem} shows that $\RPi$ is a natural chain equivalence inverse to $\Pi$. Theorem~\ref{thm:Eilenberg--Zilber} implies that $\Psi$ is also a natural chain equivalence.
\end{proof}

To determine an explicit formula for $\RPi_k$ we combine the formula~\eqref{eq:psi} for $\Psi_k$ with the formula~\eqref{eq:eilenberg-zilber} for the Eilenberg--Zilber map $\nabla_k$. Explicit formulae for the first few $\RPi_i$ are
\begin{align*}
	\RPi_1 [d_1c_1] &= [c_1;\_] + [\_;d_1]\\
	\RPi_2 [d_1c_1,d_2c_2] &=  [c_1 \ra d_2, c_2 ;\_,\_] + [c_1,\_; \_,d_2]  - [\_,c_1;d_2,\_]  \\
	&\qquad + [\_,\_ ; d_1,c_1 \la d_2] \\
	\RPi_3 [d_1c_1,d_2c_2,d_3c_3] &= [d_1,c_1 \la d_2, ((c_1 \ra d_2)c_2) \la d_3;\_,\_,\_]
	   + [\_,\_,c_1;d_2,c_2 \la d_3,\_] \\
    &\qquad - [\_,c_1,\_;d_2,\_,c_2 \la d_3] + [c_1,\_,\_;\_,d_2,c_2 \la d_3] \\
	&\qquad + [\_,c_1 \ra d_2, c_2;d_3,\_,\_] - [c_1 \ra d_2,\_, c_2;\_,d_3,\_] \\
	&\qquad + [c_1 \ra d_2, c_2,\_;\_,\_,d_3]
	   + [\_,\_,\_;c_1 \ra (d_2(c_2 \la d_3)),c_2 \ra d_3,c_3].
\end{align*}

\begin{rmk}
	The formula~\eqref{eq:psi} for $\Psi_k$ was not initially obvious to us. We found formulae for $\RPi_k$ for $k \le 3$ using a computer-aided search predicated on formulae that involved factorisation in $\Cc \bowtie \Dd$ of the element $d_1 c_1 \cdots d_k c_k$, interspersed with objects to obtain elements of $\Cc^k * \Dd^k$. We searched for, and found, integer coefficients that solved a $\ZZ$-linear equation ensuring a chain map that inverts $\Pi_k$ on homology.
	With those in hand, we could guess, and then check, a general formula for $\RPi_k$, and then reverse-engineer a formula for $\Psi_k$.
\end{rmk}
We can also translate between categorical and total chains using the maps $\Psi$ and $\RPsi = \Pi \circ \nabla$. For low-degree terms $\RPsi \colon \bigoplus_{p+q = k} \ZZ(\Cc^p * \Dd^q) \to C_k^{\bowtie}(\Cc,\Dd)$ is given explicitly by
\begin{align*}
	\RPsi_1([c] , [d]) &= [c] + [d]\\
	\RPsi_2([c_1,c_2] , [c_3;d_1] , [d_2,d_3])
	&= [c_1,c_2]
	 + [c_3,d_1] - [c_3 \la d_1, c_3 \ra d_1] + [d_2,d_3]\\
	\RPsi_3([c_1,c_2,c_3] , [c_4,c_5;d_1] , [c_6;d_2,d_3] , [d_4,d_5,d_6])
	&= [c_1,c_2,c_3]
	 + [(c_4c_5) \la d_1, c_4 \ra (c_5 \la d_1), c_5 \la d_1]\\
	&\,-[c_4,c_5 \la d_1, c_5 \ra d_1] + [c_4,c_3,d_1] \\
	&\, + [c_6,d_2,d_3] - [c_6 \la d_2, c_6 \ra d_2,d_3]\\
	&\, + [c_6 \ra d_2, (c_6 \ra d_2) \la d_3, c_6 \ra (d_2d_3)]
	+ [d_4,d_5,d_6].
\end{align*}

\subsection{A spectral sequence and a K\"unneth Theorem} \label{subsec:consequences}
\subsubsection{A spectral sequence}
There is a spectral sequence that computes the total homology of a double complex; this and Theorem~\ref{thm:cohomologies_are_the_same} compute the homology of $\Cc \bowtie \Dd$.

For fixed $p \in \NN$, the sequence
\begin{equation}\label{eq:pth-column}
	\begin{tikzcd}[ampersand replacement=\&]
		{} \& \cdots \& {C_{p,2}} \& {C_{p,1}} \& {C_{p,0}}
		\arrow["{d^{v}_{p,2}}", from=1-2, to=1-3]
		\arrow["{d^{v}_{p,1}}", from=1-3, to=1-4]
		\arrow["{d^{v}_{p,0}}", from=1-4, to=1-5]
	\end{tikzcd}
\end{equation}
(the $p$-th column of the double complex \eqref{eq:double_complex}) is a chain complex with homology groups
\begin{equation}
	\label{eq:vertical_homology}
	H^v_{p,q}(\Cc,\Dd) \coloneqq H_q(C_{p,\bullet},d^v_{p,\bullet}).
\end{equation}
Since $d^v_{p,q} \circ d^h_{p,q+1} = - d^h_{p,q} \circ d^v_{p+1,q}$, the maps $d^h_{p,q}$ descend to homomorphisms $\widetilde{d}^{h}_{p,q} \colon H^v_{p+1,q}(\Cc,\Dd) \to H^v_{p,q} (\Cc,\Dd)$. For each $q \in \NN$, the sequence
\begin{equation}\label{eq:qth-row}
	\begin{tikzcd}[ampersand replacement=\&]
		{} \& \cdots \& {H^v_{2,q}(\Cc,\Dd)} \& {H^v_{1,q}(\Cc,\Dd)} \& {H^v_{0,q}(\Cc,\Dd)}
		\arrow["{\widetilde{d}^{h}_{2,q}}", from=1-2, to=1-3]
		\arrow["{\widetilde{d}^{h}_{1,q}}", from=1-3, to=1-4]
		\arrow["{\widetilde{d}^{h}_{0,q}}", from=1-4, to=1-5]
	\end{tikzcd}
\end{equation}
is a chain complex. We define
$H^h_pH^v_q(\Cc,\Dd)$ to be the $p$-th homology group of this complex,
\[
H^h_pH^v_q(\Cc,\Dd)
\coloneqq
H_p(H^v_{\bullet,q}(\Cc,\Dd), \widetilde{d}^h_{\bullet,q}).
\]
We define $H^v_qH^h_p(\Cc,\Dd)$ symmetrically by first considering rows of~\eqref{eq:double_complex} and then columns:
\[
H^v_qH^h_p(\Cc,\Dd)
\coloneqq
H_q(H^h_{p,\bullet}(\Cc,\Dd), \widetilde{d}^v_{p,\bullet}).
\]
\begin{cor}[{cf. \cite[\S 5.6]{Wei94}}]\label{cor:spectral sequence}
Let $C_{\bullet,\bullet}(\Cc,\Dd)$ be the matched complex of a matched pair $(\Cc,\Dd)$. Then there are homology spectral sequences $\{E_{p,q}^{hv,r}, d^{hv,r}_{p,q}\}$ and $\{E_{p,q}^{vh,r}, d^{vh,r}_{p,q}\}$ with first pages $E_{p,q}^{hv,1} = H^v_{p,q}(\Cc,\Dd)$ and $E_{p,q}^{vh,1} = H^h_{p,q}(\Cc,\Dd)$ with $d^{hv,1} = \widetilde{d}^{h}$ and $d^{vh,1} = \widetilde{d}^{v}$, and second pages
	\[
	E_{p,q}^{hv,2} =H^h_{p} H^v_{q}(\Cc,\Dd)\quad \text{and} \quad 	E_{p,q}^{vh,2} = H^v_q H^h_p (\Cc,\Dd),
	\]
	that both converge to the categorical homology of $\Cc \bowtie \Dd$.
\end{cor}

We will use these spectral sequences to compute the homology of examples in Section~\ref{sec:further_examples}.

\subsubsection{The K\"unneth Theorem for products of monoids}

Let $S$ and $R$ be monoids. Define a matched pair $(S,R)$ by $s \la r = r$ and $s \ra r = s$. The monoid $S \bowtie R$ is just $S \times R$. There is an isomorphism of double complexes $ C_\bullet(S) \ox_{\ZZ} C_\bullet(R) \cong C_{\bullet,\bullet}(S,R)$ taking $[s_1,\ldots,s_p] \ox [r_1,\ldots,r_q]$ to $[s_1,\ldots,s_p,r_1,\ldots,r_q]$.

Theorem~\ref{thm:cohomologies_are_the_same} implies that $H_\bullet(S \times R) \cong H^{\Tot}_\bullet(S, R)$. So we recover the K\"unneth formula \cite[Theorem 3.6.3]{Wei94}: an unnaturally split exact sequence
\[\begin{tikzcd}[ampersand replacement=\&,column sep=10pt]
	0 \& {\bigoplus_{p+q= n} H_p(S)\ox H_q(R)} \& {H_n(S \times R)} \& {\bigoplus_{p+q= n-1} \Tor^{\ZZ}_1(H_p(S),H_q(R))} \& 0.
	\arrow[from=1-1, to=1-2]
	\arrow[from=1-2, to=1-3]
	\arrow[from=1-3, to=1-4]
	\arrow[from=1-4, to=1-5]
\end{tikzcd}
\]

\section{Examples and homology computations}
\label{sec:further_examples}

In this section (specifically in Section~\ref{subsec:odometer_graphs}) we use our results to compute the homology of matched pairs that are like pullbacks of odometer actions over the path categories of directed graphs. The technical results we develop along the way apply to more-general systems such as Exel--Pardo self-similar systems and $k$-graphs.

We first consider, in Section~\ref{subsec:mp_path_categories}, matched pairs $(\Cc, \Dd)$ in which $\Dd = E^*$ for a directed graph $E$. We show that the vertical homology $H^v_{p,\bullet}(\Cc, E^*)$ of \eqref{eq:vertical_homology} vanishes above degree $1$.
This is unsurprising since directed graphs are $1$-dimensional; but we could not find a general theorem that applies, so we prove that $H^v_{p,q}(\Cc,E^*) = 0$ for $q \ge 2$ by direct computation.

In Section~\ref{subsec:mp_monoid_bundles} we consider matched pairs $(\Cc, \Dd)$ where $\Cc$ is a disjoint union
$\bigsqcup_{u \in \Dd^0} \Cc_u$ of monoids. We describe an isomorphism between $H^h_{\bullet, q}(\Cc, \Dd)$ and the direct sum
$\bigoplus_{u \in \Dd^0} H_\bullet(\Cc_u; \ZZ u\Dd^q)$ of the homology of the monoids $\Cc_v$ with
coefficients in $\ZZ u\Dd^q$. The Universal Coefficient Theorem gives a short exact sequence that computes $H^h_{\bullet, q}(\Cc, \Dd)$
as an extension of an appropriate $\Tor$-group by $\bigoplus_{u\in \Dd^0} H_p(\Cc_u) \otimes_{\ZZ \Cc_u} \ZZ u\Dd^q$.

In Section~\ref{subsec:mp_integer_bundles} we restrict further to $\Cc_u \cong \ZZ$. We deduce that $H^h_{\bullet, q}(\ZZ \times \Dd^0, \Dd)$ vanishes in degree~2 or more, and compute $H^h_{0, q}(\ZZ \times \Dd^0, \Dd)$ and $H^h_{1, q}(\ZZ \times \Dd^0, \Dd)$ in terms of the groups of invariants and coinvariants of $\ZZ u\Dd^q$.

Finally, in Section~\ref{subsec:odometer_graphs} we compute $H^{\bowtie}_{\bullet}(\ZZ \times F^0, F^*)$ for matched pairs consisting of a bundle of copies of $\ZZ$ acting like odometers on the path category of a directed graph $F$. We show that in the second spectral sequence of Corollary~\ref{cor:spectral sequence}, only $E^{vh,2}_{0,0}$, $E^{vh,2}_{1,0}$, $E^{vh,2}_{0,1}$, and $E^{vh,2}_{1,1}$ can be nonzero. So the sequence converges on its second page, yielding an explicit formula for $H_{\bullet}^{\bowtie}(\ZZ \times F^0, F^*)$ in terms of a weighted incidence matrix in $M_{F^0, F^1}(\ZZ)$ (Proposition~\ref{prp:E2 for example}).

\subsection{Matched pairs involving path categories of directed graphs}
\label{subsec:mp_path_categories}

Let $E$ be a directed graph and suppose that $(\Cc, E^*)$ is a matched pair. We say that $(\Cc, E^*)$ is a \emph{length-preserving} matched pair if $|c \la \alpha| = |\alpha|$ for all $(c,\alpha) \in \Cc * E^*$.

We write $N_{p,q}$ for the subgroup of $C_{p,q}$ generated
by the nondegenerate vertical chains: chains $[c; d_1,\ldots,d_q]$ such that
$d_i \notin \Dd^0$ for all $1 \le i \le q$. By
\cite[Theorem~8.3.8]{Wei94}, the group $H_{p,q}^v(\Cc, E^*)$ of \eqref{eq:vertical_homology} is isomorphic to the $q$-th homology group of $(N_{p,\bullet}, d^v_{p,\bullet}|_{N_{p,\bullet}})$.

For $\alpha \in E^*$ we write $\alpha^i$ for the $i$-th edge of
$\alpha$. So $\alpha = \alpha^1 \alpha^2\cdots \alpha^{|\alpha|}$. For $0 \le
i \le j \le |\alpha|$, we define $\alpha^{[i,j]} \in E^{j-i}$ by $\alpha = \alpha' \alpha^{[i,j]} \alpha''$ for some $\alpha' \in E^i$ and
$\alpha'' \in E^{|\alpha|-j}$. For example, $\alpha^{[i-1, i]} = \alpha^{i}$ for $1 \le i \le
|\alpha|$, and $\alpha = \alpha^{[0,i-1]}\alpha^i\alpha^{[i, |\alpha|]}$ for each
$i$.

\begin{prop}\label{prp:cat-graph homology}
Let $\Cc$ be a small category and let $E$ be a directed graph, and suppose that $(\Cc,E^*)$ is a length-preserving matched pair. Taking the convention that the empty sum is zero, for $p \ge 0$ and $q \ge 1$, define $s^v_{p,q} \colon N_{p,q} \to N_{p,q+1}$ by
	\[
	s^v_{p,q}[c\,;\alpha] = -\sum_{i=1}^{|\alpha_1|-1} [c \ra \alpha^{[0,i-1]}_1; \alpha^i_1,\alpha^{[i,|\alpha_1|]}_1,\alpha_2, \ldots,\alpha_q]
	\]
	for $c \in \Cc^p$ and $\alpha = (\alpha_1,\ldots,\alpha_q) \in (E^*)^q$ with $s(c) = r(\alpha_1)$. Then for $q \ge 2$,
    \begin{equation}\label{eq:ds+sd}
        d^v_{p,q} s^v_{p,q} + s^v_{p,q-1} d^v_{p,q-1} = \id_{N_{p,q}}.
    \end{equation}
    In particular, for $q \ge 2$ we have $H^v_{p,q}(\Cc,E^*) = 0$, and using the spectral sequence of Corollary~\ref{cor:spectral sequence} there is a long exact sequence
	\[\begin{tikzcd}[column sep = 7pt,ampersand replacement=\&]
		\cdots \& {H_{p+1}^{\bowtie}(\Cc,E^*)} \& {E^{hv,2}_{p+1,0}} \& {E_{p-1,1}^{hv,2}} \& {H_{p}^{\bowtie}(\Cc,E^*)} \& {E_{p,0}^{hv,2}} \& {E_{p-2,1}^{hv,2}} \& {H_{p-1}^{\bowtie}(\Cc,E^*)} \& \cdots.
		\arrow[from=1-1, to=1-2]
		\arrow[from=1-2, to=1-3]
		\arrow["d",from=1-3, to=1-4]
		\arrow[from=1-4, to=1-5]
		\arrow[from=1-5, to=1-6]
		\arrow["d",from=1-6, to=1-7]
		\arrow[from=1-7, to=1-8]
		\arrow[from=1-8, to=1-9]
	\end{tikzcd}\]
\end{prop}

\begin{proof}
Fix $q \ge 2$. We write $d \coloneqq \bigoplus_p d^v_{p,q}$, $s \coloneqq
\bigoplus_p s^v_{p,q}$, and $N_q \coloneqq \bigoplus_p N_{p,q}$. The restrictions $\partial_j \coloneqq \bigoplus_{p} \partial^{v,j}_{p,q} \colon N_q \to
N_{q-1}$ of the vertical face maps~\eqref{eq:vertical_face_maps} satisfy
\[
\partial_j([c;\alpha_1, \dots, \alpha_q]) =
    \begin{cases}
    [c\ra \alpha_1; \alpha_2, \dots, \alpha_q] &\text{ if $j=0$}\\
    [c; \alpha_1, \dots, \alpha_j\alpha_{j+1},\dots, \alpha_q] &\text{ if $1\le j \le q-1$}\\
    [c; \alpha_1, \dots, \alpha_{q-1}] &\text{ if $j = q$},
    \end{cases}
\]
and $d_q = \sum^q_{j=0} (-1)^j \partial_j$.

We first claim that for $j \ge 3$ we have $\partial_j \circ s = s \circ \partial_{j-1}$. Indeed,
\begin{align*}
\partial_j(s([c; \alpha_1, \alpha_2, \dots, \alpha_q]))
    &= -\sum_{i=1}^{|\alpha_1|-1} \partial_j([c \ra \alpha^{[0,i-1]}_1; \alpha^i_1,\alpha^{[i,|\alpha_1|]}_1,\alpha_2, \ldots,\alpha_q])\\
    &= -\sum_{i=1}^{|\alpha_1|-1} [c \ra \alpha^{[0,i-1]}_1; \alpha^i_1,\alpha^{[i,|\alpha_1|]}_1, \partial_{j-2}(\alpha_2, \ldots, \alpha_q)]\\
    &= s([c;\alpha_1, \partial_{j-2}(\alpha_2, \dots, \alpha_q)])\\
    &= s(\partial_{j-1}([c; \alpha_1, \alpha_2, \dots, \alpha_q])).
\end{align*}
Using this at the final equality, we obtain
\begin{align*}
d \circ s + s \circ d
    &= \sum_{i=0}^{q+1} (-1)^i \partial_i \circ s + \sum_{j=0}^q (-1)^j s \circ \partial_j \\
    &= \partial_0 \circ s - \partial_1 \circ s + \partial_2 \circ s + \Big(\sum^{q+1}_{i=3} (-1)^i \partial_i \circ s\Big)
     + s \circ \partial_0  - s \circ \partial_1  + \Big(\sum_{j=2}^q (-1)^j s \circ \partial_j\Big)\\
    &= \partial_0 \circ s - \partial_1 \circ s + \partial_2 \circ s + s \circ \partial_0  - s \circ \partial_1
     + \sum^{q+1}_{j=3} \Big((-1)^j \partial_j \circ s - (-1)^{j} s \circ \partial_{j-1}\Big)\\
    &= \partial_0 \circ s - \partial_1 \circ s + \partial_2 \circ s + s \circ \partial_0  - s \circ \partial_1.
\end{align*}
So it suffices to show that
\begin{equation}\label{eq:toshow}
\partial_0 \circ s - \partial_1 \circ s + \partial_2 \circ s + s \circ \partial_0  - s \circ \partial_1
    = \id_{N_q}.
\end{equation}
For this, fix
$x \coloneqq [c;\alpha_1, \dots, \alpha_q] \in N_q$. Let $l \coloneqq |\alpha_1|$. We claim that
\begin{equation}\label{eq:p0s-p1s}
	(\partial_0 \circ s - \partial_1 \circ s)(x)
	= x - [c\ra \alpha_1^{[0,l-1]}; \alpha_1^l,\alpha_2, \dots, \alpha_q].
\end{equation}
To see this, we compute:
\begin{align*}
	\partial_0 \circ s(x)
	&= \partial_0\big(-\sum^{l-1}_{i=1} [c\ra \alpha_1^{[0,i-1]}; \alpha_1^i, \alpha_1^{[i,l]}, \alpha_2, \dots, \alpha_q]\Big)
	= -\sum^{l-1}_{i=1} [c\ra \alpha_1^{[0,i]}; \alpha_1^{[i,l]}, \alpha_2, \dots, \alpha_q], \text{ and}\\
	\partial_1 \circ s(x)
	&= \partial_1\Big(-\sum^{l-1}_{i=1} [c\ra \alpha_1^{[0,i-1]}; \alpha_1^i, \alpha_1^{[i,l]}, \alpha_2, \dots, \alpha_q]\Big)
	= -\sum^{l-1}_{i=1} [c\ra \alpha_1^{[0,i-1]}; \alpha_1^{[i-1,l]}, \alpha_2, \dots, \alpha_q].
\end{align*}
So
\[
(\partial_0 \circ s - \partial_1 \circ s)(x)
= \sum^{l-1}_{i=1}\big(-[c\ra \alpha_1^{[0,i]}; \alpha_1^{[i,l]}, \alpha_2, \dots, \alpha_q] + [c\ra \alpha_1^{[0,i-1]}; \alpha_1^{[i-1,l]}, \alpha_2, \dots, \alpha_q]\big)
\]
telescopes to~\eqref{eq:p0s-p1s}.
Next, we claim that
\begin{equation}\label{eq:sp0-sp1}
	(s \circ \partial_0 - s \circ \partial_1)(x)
	= \Big(\sum^{l-1}_{i=1} [c\ra\alpha_1^{[0,i-1]}; \alpha_1^i; \alpha_1^{[i,l]}\alpha_2, \alpha_3, \dots,\alpha_q]\Big) + [c \ra \alpha_1^{[0,l-1]}; \alpha_1^l, \alpha_2, \alpha_3, \dots, \alpha_q].
\end{equation}
Let $m \coloneqq |\alpha_2|$. Again, we compute:
\begin{equation} \label{eq:s_d0}
s \circ \partial_0(x)
= s([c\ra\alpha_1;\alpha_2, \dots, \alpha_q])
= -\sum^{m-1}_{i=1} [c\ra \alpha_2^{[0,i-1]}; \alpha_2^i,\alpha_2^{[i,m]},\alpha_3, \dots, \alpha_q],
\end{equation}
and
\begin{align*}
	s \circ \partial_1(x)
	&= s([c; \alpha_1\alpha_2, \dots, \alpha_q])\\
	&= -\sum^{l+m-1}_{i=1} [c\ra(\alpha_1\alpha_2)^{[0,i-1]}; (\alpha_1\alpha_2)^i, (\alpha_1\alpha_2)^{[i,l+m]}, \alpha_3, \dots, \alpha_q]\\
	&= -\Big(\sum^{l-1}_{i=1} [c\ra\alpha_1^{[0,i-1]}; \alpha_1^i, \alpha_1^{[i,l]}\alpha_2, \alpha_3, \dots, \alpha_q]\Big)
	- [c\ra\alpha_1^{[0,l-1]}; \alpha_1^l, \alpha_2, \dots,\alpha_q] \\
	&\qquad - \Big(\sum^{m-1}_{i=1} [c\ra(\alpha_1\alpha_2^{[0,i-1]}); \alpha_2^i, \alpha_2^{[i,m]}, \alpha_3, \dots, \alpha_q]\Big).
\end{align*}
Subtracting~\eqref{eq:s_d0} from this equation yields~\eqref{eq:sp0-sp1}.
Now we add Equations \eqref{eq:sp0-sp1}~and~\eqref{eq:p0s-p1s}, and the terms $[c \ra \alpha_1^{[0,l-1]};\alpha_1^l, \alpha_2,
\alpha_3, \dots, \alpha_q]$ cancel, giving
\begin{equation}\label{eq:allbut+p2s}
	(\partial_0 \circ s - \partial_1 \circ s + s \circ \partial_0  - s \circ \partial_1)(x)
	= x + \Big(\sum^{l-1}_{i=1} [c\ra\alpha_1^{[0,i-1]};  \alpha_1^i, \alpha_1^{[i,l]}\alpha_2, \alpha_3, \dots,\alpha_q]\Big).
\end{equation}

Finally, we compute
\[
\partial_2 \circ s(x)
= \partial_2\Big(-\sum^{l-1}_{i=1} [c\ra \alpha_1^{[0,i-1]}; \alpha_1^i, \alpha_1^{[i,l]}, \alpha_2, \dots, \alpha_q]\Big)
= -\sum^{l-1}_{i=1} [c\ra \alpha_1^{[0,i-1]};\alpha_1^i, \alpha_1^{[i,l]} \alpha_2, \dots, \alpha_q].
\]
Adding this to~\eqref{eq:allbut+p2s} gives
$
(\partial_0 \circ s - \partial_1 \circ s + \partial_2\circ s + s \circ \partial_0  - s \circ \partial_1)(x)
= x,
$
which gives~\eqref{eq:toshow}.

Now fix $a \in \ker(d_{q-1}) \cap N_q$. Then~\eqref{eq:ds+sd} gives
\[
a = d_q(s_q(a)) + s_{q-1}(d_{q-1}(a)) = d_q(s_q(a)) \in \ran(d_q),
\]
so $H_q(N_{p,\bullet}) = 0$. Theorem~8.3.8 of \cite{Wei94} gives $H_{p,q}^v(\Cc,E^*) = H_q(N_{p,\bullet}) = 0$.
The long exact sequence exists by definition of convergence of a spectral sequence~\cite[p.124]{Wei94}.
\end{proof}

\begin{rmk}
The proof of the preceding lemma relies on treating the empty sum as zero, prompting a quick reality check of the edge-case where $[c; \alpha_1, \dots, \alpha_q] \in N_q$ with $\alpha_1 = e \in E^1$.
Let $x = [c; e, \alpha_2, \dots, \alpha_q]$. Since $s^v_{p,q}(x) =0$, \eqref{eq:ds+sd} for $x$ collapses to $s^v_{p,q-1} \circ d^v_{p,q-1}(x) = x$, so the cancellation that led to~\eqref{eq:toshow} appears to fall down.

But all is well: for $j \ge 3$ we have $\partial_j([c; e, \alpha_2, \dots, \alpha_q]) = [c; e, \kappa]$ for some $\kappa \in (E^*)^{q-1}$, and so $s^v_{p,q-1}(\partial_j(x)) = s^v_{p,q-1}([c; e, \kappa]) = 0$. So $s^v_{p,q-1} \circ d^v_{p,q-1}(x) = s^v_{p,q-1}([c \ra e; \alpha_1, \dots, \alpha_q]) - s^v_{p,q-1}([c; e\alpha_1, \dots, \alpha_q])$ and all the terms in the resulting sums cancel except the first term $-(-[c;e, \alpha_1, \dots, \alpha_q]) = x$ of $s^v_{p,q-1}([c; e\alpha_1, \dots, \alpha_q])$.
\end{rmk}

Recall that $\{E_{p,q}^{vh,r}, d_{p,q}^{vh,r}\}$ is the spectral sequence of Corollary~\ref{cor:spectral sequence}, which in our current setup satisfies $E_{p,q}^{vh,2} = H^v_{p} H^h_{q}(\Cc,E^*)$.
\begin{lem}\label{lem:cohomology_graph_rows_vanish}
	Let $\Cc$ be a small category and let $E$ be a directed graph, and suppose that $(\Cc,E^*)$ is a length-preserving matched pair.
	For $p \ge 0$ and $q \ge 2$, we have
	$ s^v_{p,q} \circ d^h_{p,q} =  d^h_{p,q+1} \circ s^v_{p+1,q}$. We have $E^{vh,2}_{p,q} = 0$ for all $q \ge 2$, $H_0^{\bowtie}(\Cc, E^*) \cong E^{vh,2}_{0,0}$, and for each $n \ge 1$ there is a short exact sequence
	\[\begin{tikzcd}[ampersand replacement=\&, column sep=15pt]
		0 \& {E^{vh,2}_{0,n}} \& {H_n^{\bowtie}(\Cc, E^*)} \& {E^{vh,2}_{1,n-1}} \& 0.
		\arrow[from=1-1, to=1-2]
		\arrow[from=1-2, to=1-3]
		\arrow[from=1-3, to=1-4]
		\arrow[from=1-4, to=1-5]
	\end{tikzcd}\]
\end{lem}
\begin{proof}
	Fix $(c;\alpha) \in \Cc^{p+1} * ({E^*})^q$. We compute
	\begin{align*}
		d^h_{p,q+1} \circ s^v_{p+1,q}  [c;\alpha]
		&=\sum_{i=1}^{|\alpha_1|-1} \sum_{k=0}^{p+1} (-1)^k  \partial^{h,k}_{p,q+1} [c \ra \alpha_1^{[0,i-1]};\alpha_1^i,\alpha_1^{[i,|\alpha_1|]}, \alpha_2,\ldots,\alpha_q].
	\end{align*}
	Analogously to Lemma~\ref{lem:actions faces commute}, we observe that for $0 \le k \le p$,
	\begin{align*}
		\partial^{h,k}_{p,q+1} [c \ra \alpha_1^{[0,i-1]};\alpha_1^i,\alpha_1^{[i,|\alpha_1|]}, \alpha_2,\ldots,\alpha_q]
		&= [\partial_{p,0}^{h,k}[c \ra \alpha_1^{[0,i-1]}];\alpha_1^i,\alpha_1^{[i,|\alpha_1|]}, \alpha_2,\ldots,\alpha_q]\\
		&= [\partial_{p,0}^{h,k}[c] \ra \alpha_1^{[0,i-1]};\alpha_1^i,\alpha_1^{[i,|\alpha_1|]}, \alpha_2,\ldots,\alpha_q].
	\end{align*}
	Write $(c_p \la \alpha)_r$ for the $r$-th entry of $c \la \alpha \in (E^*)^q$. For $k = p+1$ we have
	\begin{align*}
		&\partial^{h,p+1}_{p,q+1} [c \ra \alpha_1^{[1,i-1]};\alpha_1^i,\alpha_1^{[i,|\alpha_1|]}, \alpha_2,\ldots,\alpha_q]\\
		&\quad =[\partial_{p,0}^{h,p+1}[c \ra \alpha_1^{[0,i-1]}];(c_{p} \ra \alpha_1^{[0,i-1]}) \la (\alpha_1^i,\alpha_1^{[i,|\alpha_1|]}, \alpha_2,\ldots,\alpha_q) ]\\
		&\quad =[\partial_{p,0}^{h,p+1}[c] \ra \alpha_1^{[1,i-1]};(c_p \la \alpha)_1^i, (c_p \la \alpha)_1^{[i,|\alpha_1|]}, (c_p \la \alpha)_2,\ldots,(c_p \la \alpha)_q].
	\end{align*}
	Hence,
	\begin{align*}
			d^h_{p,q+1} \circ s^v_{p+1,q}   [c;\alpha]
		= \sum_{k=0}^{p+1} (-1)^k s^v_{p,q}  \circ \partial_{p,q}^{h,k}[c;\alpha]
		= s^v_{p,q} \circ d^h_{p,q} [c;\alpha].
	\end{align*}
	Thus, the $s^v_{p,q}$ descend to sections of the differentials between the $E^{vh,1}_{p,q}$, and the resulting homology groups $E^{vh,2}_{p,q}$ vanish for $q \ge 2$.  So the spectral sequence stabilises on its second page, and short exact sequences follow from the associated filtration.
\end{proof}

\subsection{Matched pairs involving bundles of monoids}
\label{subsec:mp_monoid_bundles}

In this section $(\Cc,\Dd)$ is a matched pair in which $\Cc = \bigsqcup_{u \in \Dd^0} \Cc_u$ is a bundle of monoids over $\Cc^0 = \Dd^0$.
For each $u \in \Dd^0$ and $q \ge 0$, the free $\ZZ$-module $\ZZ u \Dd^q$ generated by $u \Dd^q = \{d \in \Dd^q \mid r(d) = u\}$ is a left $\Cc_u$-module under the action $c \cdot \big(\sum_{d \in \Dd^q} n_d d\big) = \sum_{d \in \Dd^q} n_d (c \la d)$.

For a monoid $S$, the categories of left (respectively right) $S$-modules and left (respectively right) $\ZZ [S]$-modules are equivalent. Given a left $S$-module $M$ we can compute the homology $H_\bullet(S;M)$ of $S$ with coefficients in $M$ as follows: fix a projective resolution $\cdots \to P_1 \to P_0 \to \ZZ \to 0$ of the trivial $S$-module $\ZZ$. Then $H_\bullet(S;M)$ is the homology of the chain complex $ \cdots   \to P_1 \ox_{\ZZ [S]} M \to P_0 \ox_{\ZZ [S]} M \to 0$: that is, $H_n(S;M) \coloneqq \Tor_n^{\ZZ [S]}(\ZZ;M)$.

By taking the \emph{bar resolution} of the trivial (right) $S$-module $\ZZ$ we arrive at a more familiar description. For $n \ge 1$ let $B_n$ be the free $S$-module generated by $\{ [ s_1,\ldots,s_n] \mid s_i \in S\}$. Let $B_0$ be the free $S$-module generated by the symbol $[\phantom{\cdot}]$. Define $b_{n} \colon B_{n+1} \to B_{n}$ by
\begin{equation}\label{eq:bar boundary}
b_n [s_0,\ldots,s_n] = [s_1, \ldots, s_n] + \sum_{i=1}^{n}(-1)^i [s_0,\ldots,s_{i-1}s_{i}, \ldots, s_n] + (-1)^{n+1} [s_1,\ldots,s_{n-1}] s_n
\end{equation}
for $n \ge 1$ and $b_{-1} \colon B_0 \to \ZZ$ by $b_{-1}[\phantom{\cdot}] = 1$; note that $b_0[s_0] = [\phantom{\cdot}]-[\phantom{\cdot}]s_0$. The group of $p$-chains with values in $M$ is
\[
C_p(S;M) \coloneqq B_p \ox_{\ZZ [S]} M.
\]
The boundary maps $d_p \coloneqq b_p \ox \id_M \colon C_{p+1}(S;M) \to C_{p}(S;M)$ for $p \ge 1$ satisfy
\begin{align*}
	d_p([s_0,\ldots, s_{p}] \ox m)
	&= [s_1,\ldots,s_{p}] \ox m + \sum_{i=1}^{p} (-1)^i [s_0,\ldots,s_{i-1}s_i, \ldots, s_p] \ox m \\
	&\qquad + (-1)^{p+1}[s_0,\ldots,s_{p-1}] \ox s_{p} \cdot m,
\end{align*}
and $d_{0} \colon C_1(S;M) \to C_0(S;M)$ satisfies $d_{0}([s_0] \ox m) = [\phantom{\cdot}] \ox m - [\phantom{\cdot}] \ox s_0 \cdot m$.
Then $H_{n}(S;M) \cong \ker(d_{n-1})/\im(d_{n})$. Taking $M = \ZZ$, the trivial $S$-module, recovers Definition~\ref{dfn:categorical_homology} if $\Cc = S$.

Given a matched pair $(\Cc, \Dd)$ where $\Cc$ is a bundle of monoids, we can compute the horizontal homology of the matched complex using the homology of the monoids $\Cc_u$.
\begin{prop}\label{prop:bundle_homology_sum}
	Let $\Dd$ be a small category and let $\Cc = \bigsqcup_{u \in \Dd^0} \Cc_u$ be a bundle of monoids over $\Dd^0$ such that $(\Cc,\Dd)$ is a matched pair. Then $C_{p}(\Cc_u; \ZZ u \Dd^q)
 \ni [c_0,\ldots,c_p] \ox d   \mapsto [c_0,\ldots,c_{p-1}; c_p \la d] \in C_{p,q}$ induces an isomorphism
	\[
	H^h_{p,q}(\Cc,\Dd) \cong \bigoplus_{u \in \Dd^0} H_p(\Cc_u; \ZZ u \Dd^q),
	\]
	and there is a short exact sequence
	\[\begin{tikzcd}[ampersand replacement=\&, column sep = 10pt]
		0 \& {\bigoplus_{u \in \Dd^0} H_p(\Cc_u)\ox_{\ZZ [\Cc_u]} \ZZ u \Dd^q} \& {H^h_{p,q}(\Cc,\Dd)} \& {\bigoplus_{u \in \Dd^0} \Tor^{\ZZ[\Cc_u]}_1 (H_{p-1}(\Cc_u), \ZZ u \Dd^q)} \& 0.
		\arrow[from=1-1, to=1-2]
		\arrow[from=1-2, to=1-3]
		\arrow[from=1-3, to=1-4]
		\arrow[from=1-4, to=1-5]
	\end{tikzcd}\]
\end{prop}

\begin{proof}
	In $C_p(\Cc_u;\ZZ u \Dd^q) = B_p \ox_{\ZZ[\Cc_u]} \ZZ u \Dd^q$ we have $[c_1,\ldots,c_p] \ox d = [c_1,\ldots, c_{p-1},s(c_{p-1})] \ox c_p \cdot d$.
	Consequently, the map $\pi_{p,q} \colon \bigoplus_{u \in \Dd^0} C_p(\Cc_u;\ZZ u \Dd^q) \to C_{p,q}$ defined by
	\[
	\pi_{p,q} ([c_0,\ldots,c_p] \ox d) = [c_0,\ldots,c_{p-1}; c_p \la d]
	\]
	for $[c_0,\ldots,c_p] \ox d  \in C_p(\Cc_u;\ZZ u \Dd^q)$ and $u \in \Dd^0$, is an isomorphism. Moreover,
	\begin{align*}
		d^h_{p,q} \circ \pi_{p+1,q} ([c_0,\ldots,c_{p+1}] \ox d)
		&= d^{h}_{p,q} [c_1,\ldots,c_p, c_{p+1} \la d] \\
		&= [c_1,\ldots,c_p;c_{p+1} \la d] + \sum_{i=1}^p (-1)^i  [c_0,\ldots,c_{i-1}c_i, \ldots, c_p; c_{p+1} \la d]\\
		& \qquad + (-1)^{p+1} [c_0,\ldots,c_{p-1}, c_{p}c_{p+1} \la d]\\
		&= \pi_{p,q} \circ d^u_{p,q} ([c_0,\ldots,c_{p+1}] \ox d).
	\end{align*}
	So $H^h_{p,q}(\Cc,\Dd) \cong \bigoplus_{u \in \Dd^0} H_p(\Cc_u; \ZZ u \Dd^q)$.
	
	Since $\ZZ$ is a free {$\Cc_u$-module} the short exact sequence follows  from~\cite[Theorem 3.6.1]{Wei94}.
\end{proof}

\subsection{Matched pairs involving integer bundles}
\label{subsec:mp_integer_bundles}

In this section we consider matched pairs $(\Cc,\Dd)$ where $\Cc \cong \Dd^0 \times \ZZ$. If $M$ is a $\ZZ$-module, then $M^{\ZZ} \coloneqq \{m \in M \mid k \cdot m = m \text{ for all } k \in \ZZ\}$ is its submodule of \emph{invariants} and $M_{\ZZ} \coloneqq M / \langle k \cdot m - m \mid k \in \ZZ ,\, m \in M \rangle$ is its module of \emph{coinvariants}.

If $X$ is a set and $\la \colon \ZZ \times X \to X$ is a left action, we write $\ZZ\bs X$ for the set of $\ZZ$-orbits in $X$. We denote the orbit of $x \in X$ by $\lequiv x \requiv \in \ZZ \bs X$, and the set of periodic points of $X$ by
\[\Per(X) \coloneqq \{x \in X \mid \text{there exists } k \in \ZZ \setminus \{0\} \text{ such that } k \la x = x\}.
\]
Then $\ZZ\bs{\Per(X)} \subseteq \ZZ \bs X$ is the set of finite orbits in $X$.

Each left action $\la \colon \ZZ \times X \to X$ induces a corresponding left action $\cdot \colon \ZZ  \times \ZZ X \to \ZZ X$.

\begin{lem}\label{lem:Z_action_invariants}
	Let $X$ be a set, and $\la \colon \ZZ \times X \to X$ a left action. There are isomorphisms
	\[
	\phi_0\colon \ZZ (\ZZ \bs X) \to (\ZZ X)_{\ZZ}
	\quad \text{and} \quad
	\phi_1\colon \ZZ (\ZZ \bs\!\Per(X)) \to (\ZZ X)^{\ZZ}
	\]
	such that $\phi_0(\lequiv d\requiv) = d + \langle k \cdot m - m \mid k \in \ZZ ,\, m \in \ZZ X \rangle$ and $\phi_1 (\lequiv d \requiv) = \sum_{d' \in \lequiv d \requiv} d'$.
\end{lem}
\begin{proof}
	For the first isomorphism, regard $\ZZ X$ and $\ZZ (\ZZ \bs X)$ as the sets of finitely supported $\ZZ$-valued functions on $X$ and $\ZZ \bs X$ respectively. Define $\pi\colon \ZZ X \to \ZZ (\ZZ \bs X)$ by $\pi(f)(\lequiv x \requiv) = \sum_{y \in \lequiv x \requiv} f(y)$. Then $\ker(\pi) = \langle k \cdot m - m \mid k \in \ZZ ,\, m \in \ZZ X \rangle$, so $\pi$ descends to an isomorphism $(\ZZ X)_{\ZZ } \cong \ZZ (\ZZ \bs X)$ whose inverse is the desired map $\phi_0$.

	For the second isomorphism, fix $m \coloneqq \sum_{x \in X} a_x x \in \ZZ X$, where each $a_x \in \ZZ$. Then $\sum_{x \in X} a_x x = \sum_{ \lequiv y \requiv\in \ZZ \bs X} \sum_{x \in \lequiv y \requiv} a_x x$. For each $k \in \ZZ$,
	\[
	k \cdot m - m = \sum_{\lequiv y \requiv \in \ZZ \bs X} \Big(\sum_{x \in \lequiv y \requiv} a_x (k \la x) - a_x x\Big).
	\]
	Hence, $k \cdot m - m = 0$ for all $k \in \ZZ$ if and only if $a_x = a_y$ whenever $\lequiv x \requiv= \lequiv y \requiv$; that is,  $a \colon x \mapsto a_x$ is constant on orbits. Since $a$ is finitely supported, if $m \in (\ZZ X)^\ZZ$ then $a$ is nonzero only on finite orbits. Hence, the formula for $\phi_1$ determines an isomorphism.
\end{proof}

\begin{prop}\label{prop:integer_bundle_h_homology}
	Let $(\Cc, \Dd)$ be a matched pair with $\Cc = \Dd^{0} \times \ZZ$. There are isomorphisms $\alpha_0 \colon \bigoplus_{u \in \Dd^0}  \ZZ (\Cc_u \bs u \Dd^q) \to H_{0,q}^h(\Cc,\Dd)$ and $\alpha_1 \colon  \bigoplus_{u \in \Dd^0}  \ZZ (\Cc_u \bs \Per( u \Dd^q)) \to H_{1,q}^h(\Cc,\Dd)$ satisfying
	\begin{equation}\label{eq:explicit_isomorphisms}
		\alpha_0(\lequiv d \requiv) = [r(d);d] + \im(d^h_{0,q})
		\quad \text{and} \quad
		\alpha_1(\lequiv d \requiv) = \sum_{d' \in \lequiv d \requiv} [1;d'] + \im(d_{1,q}^h).
	\end{equation} Moreover,
	\begin{equation}\label{eq:horizontal_homology_integer_bundle}
			H_{p,q}^h(\Cc,\Dd)
		\cong \begin{cases}
			\bigoplus_{u \in \Dd^0} (\ZZ u \Dd^q)_{\ZZ}
			&\text{if } p = 0\\
			\bigoplus_{u \in \Dd^0} (\ZZ u \Dd^q)^{\ZZ}
			&\text{if } p=1\\
			0
			& \text{otherwise}
		\end{cases}
		\cong \begin{cases}
			\bigoplus_{u \in \Dd^0}  \ZZ (\Cc_u \bs u \Dd^q)
			&\text{if } p = 0\\
			\bigoplus_{u \in \Dd^0} \ZZ (\Cc_u \bs \Per( u \Dd^q))
			&\text{if } p=1\\
			0
			& \text{otherwise.}
		\end{cases}
	\end{equation}
\end{prop}
\begin{proof}
	We identify the group ring $\ZZ[\ZZ]$ with the ring of Laurent polynomials $\ZZ[t,t^{-1}]$. Let $\Sigma \colon \ZZ[t,t^{-1}] \to \ZZ$ be evaluation at~1, that is $\Sigma$ sums the coefficients of a polynomial.

	 As in~\cite[Example 6.1.4]{Wei94},
	\begin{equation}\label{eq:laurent_resolution}
	\cdots \longrightarrow 0 \longrightarrow\ZZ[t,t^{-1}] \overset{\times (t - 1)}{\longrightarrow} \ZZ[t,t^{-1}] \overset{\Sigma}{\longrightarrow} \ZZ \longrightarrow 0,
	\end{equation}
	is a projective resolution of $\ZZ$ by $\ZZ[\ZZ]$-modules.
	Since any projective resolution computes the homology of a group, it follows that $H_0(\ZZ; \ZZ u \Dd^q) \cong (\ZZ u \Dd^q)_{\ZZ}$ and $H_1(\ZZ; \ZZ u \Dd^q) \cong  (\ZZ u \Dd^q)^{\ZZ}$, and that $H_n(\ZZ; \ZZ u \Dd^q) = 0$ for $n \ge 2$.	
	  The first isomorphism of \eqref{eq:horizontal_homology_integer_bundle} follows from Proposition~\ref{prop:bundle_homology_sum}; the second follows from Lemma~\ref{lem:Z_action_invariants}.
	
	  Let $N_q \coloneqq \langle k \cdot m - m \mid k \in \ZZ ,\, m \in \ZZ u \Dd^q \rangle$.
	  To establish \eqref{eq:explicit_isomorphisms} we describe the chain map connecting the bar resolution~\eqref{eq:bar boundary} with the resolution~\eqref{eq:laurent_resolution}. Let $\delta_n$ be the generator of $\ZZ[\ZZ]$ corresponding to $n \in \ZZ$, and let $[n]$ be the basis element of the $\ZZ[\ZZ]$-module $B_1$ (by definition, $B_1$ is the free $\ZZ[\ZZ]$-module over $\ZZ$, but who wants to write $\ZZ[\ZZ][\ZZ]$?) Then the diagram
		\begin{equation}\label{eq:laurent_walks_into_a_bar}
		\begin{tikzcd}
		\cdots & 0 & {\ZZ[t,t^{-1}]} & {\ZZ[t,t^{-1}]} & \ZZ & 0 \\
		\cdots & {B_2} & {B_1} & {B_0} & \ZZ & 0
		\arrow[from=1-2, to=1-3]
		\arrow["{\times (t-1)}", from=1-3, to=1-4]
		\arrow["\Sigma", from=1-4, to=1-5]
		\arrow["{b_1}", from=2-2, to=2-3]
		\arrow["{b_0}", from=2-3, to=2-4]
		\arrow["{b_{-1}}", from=2-4, to=2-5]
		\arrow[from=1-2, to=2-2]
		\arrow["{t^n \mapsto[1]\cdot\delta_n}", from=1-3, to=2-3]
		\arrow["{t^n \mapsto [\phantom{\cdot}]\cdot\delta_n}", from=1-4, to=2-4]
		\arrow["\id", from=1-5, to=2-5]
		\arrow[from=1-5, to=1-6]
		\arrow[from=2-5, to=2-6]
		\arrow[from=1-1, to=1-2]
		\arrow[from=2-1, to=2-2]
		\end{tikzcd}
		\end{equation}
	   commutes. The homology $H_\bullet(\ZZ; \ZZ u\Dd^q)$ as computed by each of these resolutions is obtained by tensoring by $\ZZ u\Dd^q$ on the right, replacing $\Sigma \otimes 1$ and $b_{-1} \otimes 1$ with $0$, and taking homology.
	
	   Hence, for each $u \in \Dd^0$, the vertical map $t^n \mapsto [\phantom{\cdot}]\cdot \delta_n$ in~\eqref{eq:laurent_walks_into_a_bar} induces an isomorphism $\coker(\times (t - 1) \otimes \id_{\ZZ u\Dd^q}) \to H_0(\ZZ, \ZZ u \Dd^q)$ taking $t^0 \otimes d + \im(\times (t - 1) \otimes \id_{\ZZ u\Dd^q})$ to $d + \im(b_0)$. The isomorphism $(\ZZ u\Dd^q)_\ZZ \to \coker(\times (t - 1) \otimes \id_{\ZZ u\Dd^q}) \to H_0(\ZZ, \ZZ u \Dd^q)$ of \cite[Example 6.1.4]{Wei94} carries $d + N_q$ to $t^0 \otimes d + \im(\times(t - 1) \otimes \id_{\ZZ u\Dd^q})$. So composing these maps gives an isomorphism $\psi_{u,0} \colon (\ZZ u \Dd^q)_{\ZZ} \to H_0(\ZZ; \ZZ u \Dd^q)$ such that $\psi_{u,0} (d + N_q) = d + \im(b_0 \otimes \id_{\ZZ u \Dd^q}) = d + \im(d^h_{0,q})$.
	
	   Similarly, $t^n \mapsto [1]\cdot \delta_n$ restricts to an isomorphism $\ker(\times (t - 1) \otimes \id_{\ZZ u\Dd^q}) \to H_1(\ZZ, \ZZ u \Dd^q)$ taking $\sum_d k_d (t^0 \otimes d)$ to $\sum_d k_d [1;d] + \im(d_{1,q}^h)$. The isomorphism $(\ZZ u\Dd^q)^\ZZ \to \ker(\times (t - 1) \otimes \id_{\ZZ u\Dd^q})$ of \cite[Example 6.1.4]{Wei94} carries $\sum k_d d$  to $\sum k_d (t^0 \otimes d)$. So composing these maps yields an isomorphism $\psi_{u,1}  \colon (\ZZ u \Dd^q)^{\ZZ} \to H_1(\ZZ; \ZZ u \Dd^q)$ given by $\psi_{u,1} (\sum_d k_d d) = \sum_d k_d [1;d] + \im(d_{1,q}^h)$.

	   Lemma~\ref{lem:Z_action_invariants} gives an isomorphism $\phi_{u,0}\colon \ZZ(\Cc_u \bs u\Dd^q) \to (\ZZ u\Dd^q)_\ZZ$ such that $\phi_{u,0}(\lequiv d\requiv) = d + N_q$ and an isomorphism $\phi_{u,1} \colon \ZZ (\Cc_u \bs \Per(u \Dd^q)) \to  (\ZZ u \Dd^q)^{\ZZ}$ such that $\phi_{u,1} (\lequiv d \requiv) = \sum_{d' \in \lequiv d \requiv} d'$. The maps $\alpha_i \coloneqq \bigoplus_{u \in \Dd^0} \psi_{u,i}  \circ \phi_{u,i}$ satisfy~\eqref{eq:explicit_isomorphisms}.
\end{proof}

The next lemma helps to compute the terms $E^{vh}_{p,q}$ in Corollary~\ref{cor:spectral sequence}.

\begin{lem}\label{lem:rho map}
Let $(\Cc, \Dd)$ be a matched pair with $\Cc = \Dd^{0} \times \ZZ$. For
$d = (d_0, \dots, d_q) \in \Per(\Dd^{q+1})$, let $O(d) \coloneqq \min \{ n \ge 1 \mid n \la d = d \}$. For each
$0 \le i \le q$, let $\partial_i\colon \Dd^{q+1} \to \Dd^q$ be the face map
of~\eqref{eq:face_map}. Regarding $O(d) \ra d \in \{s(d)\} \times \ZZ$ as
an integer, for $0 \le i \le q$ define
\[
\rho_i(d) \coloneqq
    \begin{cases}
    (O(d)\ra d_0)/O(\partial_0(d)) &\text{ if $i = 0$}\\
    O(d)/O(\partial_i(d)) &\text{ if $i \ge 1$.}
    \end{cases}
\]
Then $\rho_i(d)$ is a nonnegative integer for each $i$.
\end{lem}
\begin{proof}
First suppose that $i \ge 1$. It suffices to show that $O(d) \la
\partial_i(d) = \partial_i(d)$. By Lemma~\ref{lem:actions faces commute},
\[
d = O(d) \la d
= \big(O(d)\la d_0, (O(d)\ra d_0)\la d_1, \dots, (O(d)\ra( d_0\dots d_{q-1}))\la d_{q}\big).
\]
Hence, $\big(O(d)\ra(d_0\dots d_{j-1})\big) \la d_j = d_j$ for all $1 \le j \le q$. In particular, for $i \le q$, we have
\[
(O(d)\ra d_0\dots d_{i-2})\la (d_{i-1}d_i) = \big(O(d)\ra d_0\dots d_{i-2})\la d_{i-1} \big)\big((O(d)\ra d_0\dots d_{i-1}) \la d_i\big) = d_{i-1}d_i.
\]
Thus for $i \le q$, we have
\begin{align*}
O(d) \la \partial_i(d)
	&= \big(O(d)\la d_0, \dots, (O(d)\ra d_0\dots d_{i-3})\la d_{i-2}, (O(d)\ra d_0\dots d_{i-2})\la (d_{i-1}d_i), \\
		&\qquad\qquad (O(d)\ra d_0\dots d_i)\la d_{i+1}, \dots, (O(d)\ra d_0\dots d_{q-1})\la d_{q}\big)\\
	&= (d_0,\ldots, d_{i-2},d_{i-1} d_i, d_{i+1}, \ldots, d_q )
	= \partial_i(d).
\end{align*}
When $i = q+1$, $O(d) \la \partial_{q+1}(d) = \partial_{q+1}(O(d) \la d) = \partial_{q+1}(d)$; and when $i = 0$,
\begin{align*}
(O(d)\ra d_0) \la \partial_0(d)
    &= (O(d)\ra d_0) \la (d_1, d_2, \dots, d_q) \\
    &= \big((O(d) \la d)_1, (O(d) \la d)_2, \dots, (O(d) \la d)_q\big)
    = (d_1, \dots, d_q) = \partial_0(d).\qedhere
\end{align*}
\end{proof}

\begin{prop}\label{prp:homology in specseq}
Let $(\Cc, \Dd)$ be a matched pair with $\Cc = \Dd^{0} \times \ZZ$. For
$p \in \{0,1\}$, let $d^1_{p,q}\colon E^{vh, 1}_{p, q+1} \to E^{vh, 1}_{p, q}$, $q \ge 0$, be
the differentials in the first sheet of the spectral sequence $E^{vh}_{p,q}$
of Corollary~\ref{cor:spectral sequence}. Let $\alpha_0, \alpha_1$ be as in
Proposition~\ref{prop:integer_bundle_h_homology}. For $0 \le i \le q$,
let $\partial_i\colon \Dd^{q+1} \to \Dd^q$ be the face map
of~\eqref{eq:face_map}, and let $\rho_i\colon \Per(\Dd^q) \to \ZZ$ be as defined in
Lemma~\ref{lem:rho map}. Define
\begin{align*}
 \Delta_{0,q} {}&:  \textstyle\bigoplus_{u \in \Dd^0} \ZZ(\Cc_u \bs u\Dd^{q+1}) \to \bigoplus_{u \in \Dd^0} \ZZ(\Cc_u \bs u\Dd^q)\quad\text{ and}\\
\Delta_{1,q} {}&: \textstyle\bigoplus_{u \in \Dd^0} \ZZ(\Cc_u \bs \Per(u\Dd^{q+1})) \to \bigoplus_{u \in \Dd^0} \ZZ(\Cc_u \bs \Per(u\Dd^q))
\end{align*}
by
\[
\Delta_{0,q}(\lequiv d\requiv) = \sum^q_{i=0} (-1)^i \lequiv \partial_i(d)\requiv,\qquad\text{ and }\qquad
\Delta_{1,q}(\lequiv d\requiv) = \sum^q_{i=0} (-1)^i \rho_i(d)\lequiv \partial_i(d)\requiv.
\]
Then $\alpha_p \circ \Delta_{p,q} = d^1_{p,q} \circ \alpha_p$ for $p = 0,1$
and all $q \in \NN$. In particular, $E^{vh,2}_{0,\bullet}$ is isomorphic to
the homology of the chain complex $(\bigoplus_{u \in \Dd^0} \ZZ(\Cc_u\bs
u\Dd^{\bullet}), \Delta_{0,\bullet})$, and $E^{vh,2}_{1,\bullet}$ is
isomorphic to the homology of the chain complex $(\bigoplus_{u \in \Dd^0}
\ZZ(\Cc_u\bs \Per(u\Dd^{\bullet})), \Delta_{1,\bullet})$.
\end{prop}
\begin{proof}
To see that $\alpha_0 \circ \Delta_{0,q} = d^1_{0,q} \circ \alpha_0$, we use that $d_{0,q}^v = d_{0,q}^1$ to compute:
\begin{align*}
	d^1_{0,q}(\alpha_0(\lequiv d\requiv))
		= d^v_{0,q}([r(d);d]) + \im(d^h_{0, q+1})
		= \alpha_0(\Delta_{0,q}(\lequiv d\requiv)).
\end{align*}
To see that $\alpha_1 \circ \Delta_{1,q} = d^1_{1,q} \circ \alpha_1$, we first claim that for $d \in \Per(\Dd^q)$ and $n \ge 0$,
\begin{equation}\label{eq:(n;d)relation}
[n; d] + \im(d^h_{1,q+1})
	= \sum^{n-1}_{j=0}[1; j \la d] +  \im(d^h_{1,q+1}) \in H^h_{1, q+1}(\Cc, \Dd).
\end{equation}
We argue by induction. The case $n = 0$ is trivial: $[0; \lequiv d\requiv]$ is a sum of degenerate chains. Suppose inductively that~\eqref{eq:(n;d)relation} holds for $n$. We calculate, using~\eqref{eq:(n;d)relation} at the third equality:
\begin{align*}
0 + \im(d^h_{1, q+1})
    &= d^h_{1,q+1}[1,n; d]  + \im(d^h_{1,q+1})\\
	&= [n; d] - [n+1; d] + [1; n \la d]  + \im(d^h_{1,q+1})\\
    &= \sum^{n-1}_{j=0} [1; j \la d] - [n+1; d] + [1; n \la d]  + \im(d^h_{1,q+1})\\
    &= \sum^{n}_{j=0} [1; j \la d] - [n+1; d] + \im(d^h_{1,q+1}),
\end{align*}
and rearranging gives~\eqref{eq:(n;d)relation} for $n+1$.

Since $\lequiv d \requiv = \{j \la d\colon 0 \le j \le O(d)-1\}$,
Equation~\eqref{eq:(n;d)relation} gives
\[
\alpha_1(\lequiv d \requiv)
    = \sum_{d' \in \lequiv d\requiv} [1;d'] + \im(d^h_{1,q+1})
    = \sum^{O(d)-1}_{j=0} [1; j\la d] + \im(d^h_{1,q+1})
    = [O(d); d].
\]

Using this at the first line, we calculate:
\begin{align}
	d^1_{1,q}(\alpha_1(\lequiv d\requiv))
		&= d^1_{1,q}([O(d); d]) + \im(d^h_{1,q})\nonumber\\
		&= [O(d)\ra d_0; \partial_0(d)]
			 + \sum^{q+1}_{i=1} (-1)^i [O(d); \partial_i(d)]
			 + \im(d^h_{1,q})\nonumber\\
		&= \sum^{q+1}_{i=0} (-1)^i [\rho_i(d)O(\partial_i(d)); \partial_i(d)]
			+ \im(d^h_{1,q}).\label{eq:d1alpha1}
\end{align}

For any $d' \in \Dd^q$ and any $n \ge 0$ we have
\begin{align*}
[nO(d'); d'] + \im(d^h_{1,q})
    &= \sum^{nO(d')-1}_{j=0}[1; j \la d'] + \im(d^h_{1,q})\\
    &= \sum^{n-1}_{k=0} \sum^{O(d')-1}_{j=0} [1; j \la (kO(d')\la d')] + \im(d^h_{1,q})\\
    &= n\sum^{O(d')-1}_{j=0} [1; j \la d'] + \im(d^h_{1,q})
    = n[O(d'); d'] + \im(d^h_{1,q}).
\end{align*}
Applying this to each term of~\eqref{eq:d1alpha1}, we obtain
\begin{align*}
d^1_{1,q}(\alpha_1(\lequiv d\requiv))
    &= \sum^{q+1}_{i=0} (-1)^i \rho_i(d)[O(\partial_i(d)); \partial_i(d)]
			+ \im(d^h_{1,q})\\
    &= \sum^{q+1}_{i=0} (-1)^i \rho_i(d)\alpha_1(\lequiv \partial_i(d)\requiv)
    = \alpha_1(\Delta_{1,q}(\lequiv d\requiv)).
\end{align*}

The remaining statements follow.
\end{proof}

\subsection{Graphs of odometers}
\label{subsec:odometer_graphs}

Here, we apply Proposition~\ref{prp:homology in specseq} and Theorem~\ref{thm:cohomologies_are_the_same} to the following class of examples
generalising the odometer action.

\begin{setup}\label{setup:generalised odometer}
Let $E$ be a finite directed graph, and let $p \colon E^1 \to \NN \setminus\{0\}$
be a function. Define $F = (F^0,F^1,r,s)$ by $F^0 =
E^0$, $F^1 = \{(e,i)\colon  e \in E^1, i \in \ZZ/p(e)\ZZ\}$, $r(e,i) = r(e)$, and $s(e,i) = s(e)$. We write $+_{p}$ for
the group operation on $\ZZ/p\ZZ$. Let $\Gg \coloneqq E^0 \times \ZZ$. We obtain a self-similar action of $\Gg$ on the $1$-graph $F^*$ (in the sense of Definition~\ref{dfn:ssa_not_faithful}) by
the unique possible extension of the formulae
\[
(r(e), 1) \la (e,i)
= (e,i+_{p(e)} 1)
\quad\text{ and }\quad
(r(e), 1) \ra (e,i)
	= \begin{cases}
		(s(e), 1)&\text{ if $i = p(e)-1$}\\
		(s(e), 0)&\text{ otherwise.}
	\end{cases}
\]
If $E^0 = \{v\}$, $E^1 = \{e\}$, and $p(e) = 2$, then $(\Gg, F^*)$ is the binary odometer.

Extend $p$ to a functor $p \colon E^* \to \NN^\times$. Then
$\Theta \colon F^* \to \{(\mu,i) \mid \mu \in E^*,\, i \in \ZZ/p(\mu)\ZZ\}$ given by
\[
\Theta((e_1,m_1) (e_2,m_2) \cdots (e_k, m_k)) =
\big(e_1e_2\cdots e_k, \sum^k_{j=1} m_j p(e_1\cdots e_{j-1})\big).
\]
is a bijection. Identifying $F^*$ with $\{(\mu,i) \mid \mu \in E^*,\, i \in \ZZ/p(\mu)\ZZ\}$ via $\Theta$, and writing $\lfloor \cdot \rfloor : \RR \to \ZZ$ for the floor function $\lfloor x\rfloor = \max\{n \in \ZZ : n \le x\}$, we have
\[
(r(\mu), a) \la (\mu,m) = (\mu, a +_{p(\mu)} m)
\quad\text{ and }\quad
(r(\mu), a) \ra (\mu, m) = \big(s(\mu), \lfloor (a+m) / p(\mu) \rfloor\big)
\]
(in the second formula, $m$ is regarded as an element of $\{0, \dots
p(\mu)-1\}$ and the addition $a + m$ is computed in $\ZZ$).
It is helpful to
keep in mind the special case that
\begin{equation}\label{eq:special case odom action}
	a \la (\mu,0) = (\mu, a \bmod p(\mu)),
	\qquad\text{ and }\qquad
	a \ra (\mu,0) = \lfloor a/p(\mu)\rfloor.
\end{equation}
\end{setup}

\begin{rmk}
The self-similar actions of Set-Up~\ref{setup:generalised odometer} are faithful self-similar actions as in Definition~\ref{dfn:selfsimilaraction}. They are also length-preserving matched pairs as in Subsection~\ref{subsec:mp_path_categories}.
\end{rmk}

We use the symbols $\lambda,\mu,\nu$ for paths in $E$ and $\xi,\eta,\zeta$ for paths in $F$. So an element of $F^*$ might
be written as $\xi = (\mu, m)$. We write $\ol{p}\colon F^* \to \NN\setminus\{0\}$ for the map $\ol{p}(\mu, m) = p(\mu)$.

\begin{lem}\label{lem:odometer vs our action}
In the situation of Set-up~\ref{setup:generalised odometer}, we have
$\Per(uF^{*q}) = uF^{*q}$ for each $u \in F^0$ and $q \in E^0$. For each $u \in E^0$ the map $(\mu_0, \dots, \mu_{q-1})
\mapsto \Gg_u \la ((\mu_0, 0), \dots, (\mu_{q-1}, 0))$ induces an
isomorphism $\kappa_q \colon \ZZ uE^{*q} \to \ZZ(\Gg_u \bs uF^{*q})$. The functions $O, \rho_i \colon (F^*)^{q+1} \to \ZZ$ of Lemma~\ref{lem:rho map} satisfy
\begin{equation}\label{eq:odometer order}
O(\xi_0,\dots,\xi_q) = p(\xi_0\xi_1\cdots\xi_q)
\quad \text{and} \quad
\rho_i(\xi_0, \dots, \xi_q)
= \begin{cases}
	1 &\text{ if $0 \le i < q$}\\
	\ol{p}(\xi_q)&\text{ if $i = q$}.
\end{cases}
\end{equation}
\end{lem}
\begin{proof}
For $p_0, \ldots, p_q > 0$, the odometer
action $\operatorname{Od}$ of $\ZZ$ on $\prod^q_{i=0}(\ZZ/p_i\ZZ)$ is transitive, so the order of any point under $\operatorname{Od}$ is $\prod_i p_i$.	
For $u \in E^0$ and $\mu = (\mu_0, \dots, \mu_{q-1}) \in E^*u$, the action of $\Gg_ u$ on $\{((\mu_0,
m_0), \dots, (\mu_q, m_q)) \mid m_i \in \ZZ/p(\mu_i)\ZZ\}$ is conjugate to
this odometer with $p_i = p(\mu_i)$. So each $O(\xi_0, \dots, \xi_q) = \prod^q_{i=0} \ol{p}(\xi_i) = \ol{p}(\xi_0\cdots\xi_q)$, we have $\Per(uF^{*q}) =
uF^{*q}$, and $(\mu_0, \dots, \mu_{q-1}) \mapsto \Gg_u \la
((\mu_0, 0), \dots, (\mu_{q-1}, 0))$ is a bijection $uE^{*q} \to \Gg_u \bs
uF^{*q}$.	

For the $\rho_i$, observe that if $i \ge 1$, then writing $\partial^i(\xi)
= (\eta_0, \dots, \eta_{q-1})$, we have $\eta_0 \cdots \eta_{q-1} =
\xi_0\cdots\xi_q$ if $i \not= q$ and $\eta_0 \cdots \eta_{q-1} = \xi_0\cdots
\xi_{q-1}$ if $i = q$. Hence~\eqref{eq:odometer order} implies that
$\rho_i(\xi) = O(\xi)/O(\partial_i(\xi)) = 1$ if $i < q$, and $\rho_q(\xi) =
O(\xi)/O(\partial_q(\xi)) = \ol{p}(\xi_q)$.

It remains to calculate $\rho_0(\xi)$. Since $\operatorname{Od}$ is transitive,  $\id \times \operatorname{Od}$ is transitive on $\{(\mu_0, \dots, \mu_q)\} \times
\prod^q_{i=0} \ZZ/p(\mu_i)\ZZ$. So it suffices to show that $\xi = ((\mu_0,
0), (\mu_1,0),\dots,(\mu_q,0))$ satisfies $\rho_0(\xi) = 1$. Applying~\eqref{eq:special case odom action}, with $a = O(\xi) = p(\xi_0\cdots\xi_q)$, gives
\[
O(\xi)\ra(\mu_0,0)
     = \lfloor O(\xi)/p(\mu_0)\rfloor
     = \lfloor p(\mu_0\cdots\mu_q)/p(\mu_0)\rfloor
     = p(\mu_1\cdots\mu_q)
     = O(\partial_0(\xi)).\qedhere
\]
\end{proof}

\begin{lem}\label{lem:translate F to E}
With Set-up~\ref{setup:generalised odometer} let $\kappa_q \colon \ZZ uE^{*q} \to \ZZ(\Gg_u \bs uF^{*q})$ be the isomorphism of Lemma~\ref{lem:odometer vs our action}. For each
$0 \le i \le q$, let $\partial_i\colon E^{*(q+1)} \to E^{*q}$ be the face map
of~\eqref{eq:face_map}. Let $\Delta_{1,q}$ be as in
Proposition~\ref{prp:homology in specseq}, and define
$
\widetilde{\Delta}_{1,q} {}\colon \textstyle \bigoplus_{u \in E^0} \ZZ uE^{*(q+1)} \to \bigoplus_{u \in \Dd^0} \ZZ uE^{*q}
$
by
\begin{equation}
\label{eq:delta_twidle}
\widetilde{\Delta}_{1,q}(\mu) = \Big(\sum^{q-1}_{i=0} (-1)^i \partial_i(\mu)\Big) + (-1)^q p(\mu_q)\partial_q(\mu).
\end{equation}
Then $\kappa_q \circ \Delta_{1,q} = \widetilde{\Delta}_{1,q} \circ
\kappa_{q+1}$ for all $q$.
\end{lem}
\begin{proof}
This follows directly from~\eqref{eq:delta_twidle} and Lemma~\ref{lem:odometer vs our action}.
\end{proof}

To compute homology for Set-up~\ref{setup:generalised odometer}, we must compute $\widetilde{\Delta}_{1,1}\big(\bigoplus_{u \in E^0} \ZZ uE^{*2}\big) \subseteq \bigoplus_{u \in \Dd^0} \ZZ uE^{*}$.

\begin{lem}\label{lem:image Deltatilde}
In the situation of Set-up~\ref{setup:generalised odometer}, we have
$\ZZ E^* = \ZZ E^1 + \im(\widetilde{\Delta}_{1,1})$, and
$\im(\widetilde{\Delta}_{1,1}) \cap \ZZ E^1 = \{0\}$. In particular, $\ZZ E^* \cong \ZZ E^1 \oplus \im(\widetilde{\Delta}_{1,1})$.
\end{lem}

\begin{proof}
	Since $\widetilde{\Delta}_{1,0} (\ZZ E^0) = 0$ and $\widetilde{\Delta}_{1,1}(\ZZ (E^0 * E^*)) = \widetilde{\Delta}_{1,1}(\ZZ (E^* * E^0)) = \ZZ E^0$, it suffices to show that $\ZZ E^{\ge 1} = \ZZ E^1 + \im(\widetilde{\Delta}_{1,1}|_{\ZZ E^{\ge 1}}) $ and
	$\im(\widetilde{\Delta}_{1,1}|_{\ZZ E^{\ge 1}}) \cap \ZZ E^1 = \{0\}$.
	
 	Recall that for $\mu \in E^{\ge 1}$ and $1 \le i
	\le |\mu|$, the elements $\mu^{[0,i-1]} \in E^{i-1}$, $\mu^i \in E^1$ and $\mu^{[i,|\mu|]} \in E^{|\mu|-i}$ are defined implicitly by $\mu = \mu^{[0,i-1]} \mu^i \mu^{[i,|\mu|]}$. Let $\widetilde{\Delta} \coloneqq
	\widetilde{\Delta}_{1,1}$. To see that $\ZZ E^{\ge 1} = \ZZ E^1 + \im(\widetilde{\Delta}|_{\ZZ E^{\ge 1}})$, it suffices to show that for $\mu \in E^* \setminus E^0$,
	\begin{equation}\label{eq:delta-mu equiv}
		\mu \in  \Big(\sum^{|\mu|}_{i=1} p(\mu^{[i,|\mu|]}){\mu^i}\Big)
		-  \widetilde{\Delta}({\mu^{[0,|\mu|-1]}, \mu^{|\mu|}})
		+ \lsp_\ZZ\{\widetilde{\Delta}({\alpha,\beta})\colon |\alpha\beta| < |\mu|\}.
	\end{equation}
	in $\ZZ E^*$.
	We induct on $|\mu|$. If $|\mu| = 1$ then $\sum^{|\mu|}_{i=1}
	p(\mu^{[i,|\mu|]}){\mu^i} = p(s(\mu)) \mu = \mu$ and~\eqref{eq:delta-mu equiv} is trivial. Now suppose that~\eqref{eq:delta-mu
		equiv} holds for $|\mu| \le n$ and fix $\mu \in E^{n+1}$. Write $\mu =
	\nu e$ with $e \in E^1$. Then $e = \mu^{[n,n+1]} = \mu^{n+1}$, and $\mu^{[i,n+1]} = \nu^{[i,n]} e$ and $\nu^i =
	\mu^i$ for $i \le n$. We calculate (in $\ZZ E^*$):
	\begin{align}
		\mu
		= -\widetilde{\Delta}({\nu, e}) + e + p(e)\nu
		= -\widetilde{\Delta}({\nu, e}) + p(\mu^{[n,n+1]}){\mu^{n+1}} + p(\mu^{n+1}) \nu.\label{eq:inductive step}
	\end{align}
	By the inductive hypothesis,
	\begin{align*}
		p(\mu^{n+1}) \nu
		&\in  p(\mu^{n+1})\Big( \Big(\sum^{n}_{i=1} p(\nu^{[i,n]}){\nu^i}\Big) - \widetilde{\Delta}({\nu^{[0,n-1]}}, \nu^n) + \lsp_\ZZ\{\widetilde{\Delta}({\alpha,\beta})\colon |\alpha\beta| < n\}\Big)\\
		&\subseteq \sum^{n}_{i=1} p(\mu^{n+1}) p(\nu^{[i,n]}){\nu^i} + \lsp_\ZZ\{\widetilde{\Delta}({\alpha,\beta})\colon |\alpha\beta| < n\}.
	\end{align*}
	Since $p$ is multiplicative and each $\nu^{[i,n]} \mu^{n+1} = \mu^{[i,n+1]}$, we obtain
	\[
	p(\mu^{n+1}) \nu \in \sum^{n}_{i=1} p(\mu^{[i,n+1]}){\mu^i} + \lsp_\ZZ\{\widetilde{\Delta}({\alpha,\beta})\colon |\alpha\beta| < n+1\}.
	\]
	Substituting this into~\eqref{eq:inductive step}
	completes the induction, proving the first statement.
	
	For $\im(\widetilde{\Delta}|_{\ZZ E^{\ge 1}}) \cap \ZZ E^1 = \{0\}$,
	fix $(\mu,\nu), (\eta,\zeta) \in E^{\ge 1} * E^{\ge1}$ with $\mu\nu = \eta\zeta$. We claim that
	\begin{equation}\label{eq:same product Delta cancellation}
		\widetilde{\Delta}({\mu,\nu}) - \widetilde{\Delta}({\eta,\zeta})
		\in \lsp_\ZZ\{\widetilde{\Delta}({\alpha,\beta})\colon |\alpha\beta| < |\mu\nu|\}.
	\end{equation}
	We first show that if $\mu,\nu \in E^*\setminus E^0$ and $\mu\nu
	= \lambda$, then
	\begin{equation}\label{eq:coset of delta mu nu}
		\widetilde{\Delta}({\mu,\nu}) \in -{\lambda} + \sum^{|\lambda|}_{i=1} p(\lambda^{[i,|\lambda|]}){\lambda^i} + \lsp_\ZZ\{\widetilde{\Delta}({\alpha,\beta})\colon |\alpha\beta| < |\lambda|\}.
	\end{equation}
	For this, we calculate, applying~\eqref{eq:delta-mu equiv} twice at the
	second step,
	\begin{align*}
		\widetilde{\Delta}({\mu,\nu})
		&= \nu - {\mu\nu} + p(\nu)\mu\\
		&\in -{\mu\nu} - \widetilde{\Delta}(\nu^{[0,|\nu|-1]},\nu^{|\nu|}) + \sum^{|\nu|}_{i=1} p(\nu^{[i,|\nu|]}){\nu^i} -  p(\nu)\Big(\widetilde{\Delta}(\mu^{[0,|\mu|-1]},\mu^{|\mu|}) + \sum^{|\mu|}_{i=1} p(\mu^{[i,|\mu|]}){\mu^i}\Big)\\
		&\qquad\qquad{} + \lsp_\ZZ\big\{\widetilde{\Delta}({\alpha,\beta})\colon |\alpha\beta| < \max\{|\mu|,|\nu|\} \big\}.
	\end{align*}
	Since each $\mu^{[i,|\mu|]}\nu = (\mu\nu)^{[i,|\mu\nu|]} =\lambda^{[i,|\lambda|]}$
	and since $p$ is multiplicative, this gives
	\begin{align*}
		\widetilde{\Delta}({\mu,\nu})
		&\in -{\mu\nu} + \sum^{|\lambda|}_{i=1} p(\lambda^{[i,|\lambda|]}){\lambda^i} +
		-\widetilde{\Delta}(\nu^{[0,|\nu|-1]},\nu^{|\nu|})
		- \widetilde{\Delta}(p(\nu) \mu^{[0,|\mu|-1]},\mu^{|\mu|})\\
		&\qquad\qquad{} + \lsp_\ZZ\big\{\widetilde{\Delta}({\alpha,\beta})\colon |\alpha\beta| \le \max\{|\mu|,|\nu|\}\big\}.
	\end{align*}
	Since $|\mu|,|\nu| <|\lambda|$, we
	obtain~\eqref{eq:coset of delta mu nu}.
	Since the terms $-{\lambda} + \sum^{|\lambda|}_{i=1}
	p(\lambda^{[i,|\lambda|-1]}){\lambda^i}$ in the right-hand side
	of~\eqref{eq:coset of delta mu nu} depend only on the product $\mu\nu$, we obtain~\eqref{eq:same product Delta cancellation}.
	
	Now, we suppose that $\im(\widetilde{\Delta}_{1,1}|_{\ZZ E^{\ge 1}}) \cap \ZZ E^1 \ne \{0\}$ and derive a contradiction. Let $l \in \NN$ be minimal such that there exists $a \in
	\lsp_\ZZ\{{(\mu,\nu)}\colon |\mu\nu| \le l\}$ with
	$\widetilde{\Delta}(a) \in \ZZ E^1 \setminus \{0\}$. Write $a = \sum a_{\mu,\nu} (\mu,\nu)$. For each $(\mu,\nu)$ such
	that $|\mu\nu| = l$ and $0 \ne a_{\mu,\nu} \in \ZZ$, Equation~\eqref{eq:same
		product Delta cancellation} gives
	\[
	\widetilde{\Delta}\Big(a_{\mu,\nu}\big((\mu^1,(\mu\nu)^{[1,l]}) - {(\mu,\nu)}\big)\Big)
	\in \lsp_\ZZ\{\widetilde{\Delta}({\alpha,\beta})\colon |\alpha\beta| < l\}.
	\]
	Hence,
	\[
	a' \coloneqq a + \sum_{|\mu\nu| = l} a_{\mu,\nu}\big((\mu^1,(\mu\nu)^{[1,l]}) - {(\mu,\nu)}\big)
	\]
	satisfies $\widetilde{\Delta}(a) - \widetilde{\Delta}(a') \in
	\lsp_\ZZ\{\widetilde{\Delta}({\alpha,\beta})\colon |\alpha\beta| < l\}$.
	Fix $b \in \lsp_\ZZ\{{(\alpha,\beta)}\colon |\alpha\beta| < l\}$ such that
	$\widetilde{\Delta}(a) = \widetilde{\Delta}(a') + \widetilde{\Delta}(b) =
	\widetilde{\Delta}(a'+b)$. Let $a'' \coloneqq a' + b$; by construction, $a'' \in
	\lsp_\ZZ\{{(\mu,\nu)}\colon |\mu\nu| \le l\}$.
	
	Write $a' = \sum a'_{\mu,\nu} (\mu,\nu)$. Then $a'_{\mu,\nu} = 0$ for all $(\mu,\nu)$ such that
	$|\mu\nu| = l$ and $|\mu| > 1$. Write $b = \sum b_{\mu,\nu} (\mu,\nu)$. Then $b_{\mu,\nu} =
	0$ for all $\mu,\nu$ with $|\mu\nu| = l$. Hence $a''_{\mu,\nu} = 0$ whenever
	$|\mu\nu| = l$ and $|\mu| > 1$. We have $\widetilde{\Delta}(a'') =
	\widetilde{\Delta}(a) \in \ZZ E^1 \setminus \{0\}$, and since $l$ is minimal there
	exist $\nu \in E^{l-1}$ and $e \in E^1 r(\nu)$ such that $a''_{e,\nu} \not=
	0$.
	We have
	\begin{equation}\label{eq:three_sums}
	\widetilde{\Delta}(a'')_{ e\nu } =
	\sum_{\alpha \in E^* r(e)} a''_{\alpha,  e \nu} -\sum_{\eta\zeta = e\nu} a''_{\eta,\zeta}
	+ \sum_{\tau \in s(e)E^*} p(\tau) a''_{e \nu, \tau}.
	\end{equation}
	By construction of $a''$, the only nonzero term in~\eqref{eq:three_sums} is $-a''_{e,\nu}$ in the middle sum, so
	\[
	\widetilde{\Delta}(a'')_{e \nu} = - a''_{e,\nu} \not= 0,
	\]
	which contradicts that $\widetilde{\Delta}(a'') \in \ZZ E^1$.
\end{proof}

Define $M_s \in M_{E^0,E^1}(\ZZ)$ by $M_s(v,e) = \delta_{v,s(e)}$, regarded as a group homomorphism from $\ZZ E^1$ to $\ZZ E^0$. Similarly, define $M_r \in M_{E^0,E^1}(\ZZ)$ by $M_r(v,e) = \delta_{v,r(e)}$. Let $P \in M_{E^1} (\ZZ)$ be the diagonal matrix $P(e,f) =  \delta_{e,f}p(e)$. Finally, define $M  \colon \ZZ E^1 \to \ZZ E^0$ by
\begin{equation}
	\label{eq:enter_the_matrix}
	M \coloneqq M_r P - M_s.
\end{equation}
In matrix form, $M \in M_{E^0,E^1}(\ZZ)$ is given by
\[
M(v, e) = p(e) \delta_{v,r(e)} - \delta_{v,s(e)}= \begin{cases}
	p(e) &\text{ if $v = r(e)$ and $s(e)\not= r(e)$}\\
	-1 &\text{ if $v = s(e)$ and $s(e)\not= r(e)$}\\
	p(e) - 1 &\text{ if $v = r(e) = s(e)$}\\
	0 &\text{ if $v \not\in\{r(e), s(e)\}$.}
\end{cases}
\]
If $E$ is a finite directed graph we write $H_{\bullet}(E)$ for the categorical homology of its path category $E^*$. 
\begin{prop}\label{prp:E2 for example}
With Set-up~\ref{setup:generalised odometer}, let $M \colon \ZZ E^1 \to \ZZ E^0$ be the homomorphism~\eqref{eq:enter_the_matrix}. The spectral sequence of Corollary~\ref{cor:spectral
sequence} satisfies $E_{i,j}^{vh,2} = 0$ whenever $\max\{i,j\} \ge 2$,
\begin{align*}
	E_{0,1}^{vh,2} \cong H_1(E), \quad
	E_{1,1}^{vh,2} \cong \ker(M), \quad
	E_{0,0}^{vh,2} \cong H_0(E), \quad \text{ and } \quad
	E_{1,0}^{vh,2} \cong \coker(M).
\end{align*}
\end{prop}
\begin{proof}
Lemma~\ref{lem:odometer vs our action} gives $\Per((F^*)^q) = (F^*)^q$, so
Proposition~\ref{prp:homology in specseq} implies that $E^{vh,2}_{p,\bullet}$
is isomorphic to the homology of the chain complex $\big(\bigoplus_{u \in
E^0} \ZZ(\Gg_u\bs uF^{*\bullet}), \Delta_{p,\bullet}\big)$ for $p=0,1$.
Lemma~\ref{lem:translate F to E} gives isomorphisms $\kappa_q \colon E^{*q} \to
\Gg \bs F^{*q}$ intertwining the $\Delta_{0,\bullet}$ with the
categorical-homology boundary maps for $E^*$, giving the descriptions of
$E_{0,\bullet}^{vh,2}$.

Lemma~\ref{lem:translate F to E} yields an isomorphism $\kappa_\bullet \colon \big(\bigoplus_{u \in E^0} \ZZ uE^{*\bullet},
\widetilde{\Delta}_{1,\bullet}\big) \to \big(\bigoplus_{u \in E^0}
\ZZ(\Gg_u\bs uF^{*\bullet}), \Delta_{1,\bullet}\big)$ of chain
complexes induced by the $\kappa_q$. So
$E_{1,\bullet}^{vh,2} \cong H_{\bullet}\big(\bigoplus_{u
\in E^0} \ZZ uE^{*\bullet}, \widetilde{\Delta}_{1,\bullet}\big)$.

We have $\widetilde{\Delta}_{1,0}(\mu) = p(\mu){r(\mu)} -
{s(\mu)}$, giving
$
E_{1,0}^{vh, 2}
    = \ZZ E^0/\lsp_\ZZ\{p(\mu){r(\mu)} - {s(\mu)}\colon \mu \in E^*\}.
$
Fix $\mu \in E^*$. The telescoping identity
\[
p(\mu){r(\mu)} - {s(\mu)}
    = \sum^{|\mu|}_{i=1} p(\mu^{i+1}\dots \mu^{|\mu|})\big(p(\mu^i){r(\mu^i)} - {s(\mu^i)}\big),
\]
gives $\lsp_\ZZ\{p(\mu){r(\mu)} - {s(\mu)}\colon \mu \in E^*\}
\subseteq \lsp_\ZZ\{p(e){r(e)} - {s(e)}\colon e \in E^1\}$. The
reverse containment is trivial. Since each $p(e)r(e) - {s(e)} = M
e$, we deduce that $E_{1,0}^{vh, 2} = \coker(M)$.

It remains to calculate $E_{1,1}^{vh, 2} \cong
\ker(\widetilde{\Delta}_{1,0})/\im(\widetilde{\Delta}_{1,1})$.
Clearly, $(\ker(\widetilde{\Delta}_{1,0}) \cap \ZZ E^1) + \im (\widetilde{\Delta}_{1,1}) \subseteq \ker(\widetilde{\Delta}_{1,0})$. Conversely, if $a \in \ker(\widetilde{\Delta}_{1,0})$, then Lemma~\ref{lem:image Deltatilde} says that $a = a' + x$ for some $a' \in \ZZ E^1$ and $x \in \im (\widetilde{\Delta}_{1,1})$. Then $\widetilde{\Delta}_{1,0}(a') = \widetilde{\Delta}_{1,0}(a - x) = 0$, so $a' \in \ker(\widetilde{\Delta}_{1,0}) \cap \ZZ E^1$. Hence,
\[
\frac{\ker(\widetilde{\Delta}_{1,0})}{\im(\widetilde{\Delta}_{1,1})}
=
\frac{(\ker(\widetilde{\Delta}_{1,0}) \cap \ZZ E^1) +\im(\widetilde{\Delta}_{1,1}) }{\im(\widetilde{\Delta}_{1,1})}
\cong
\frac{\ker(\widetilde{\Delta}_{1,0}) \cap \ZZ E^1}{\im (\widetilde{\Delta}_{1,1}) \cap \ZZ E^1}
= \ker(\widetilde{\Delta}_{1,0}|_{\ZZ E^1}).
\]
The restriction of
$\widetilde{\Delta}_{1,0}$ to $\ZZ E^1$ is $M$, so
$E_{1,1}^{vh, 2} \cong \ker(M)$.
\end{proof}

We obtain a computation of the homology of matched pairs $(\Gg,F^*)$ as in Set-up~\ref{setup:generalised odometer}.
\begin{thm}\label{thm:example homology}
In the situation of Set-up~\ref{setup:generalised odometer}, with $M \in
	M_{E^0, E^1}(\ZZ)$ as in \eqref{eq:enter_the_matrix},
	\[
		H_0^{\bowtie}(\Gg, F^*) \cong H_0(E), \qquad\text{ and }\qquad
		H_2^{\bowtie}(\Gg,F^*) \cong \ker(M),
	\]	
	and there is a short exact sequence
	\[
	\begin{tikzcd}[ampersand replacement=\&, column sep=15pt]
		0 \& H_1(E) \& {H_1^{\bowtie}(\Gg, F^*)} \& \coker(M) \& 0.
		\arrow[from=1-1, to=1-2]
		\arrow[from=1-2, to=1-3]
		\arrow[from=1-3, to=1-4]
		\arrow[from=1-4, to=1-5]
	\end{tikzcd}
	\]
\end{thm}

\begin{proof}
	This follows immediately from Lemma~\ref{lem:cohomology_graph_rows_vanish} and Proposition~\ref{prp:E2 for example}.
\end{proof}

Given a finite directed graph $E$, we define $\chi(E) \coloneqq |E^0| - |E^1|$, the Euler characteristic of $E$.

\begin{cor}\label{cor:strongly connected}
In the situation of Set-up~\ref{setup:generalised odometer}, suppose that $v E^* w \not= \varnothing$
for all $v,w \in E^0$, and that $E^1 \not= \varnothing$.
\begin{enumerate}
\item\label{itm:str_connect_1} If $p(e) = 1$ for all $e \in E^1$, then
\begin{align*}
	&H_0^{\bowtie}(\Gg, F^*) \cong \ZZ,&
	 &H_1^{\bowtie}(\Gg, F^*) \cong \ZZ^{2 - \chi(E)},\quad\text{ and}&
	 &H_2^{\bowtie}(\Gg,F^*) \cong \ZZ^{1 - \chi(E)}.&
\end{align*}
\item\label{itm:str_connect_2} If $p(e) > 1$ for some $e \in E^1$, then $\coker(M)$ is a finite cyclic
    group,
\[
	H_0^{\bowtie}(\Gg, F^*) \cong \ZZ \qquad\text{ and }\qquad
		H_2^{\bowtie}(\Gg,F^*) \cong \ZZ^{-\chi(E)},
\]
and there is a short exact sequence
\[
\begin{tikzcd}[ampersand replacement=\&, column sep=15pt]
	0 \& \ZZ^{1-\chi(E)} \& {H_1^{\bowtie}(\Gg, F^*)} \& \coker(M) \& 0;
	\arrow[from=1-1, to=1-2]
	\arrow[from=1-2, to=1-3]
	\arrow[from=1-3, to=1-4]
	\arrow[from=1-4, to=1-5]
\end{tikzcd}
\]
if $\gcd\big\{p(\mu) - p(\nu)\colon \mu,\nu \in E^*, s(\mu) = s(\nu)\text{ and
}r(\mu) = r(\nu)\big\} = 1$, then $\coker(M) = 0$ and $H_1^{\bowtie}(\Gg,
F^*) \cong \ZZ^{1-\chi(E)}$.
\end{enumerate}
\end{cor}
\begin{proof}
 By \cite[p.194]{Massey:Homology} (immediately after
Theorem~3.4), $H_0(E)$ is the free abelian group generated by the connected components of $E$. Since $v E^* w \not= \varnothing$
for all $v,w \in E^0$, there is only one connected component of $E$. So $H_0(E) = \ZZ$. This and \cite[Theorem~3.4]{Massey:Homology} give $H_1(E) \cong \ZZ^{1-\chi(E)}$. We must compute $\ker(M)$ and $\coker(M)$.

\ref{itm:str_connect_1} Suppose that $p(e) = 1$
for all $E$. Lemma~\ref{lem:odometer vs our action} gives $\rho_i
\equiv 1$ for all $i$, and $\Delta_{1,q} = \Delta_{0,q}$ for all $q$. So
$\ker(M) \cong E^{vh,2}_{1,1} \cong E^{vh,2}_{0,1} \cong H_1(E) \cong \ZZ^{1 - \chi(E)}$ and $\coker(M) \cong E^{vh,2}_{1,0} \cong E^{vh,2}_{0,0} \cong H_0(E)
\cong \ZZ$. In particular, $\coker(M)$ is free abelian, so the extension
\[
	\begin{tikzcd}[ampersand replacement=\&, column sep=15pt]
		0 \& H_1(E) \& {H_1^{\bowtie}(\Gg, F^*)} \& \coker(M) \& 0
		\arrow[from=1-1, to=1-2]
		\arrow[from=1-2, to=1-3]
		\arrow[from=1-3, to=1-4]
		\arrow[from=1-4, to=1-5]
	\end{tikzcd}
\]
of Theorem~\ref{thm:example homology} splits, giving the desired formulae for $H_{\bullet}^{\bowtie}(\Gg ,F^*)$.

\ref{itm:str_connect_2} Now suppose that $p(e) > 1$ for some $e$. By assumption, there exists $\mu \in s(e) E^*
r(e)$, and $p(e\mu) = p(e)p(\mu) > 1$. For $\nu \in E^*$, we have
$p(\nu){r(\nu)} - {s(\nu)} = \sum^{|\nu|}_{i=1}
p(\nu^{i+1}\cdots\nu^{|\nu|}) (p(\nu^i){r(\nu^i)}
-{s(\nu^i)}) \in \im(M)$.  In particular, $(p(e\mu)-1)
r(e) = p(e\mu){r(e\mu)} - {s(e\mu)} \in \im(M)$, and so
${r(e)} + \im(M)$ has finite order in $\coker(M)$.

Fix $w \in E^0$. By assumption, there exists $\nu \in r(e) E^* w$ and so $w + \im(M) =
w + p(\nu){r(\nu)}  - {s(\nu)} + \im(M) =
p(\mu){r(e)} + \im(M)$. So $\coker(M) = \ZZ {r(e)} + \im(M)$ is
a finite cyclic group. Hence, $\operatorname{rank}(\im(M)) =
\operatorname{rank}(\ZZ E^0) = |E^0|$, and Rank-Nullity
for $\ZZ$-modules gives $\operatorname{rank}(\ker(M)) = \operatorname{rank}(\ZZ E^1) -
\operatorname{rank}(\ZZ E^0)$. Since $\ker(M)$ is a subgroup of a free
abelian group, it is free abelian, so $\ker(M) \cong
\ZZ^{-\chi(E)}$. The formulae for $H_0$ and $H_2$ and the exact sequence
involving $H_1$ now follow from Theorem~\ref{thm:example homology}.

Finally, suppose that $\gcd\big\{p(\mu) - p(\nu)\colon \mu,\nu \in E^*, s(\mu) =
s(\nu)\text{ and }r(\mu) = r(\nu)\big\} = 1$. As above, $a \coloneqq {r(e)}
+ \im(M)$ generates $\coker(M)$. So it suffices
to show that $O(a)$ divides $p(\mu) - p(\nu)$ whenever
$s(\mu) = s(\nu)$ and $r(\mu) = r(\nu)$.
Fix $v,w \in E^0$ and
$\mu,\nu \in vE^*w$. We have $(p(\mu) - p(\nu))v =
(p(\mu){r(\mu)} - {s(\mu)}) - (p(\nu){r(\nu)} -
{s(\nu)}) \in \im(M)$. Fix $\alpha \in v E^*
r(e)$. Then $r(e) + \im(M) = r(e) +
(p(\alpha){r(\alpha)} - {s(\alpha)}) + \im(M) =
p(\alpha)v + \im(M)$. In particular, $(p(\mu) - p(\nu)) r(e)
+ \im(M) = p(\alpha)(p(\mu) - p(\nu))v + \im(M) = 0 +
\im(M)$. So $O(v + \im(M))$ divides $p(\mu) - p(\nu)$.
\end{proof}

\begin{rmk}
The situation when $p(e) = 1$ for all $e$ in Corollary~\ref{cor:strongly connected} boils down to $\Gg \bowtie F^* = \ZZ \times E^*$, so Corollary~\ref{cor:strongly connected}\ref{itm:str_connect_1} is a nice reality check: it says that $H^{\bowtie}_p(\Gg, F^*) = \bigoplus_{i+j = p} H_i(\ZZ) \otimes H_j(E^*)$, in the spirit of the usual K\"unneth formula.
\end{rmk}

\begin{example}
	Suppose that $E$ is the directed graph with a single vertex $v$ and a single edge $e$, so $\chi(E) = 0$. Fix $p(e) \in \NN \setminus\{0\}$ and form the matched pair $(\Gg,F^*)$ of Set-up~\ref{setup:generalised odometer}. Then $H_0^{\bowtie}(\Gg, F^*) \cong H_0(E) \cong \ZZ$. The map $M \colon \ZZ E^1 \to \ZZ E^0$ is  $\times(p(e) - 1)$, so we obtain an exact sequence
	\[\begin{tikzcd}[ampersand replacement=\&, column sep=15pt]
		0 \& \ZZ \& {H_1^{\bowtie}(\Gg,F^*)} \& {\ZZ/(p(e)-1)\ZZ} \& 0.
		\arrow[from=1-1, to=1-2]
		\arrow[from=1-2, to=1-3]
		\arrow[from=1-3, to=1-4]
		\arrow[from=1-4, to=1-5]
	\end{tikzcd}\]
	If $p(e) = 1$, then $H_1^{\bowtie}(\Gg, F^*) \cong \ZZ^2$ and $H_2^{\bowtie}(\Gg, F^*) \cong \ZZ$; if $p(e)= 2$ (the binary odometer), then $H_1^{\bowtie}(\Gg, F^*) \cong \ZZ$ and $H_2^{\bowtie}(\Gg, F^*)=0$.
	
	For $p(e) > 2$ the group cohomology $H^2(\ZZ / (p(e) - 1) \ZZ ; \ZZ) \cong  \ZZ / (p(e) - 1) \ZZ$, so the exact sequence in Corollary~\ref{cor:strongly connected}~\ref{itm:str_connect_2} does not completely determine ${H_1^{\bowtie}(\Gg, F^*)}$.
\end{example}

\section{Twisted \texorpdfstring{$C^*$}{C*}-algebras of self-similar groupoid actions on \texorpdfstring{$k$}{k}-graphs}
\label{sec:CStars}

We give two constructions of a twisted $C^*$-algebra from a self-similar action of a groupoid on a $k$-graph as in Definition~\ref{dfn:ssa_not_faithful}. This is a matched pair consisting of a groupoid and a $k$-graph in which the left action preserves the degree map. By Proposition~\ref{prop:ssa_is_matched_pair}, these generalise the faithful self-similar actions of groupoids on graphs and $k$-graphs of \cite{LRRW18, ABRW}. Our self-similar actions are the examples of \cite{LV22} in which the generalised higher-rank $k$-graphs are $k$-graphs.

Our first construction of such a $C^*$-algebra is twisted by a normalised 2-cocycle in $C_{\Tot}^2(\Gg, \Lambda; \TT)$ and the second is twisted by a normalised 2-cocycle in
$C_{\bowtie}^2(\Gg, \Lambda; \TT)$. We show that cohomologous cocycles yield isomorphic twisted $C^*$-algebras, and that our two constructions are compatible via the isomorphism of cohomology groups of Corollary~\ref{cor:main_theorem_cohomology}. So
all possible twisted $C^*$-algebras arise via total 2-cycles.

We first study $C^*$-algebras twisted by categorical cocycles, and
establish some elementary structure theory, including a gauge-invariant uniqueness theorem.

Recall from \cite{KP00} that a $k$-graph $\Lambda$ is \emph{row-finite} and has \emph{no sources} if $0 < |v\Lambda^n| < \infty$ for all $v \in \Lambda^0$ and $n \in \NN^k$. Following \cite{RS05} (see also \cite[Remark~2.3]{RSY04}), if $\Lambda$ is a $k$-graph and $\mu, \nu \in \Lambda$ we define
\[
\MCE(\mu,\nu) \coloneqq
\mu\Lambda \cap \nu \Lambda \cap \Lambda^{d(\mu) \vee d(\nu)}.
\]
Elements of $\MCE(\mu,\nu)$ are called \emph{minimal common extensions} of $\mu$ and $\nu$.

We adopt the usual conventions from the theory of $C^*$-algebras that homomorphisms are $*$-homomorphisms, and that ideals are closed 2-sided ideals.
\subsection{Twists by categorical \texorpdfstring{$2$}{2}-cocycles}
\label{subsec:categorical twists}

Given a normalised categorical $2$-cocycle $c\colon \Cc^{2} \to \TT$ and
a subcategory $\Cc' \subseteq \Cc$, we continue to
write $c$ for the restriction of $c$ to $(\Cc')^{2} \subseteq \Cc^{2}$.
Given a self similar action of a groupoid $\Gg$ on a $k$-graph $\Lambda$, we regard $\Gg$ and $\Lambda$ as subsets of
$\Gg \bowtie \Lambda$; so  $g \lambda = (g
\la \lambda) (g\ra\lambda)$ for $(g, \lambda) \in \Gg * \Lambda$.

\begin{dfn}[{cf.~\cite{Yusnitha,ABRW}}]\label{dfn:TCKfam}
Let $(\Gg, \Lambda)$ be a self-similar action of a groupoid on a row-finite $k$-graph with no sources.
Let $c\colon (\Gg \bowtie \Lambda)^2 \to \TT$ be a normalised categorical
$2$-cocycle. A \emph{Toeplitz--Cuntz--Krieger $(\Gg, \Lambda; c)$-family} in
a $C^*$-algebra $A$ is a function $t\colon \Gg \bowtie \Lambda \to A$, such that
\begin{enumerate}[labelindent=0pt,labelwidth=\widthof{\ref{itm:TCK1}},label=(TCK\arabic*), ref=(TCK\arabic*),leftmargin=!]
\item \label{itm:TCK1} $t_\zeta t_\eta = \delta_{s(\zeta), r(\eta)} c(\zeta, \eta)
    t_{\zeta\eta}$ for all $(\zeta, \eta) \in (\Gg \bowtie \Lambda)^{2}$,
\item \label{itm:TCK2} $t_{s(\zeta)} = t^*_\zeta t_\zeta$ for all $\zeta \in \Gg \bowtie
    \Lambda$, and
 \item \label{itm:TCK3} for all $\mu,\nu \in \Lambda$ we have $t_\mu t_\mu^* t_{\nu} t_{\nu}^* = \sum_{\lambda \in \MCE(\mu,\nu)} t_{\lambda}t_{\lambda}^*$.
\end{enumerate}

We call $t$ a \emph{Cuntz--Krieger $(\Gg, \Lambda; c)$-family} if, in addition
\begin{enumerate}[labelindent=0pt,labelwidth=\widthof{\ref{itm:CK}},label=(CK), ref=(CK),leftmargin=!]
	\item \label{itm:CK} $t_v = \sum_{\lambda \in v\Lambda^n} t_\lambda t^*_\lambda$ for all
	$v \in \Lambda^0$ and $n \in \NN^k$.
\end{enumerate}

We write $C^*(t) \coloneqq C^*(\{t_\zeta \mid \zeta \in \Gg \bowtie \Lambda\}) \subseteq A$.
\end{dfn}

\begin{rmk}\label{rmk:consequences}
	Relation~\ref{itm:TCK2} for $\zeta = v \in \Lambda^0$ is $t_v^*t_v = t_v$, so $t_v$ is a projection. Now~\ref{itm:TCK2} for any $\zeta$ implies that $t_{\zeta}$ is a partial isometry.
	
	Relation~\ref{itm:TCK1} implies that the $t_v$, for $v \in \Lambda^0$, are mutually orthogonal \cite[Remarks~1.6(vi)]{KP00}. So \ref{itm:TCK1}--\ref{itm:TCK3} say that $\{t_\lambda \colon \lambda \in \Lambda\}$ is a Toeplitz--Cuntz--Krieger $(\Lambda, c)$-family as in \cite{SWW14}, and so induces a homomorphism $\iota^t_\Lambda\colon \Tt C^*(\Lambda, c) \to A$,
	which descends to a homomorphism $\iota^t_\Lambda\colon C^*(\Lambda, c) \to
	A$ if $t$ is a Cuntz--Krieger $(\Gg, \Lambda; c)$-family.
	
	If $\mu,\nu \in \Lambda$ satisfy $d(\mu) = d(\nu)$, then $\MCE(\mu,\nu) = \{\mu\}$ if $\mu = \nu$ and $\emptyset$ otherwise (see \cite[Lemma~3.2]{SWW14}). So \ref{itm:TCK3} gives $t_{\mu} t_\mu^* t_\nu t_{\nu}^* = \delta_{\mu,\nu} t_\nu t_{\nu}^*$. Since $c$ is normalised, \ref{itm:TCK1} implies that $t_{r(\mu)}t_{\mu} = t_{\mu}$, so $t_{\mu}t_{\mu}^* \le t_{r(\mu)}$ for all $\mu \in \Lambda$. Hence, as in \cite[Remark~3.4]{SWW14}, every Toeplitz--Cuntz--Krieger $(\Gg, \Lambda; c)$-family satisfies
	\begin{enumerate}[labelindent=0pt,labelwidth=\widthof{\ref{itm:TCK1}},label=(TCK\arabic*), ref=(TCK\arabic*),leftmargin=!] \setcounter{enumi}{3}
		\item \label{itm:TCK4} $t_v \ge \sum_{\lambda \in v\Lambda^n} t_\lambda t^*_\lambda$ for all
		$v \in \Lambda^0$ and $n \in \NN^k$.
	\end{enumerate}
	
	Relation~\ref{itm:TCK1} for $g \in \Gg$ gives $t_g t_{g^{-1}} = c(g, g^{-1}) t_{r(g)}$, so
	$t_g \overline{c(g, g^{-1})} t_{g^{-1}} = t_{r(g)}$. Uniqueness of
	quasi-inverses in an inverse semigroup then forces $\overline{c(g, g^{-1})}
	t_{g^{-1}} = t_g^*$, so $g \mapsto t_g$ is a twisted unitary representation of $\Gg$ \`a la~\cite{Renault}, and so induces a homomorphism
	$\iota^t_\Gg\colon C^*(\Gg, c) \to A$.
\end{rmk}

The following standard arguments \cite{Spe20, SWW14} show that every $(\Gg,\Lambda;c)$ admits a Toeplitz--Cuntz--Krieger-family of nonzero partial isometries.

\begin{example}\label{eg:left regular}
	By Example~\ref{ex:left_canc} and left cancellativity of $k$-graphs, $\Gg \bowtie \Lambda$ is left cancellative. Hence, for each $ \zeta \in \Gg \bowtie \Lambda$ there is a partial isometry, $L_\zeta \in \Bb(\ell^2(\Gg \bowtie \Lambda))$ such that
	$L_\zeta e_\eta =
	\delta_{s(\zeta), r(\eta)} c(\zeta, \eta) e_{\zeta\eta}$ for all $\eta
	\in \Gg \bowtie \Lambda$.
	Routine calculations show that this determines a
Toeplitz--Cuntz--Krieger $(\Gg, \Lambda; c)$-family $L\colon \Gg \bowtie \Lambda
\to \Bb(\ell^2(\Gg \bowtie \Lambda))$.

We claim that $\{L_\mu L_g L^*_\nu\colon \mu,\nu \in \Lambda, g \in
\Gg^{s(\mu)}_{s(\nu)}\}$ is linearly independent. To see this fix a linear combination $a = \sum_{\mu,g,\nu} a_{\mu,g,\nu} L_\mu L_g
L^*_\nu$ with at least one nonzero coefficient. Fix $(\mu, g, \nu)$ such that $a_{\mu,g,\nu} \not= 0$ and $a_{\mu',
g', \nu'} = 0$ whenever $d(\nu') < d(\nu)$. Then $L_{\nu'}^* e_\nu = 0$
whenever $a_{\mu', g, \nu'} \not= 0$ and $\nu' \not= \nu$. Hence, $\|a\| \ge |(a e_{\nu} \mid e_{\mu g})| = |a_{\mu,g,\nu}| \not= 0$.
\end{example}

\begin{prop}\label{prp:universal exists}
Let $(\Gg,\Lambda)$ be a self-similar action of a groupoid on a row-finite $k$-graph with no sources. Let $c\colon (\Gg \bowtie \Lambda)^2 \to \TT$ be a
normalised categorical $2$-cocycle. There is a $C^*$-algebra $\Tt C^*(\Gg,
\Lambda; c)$ generated by a Toeplitz--Cuntz--Krieger $(\Gg, \Lambda;
c)$-family $t$ that is universal for Toeplitz--Cuntz--Krieger $(\Gg, \Lambda;
c)$-families: if $T$ is a Toeplitz--Cuntz--Krieger $(\Gg, \Lambda;
c)$-family, then there is a unique homomorphism $\pi^T \colon \Tt C^*(\Gg,
\Lambda; c) \to C^*(T)$ such that $T = \pi^T \circ t$.

Consider the ideal $I$ of $\Tt C^*(\Gg, \Lambda; c)$ generated
by $\{t_v -\sum_{\lambda \in v\Lambda^n} t_\lambda t^*_\lambda \mid v \in \Lambda^0\}$. Then $s \colon \zeta \mapsto t_\zeta + I$ is a Cuntz--Krieger $(\Gg, \Lambda;
c)$-family in $C^*(\Gg, \Lambda; c) \coloneqq \Tt C^*(\Gg, \Lambda; c)/I$, and is
universal for Cuntz--Krieger $(\Gg, \Lambda; c)$-families: if $S$ is a Cuntz--Krieger $(\Gg, \Lambda; c)$-family, then there is a unique homomorphism $\pi^S \colon C^*(\Gg, \Lambda; c) \to C^*(S)$ such that $S = \pi^S \circ s$.
\end{prop}

To prove Proposition~\ref{prp:universal exists}, we follow the standard construction of \cite{Blackadar, RaeburnCBMS, Loring}. We first need the following
technical lemma.

\begin{lem}\label{lem:scalars}
Let $(\Gg,\Lambda)$ be a self-similar action of a groupoid on a row-finite $k$-graph with no sources. Let $c\colon (\Gg \bowtie \Lambda)^2 \to \TT$ be a normalised categorical $2$-cocycle. Fix $(\mu, g, \nu)$ and $(\eta, h, \zeta)$ in
$\Lambda * \Gg * \Lambda$, and for each $(\alpha,\beta)$ such that $\nu\alpha
= \eta\beta \in \MCE(\nu, \eta)$, define
\begin{equation}\label{eq:dog turd}
\begin{split}
\omega(\alpha,\beta)
    \coloneqq \overline{c(\nu,\alpha)}&c(\eta,\beta)c(g,\alpha)c(h^{-1}, h)\overline{c(h^{-1},\beta)}\,\overline{c(g\ra\alpha,g\la\alpha)}\\
        & \times c(h^{-1}\la\beta, h^{-1}\ra\beta)c(\mu, g\la\alpha)\overline{c(\zeta, h^{-1}\la\beta)}\\
        & \times\overline{c(h^{-1}\ra\beta, (h^{-1}\ra\beta)^{-1})}c(g\ra\alpha, (h^{-1}\ra\beta)^{-1}).
\end{split}
\end{equation}
Then for any Toeplitz--Cuntz--Krieger $(\Gg, \Lambda; c)$-family $T$, we have
\[
T_\mu T_g T^*_\nu T_\eta T_h T^*_\zeta
    = \sum_{\nu\alpha = \eta\beta \in \MCE(\nu,\eta)} \omega_{\alpha,\beta} T_{\mu(g\la\alpha)} T_{(g\ra\alpha)(h^{-1}\ra\beta)^{-1}} T^*_{\zeta(h^{-1}\la\beta)}.
\]
\end{lem}
\begin{proof}
Relation~\ref{itm:TCK3} implies that each
\[
T^*_\nu T_\eta
= T_\nu^* (T_\nu T_{\nu}^* T_\eta T_{\eta}^*) T_{\eta}
= \sum_{\nu\alpha = \eta \beta \in \MCE(\nu,\eta)}  T_\nu^* T_{\nu \alpha}  T_{\eta \beta}^* T_{\eta}.
\]
For fixed $(\alpha,\beta)$ in the above sum, relations~\ref{itm:TCK1} and then~\ref{itm:TCK2} give
\[
T_\nu^* T_{\nu \alpha}  T_{\eta \beta}^* T_{\eta}
= \ol{c(\nu,\alpha)}c(\eta,\beta)
T_\nu^* T_{\nu} T_{\alpha}  T_{\beta}^* T_{\eta}^* T_{\eta}
= \ol{c(\nu,\alpha)}c(\eta,\beta) T_{\alpha}  T_{\beta}^*.
\]
 Fix $\nu\alpha
= \eta\beta \in \MCE(\nu,\eta)$. Remark~\ref{rmk:consequences} gives
$
T^*_\beta T_h T^*_\zeta = c(h^{-1}, h) T^*_\beta T^*_{h^{-1}} T^*_\zeta.
$
We have
\[
T_g T_\alpha
    = c(g,\alpha) T_{g\alpha}
    = c(g,\alpha) T_{(g\la\alpha)(g\ra\alpha)}
    = c(g,\alpha) \overline{c(g\la\alpha, g\ra\alpha)} T_{g\la\alpha} T_{g\ra\alpha},
\]
and similarly,
\[
T^*_\beta T^*_{h^{-1}}
    = \overline{c(h^{-1}, \beta)} c(h^{-1}\la\beta, h^{-1}\ra\beta) T^*_{h^{-1}\ra\beta} T^*_{h^{-1}\la\beta}.
\]
We have
\[
T_\mu T_{g\la\alpha} = c(\mu, g\la\alpha) T_{\mu(g\la\alpha)}
    \qquad\text{ and }\qquad
T^*_{h^{-1}\la\beta}T^*_\zeta = \overline{c(\zeta,h^{-1}\la\beta)} T^*_{\zeta(h^{-1}\la\beta)}.
\]
Finally,
\begin{align*}
T_{g\la\alpha} T^*_{h^{-1}\ra\beta}
    &= \overline{c(h^{-1}\ra\beta, (h^{-1}\ra\beta)^{-1})} T_{g\la\alpha} T_{(h^{-1}\ra\beta)^{-1}}\\
    &= \overline{c(h^{-1}\ra\beta, (h^{-1}\ra\beta)^{-1})} c(g\la\alpha, (h^{-1}\ra\beta)^{-1}) T_{(g\la\alpha) (h^{-1}\ra\beta)^{-1}}.
\end{align*}
Putting all of these identities together gives
\begin{align*}
T_\mu T_g T^*_\nu T_\eta  T_h T^*_\zeta
&= \sum_{\nu\alpha = \eta \beta \in \MCE(\nu,\eta)}  \ol{c(\nu,\alpha)}c(\eta,\beta)
T_\mu T_g T_{\alpha}  T_{\beta}^*  T_h T^*_\zeta \\
&= \sum_{\nu\alpha = \eta\beta \in \MCE(\nu,\eta)} \omega_{\alpha,\beta} T_{\mu(g\la\alpha)} T_{(g\ra\alpha)(h^{-1}\ra\beta)^{-1}} T^*_{\zeta(h^{-1}\la\beta)}.\qedhere
\end{align*}
\end{proof}

\begin{cor}[cf. {\cite[Proposition~4.5]{LRRW18}}]
Let $(\Gg,\Lambda)$ be a self-similar action of a groupoid on a row-finite $k$-graph with no sources, and let $c\colon (\Gg \bowtie \Lambda)^2 \to \TT$ be a normalised categorical $2$-cocycle. If $T$ is a Toeplitz--Cuntz--Krieger $(\Gg, \Lambda; c)$-family, then $C^*(T) = \clsp\{T_\mu T_g T^*_\nu \colon (\mu,g,\nu) \in \Lambda * \Gg * \Lambda\}$.
\end{cor}
\begin{proof}
The set $X := \clsp\{T_\mu T_g T^*_\nu \colon (\mu,g,\nu) \}$ is a closed subspace of $C^*(T)$. It is closed under adjoints since $T_\mu T_g T^*_\nu =
\overline{c(g, g^{-1})} T^*_\nu T_{g^{-1}} T^*_\mu$, and Lemma~\ref{lem:scalars} shows that it is closed under multiplication, so $X$ is a $C^*$-subalgebra of $C^*(T)$. Fix $\zeta \in \Gg \bowtie \Lambda$. Proposition~\ref{prop:factorisation_system_is_matched_pair} gives a unique factorisation $\zeta = \mu g$ with $\mu \in \Lambda$ and $g \in \Gg$. So $T_\zeta = c(\mu, g) T_\mu T_g T^*_{s(g)} \in X$. So $X$ contains the generators of $C^*(T)$ giving $X = C^*(T)$.
\end{proof}

\begin{proof}[Proof of Proposition~\ref{prp:universal exists}]
Consider the vector
space $V \coloneqq C_c(\Lambda * \Gg * \Lambda)$ of finitely supported
complex-valued functions on $\Lambda * \Gg * \Lambda$, which has basis the
indicator functions $\theta_{\mu,g,\nu}$.

These $\theta_{\mu,g,\nu}$ are linearly independent, so there is a conjugate-linear
map $^*\colon V \to V$ such that
\begin{equation}\label{eq:adjoint}
\theta^*_{\mu,g,\nu} = \overline{c(g, g^{-1})} \theta_{\nu, g, \mu},
\end{equation}
and there is a bilinear map $\cdot\colon V \times V \to V$ such that, for the
scalars $\omega_{\alpha,\beta}$ defined in~\eqref{eq:dog turd},
\begin{equation}\label{eq:abstract mpctn}
\theta_{\mu,g,\nu} \theta_{\eta,h,\zeta}
     = \sum_{\nu\alpha = \eta\beta \in \MCE(\nu,\eta)} \omega_{\alpha,\beta} \theta_{\mu(g\la\alpha), (g\ra\alpha)(h^{-1}\ra\beta)^{-1}, \zeta(h^{-1}\la\beta)}.
\end{equation}

In the Toeplitz--Cuntz--Krieger $(\Gg, \Lambda; c)$-family $L$ of Example~\ref{eg:left regular}, the $L_\mu L_g L^*_\mu$ are linearly independent, so there is a linear injection $\varphi_L\colon V \to \Bb(\ell^2(\Gg \bowtie \Lambda))$ satisfying $\varphi_L(\theta_{\mu,g,\nu}) = L_\mu L_g L^*_\nu$. Lemma~\ref{lem:scalars} and bilinearity of multiplication in $\Bb(\ell^2(\Gg \bowtie \Lambda))$ shows that $\varphi_L$ intertwines~\eqref{eq:abstract mpctn} with multiplication. Remark~\ref{rmk:consequences} shows that it carries~\eqref{eq:adjoint} to the adjoint in $\Bb(\ell^2(\Gg \bowtie\Lambda))$.

Since $\Bb(\ell^2(\Gg \bowtie \Lambda)$ is a $^*$-algebra, we deduce that the operations we have defined on $V$ satisfy the $^*$-algebra axioms, so $V$ is a $^*$-algebra.
The $\theta_{\mu, g, \nu}$ are linearly independent, so for any
Toeplitz--Cuntz--Krieger $(\Gg, \Lambda; c)$-family $T$ there is a linear map
$\varphi_T\colon V \to \clsp\{T_\mu T_g T^*_\nu\colon (\mu,g,\nu) \in \Lambda
* \Gg * \Lambda\}$ such that $\varphi_T(\theta_{\mu, g, \nu}) =  T_\mu T_g T^*_\nu$.
Lemma~\ref{lem:scalars} shows that $\varphi_T$ is a homomorphism. The
$T_\zeta$ are partial isometries, so for $a = \sum_{(\mu,g,\nu)} a_{\mu,g,\nu} \theta_{\mu,g,\nu} \in V$, we have
$\|\varphi_T(a)\| \le \sum_{(\mu,g,\nu)} |a_{\mu,g,\nu}|$. The map $\rho : V \to [0,\infty)$, given by
\[
\rho(a) \coloneqq \sup\{\|\varphi_T(a)\|\colon T\text{ is a Toeplitz--Cuntz--Krieger $(\Gg, \Lambda; c)$-family}\}
\]
is a pre-$C^*$-seminorm. Quotienting by $N \coloneqq \ker\rho$ and
completing gives a $C^*$-algebra $\Tt C^*(\Gg, \Lambda; c)$.

The map $t \colon \mu g \mapsto \overline{c(\mu,g)} \theta_{\mu, g, s(g)} + N$ is a Toeplitz--Cuntz--Krieger $(\Gg,\Lambda;c)$-family
in $\Tt C^*(\Gg, \Lambda; c)$ because $N$ contains the obstructions to relations \ref{itm:TCK1}--\ref{itm:TCK3}. This $t$ is universal:
for any family $T$ and any $a \in V$, we have $\|\varphi_T(a)\| \le \|\varphi_t(a)\|$, so $\varphi_T$ factors through a norm-decreasing
homomorphism from $\varphi_t(V)$ to $\varphi_T(V)$, which extends to a homomorphism $\pi^T \colon \Tt C^*(\Lambda, \Gg; c) \to C^*(T)$ of
$C^*$-algebras by continuity.

By definition of $I$, the map $s$
is a Cuntz--Krieger $(\Gg, \Lambda; c)$-family. Given any Cuntz--Krieger $(\Gg, \Lambda; c)$-family $S$, the kernel of the homomorphism
$\pi^S\colon \Tt C^*(\Lambda, \Gg; c) \to C^*(S)$ contains $I$, so $\pi^S$ descends to a homomorphism $C^*(\Lambda, \Gg; c) \to C^*(S)$.
\end{proof}

Since $\Lambda$ and $\Gg$ are subcategories of $\Gg \bowtie \Lambda$, if $c \colon (\Lambda * \Gg)^2 \to \TT$ is a $2$-cocycle then
$c|_{\Lambda^2}$ and $c|_{\Gg^2}$ are $2$-cocycles on $\Lambda$ and $\Gg$, which we continue to denote by $c$.

Our next steps are to show that if $\Gg$ is amenable, then $\iota^t_\Gg\colon C^*(\Gg, c) \to \Tt
C^*(\Gg,\Lambda; c)$ from Remark~\ref{rmk:consequences} is always injective, and follow Yusnitha's analysis
\cite{Yusnitha} to see that her joint-faithfulness condition implies that $\iota^s_{\Gg}\colon C^*(\Gg,
c) \to  C^*(\Gg,\Lambda; c)$ is faithful. We also show that
$\iota^t_\Lambda\colon \Tt C^*(\Lambda, c) \to \Tt C^*(\Gg, \Lambda; c)$ and
$\iota^s_\Lambda\colon C^*(\Lambda, c) \to C^*(\Gg, \Lambda; c)$ are always
injective.

If $(\Gg,\Lambda)$ is a self-similar action of a groupoid on a row-finite $k$-graph with no sources, then the degree map $d_\Lambda : \Lambda \to \NN^k$ determines a function $d_{\Gg \bowtie \Lambda} : \Gg \bowtie\Lambda \to \NN^k$ by $d_{\Gg \bowtie \Lambda}(\mu, g) = d_\Lambda(\mu)$ for all $(\mu,g) \in \Lambda * \Gg = \Gg \bowtie \Lambda$. We will just write $d$ for both $d_\Lambda$ and $d_{\Gg \bowtie \Lambda}$ unless the subscript is needed for clarity. Since $d(g \la \mu) = d(\mu)$ for all $(g,\mu) \in \Gg * \Lambda$, for each $((\mu,g),(\nu,h)) \in (\Gg \bowtie \Lambda)^2$,
\[
d((\mu,g)(\nu,h)) = d\big(\mu (g\la\nu), (g\ra\nu)h\big) = d(\mu (g\la\nu)) = d(\mu\nu) = d(\mu,g)+d(\nu,h).
\]
So $d : \Gg \bowtie \Lambda \to \NN^k$ is a functor.

For each $z \in \TT^k$, the
function $\gamma_z(t)\colon \Gg \bowtie \Lambda \to \Tt C^*(\Gg, \Lambda; c)$ defined
by $\gamma_z(t)(\zeta) = z^{d(\zeta)} t_\zeta$ is a Toeplitz--Cuntz--Krieger
$(\Gg, \Lambda;c)$-family. By the universal property, $\gamma_z$
extends to an endomorphism of $\Tt C^*(\Gg, \Lambda; c)$. Since $\gamma_z \circ
\gamma_w (s_\zeta) = \gamma_{zw}(s_\zeta)$ for all $\zeta \in \Gg \bowtie \Lambda$, this
$\gamma$ is an action by automorphisms. An $\varepsilon/3$-argument shows that it is
strongly continuous. We call $\gamma \colon \TT^k \to \Aut (\Tt C^*(\Gg, \Lambda; c))$
the \emph{gauge action}, and write $\Tt C^*(\Gg,\Lambda;c)^{\gamma}$ for the fixed-point
algebra $\{a \in \Tt C^*(\Gg,\Lambda;c) \mid \gamma_z(a) = a \text{ for all } z \in \TT\}$.

The same argument yields a strongly continuous action, also denoted $\gamma$ and called the gauge action, of $\TT$ on $C^*(\Gg, \Lambda; c)$ such that $\gamma_z(s_\zeta) = z^{d(\zeta)} s_\zeta$ for all $\zeta$, and we likewise write $C^*(\Gg, \Lambda; c)^\gamma$ for the resulting fixed-point algebra.

\begin{prop}[cf. {\cite[Proposition~3.6]{Yusnitha}}]\label{prop:faithfulness}
Let $(\Gg,\Lambda)$ be a self-similar action of a groupoid on a row-finite $k$-graph with no sources, and let $c\colon (\Gg \bowtie \Lambda)^2 \to \TT$ be a
normalised categorical $2$-cocycle. The generators $t_\zeta$ of $\Tt
C^*(\Gg,\Lambda; c)$ and $s_\zeta$ of $C^*(\Gg,\Lambda; c)$
are all nonzero. The homomorphisms $\iota^t_\Lambda : \Tt C^*(\Lambda, c) \to \Tt C^*(\Gg, \Lambda; c)$ and
$\iota^s_\Lambda : C^*(\Lambda, c) \to C^*(\Gg, \Lambda; c)$ are injective.
 If $\Gg$ is amenable, then $\iota^t_\Gg \colon
C^*(\Gg, c) \to \Tt C^*(\Gg, \Lambda; c)$ is injective. If, in addition, for
every $v \in \Lambda^0$ and every $n \in \NN^k$, there exists $\lambda \in
v\Lambda^n$ such that $g \mapsto (g\ra\lambda, g\la\lambda)$ is
injective on $\Gg^v_v$, then $\iota^s_\Gg\colon C^*(\Gg, c) \to
C^*(\Gg, \Lambda; c)$ is injective.
\end{prop}

\begin{proof}
Let $L\colon \Gg \bowtie\Lambda \to \Bb(\ell^2(\Gg \bowtie \Lambda))$ be the Toeplitz--Cuntz--Krieger $(\Gg,\Lambda;
c)$-family of Example~\ref{eg:left regular}. Since the $L_v$ are nonzero, the universal property of $\Tt C^*(\Gg, \Lambda; c)$ implies that the $t_v$ are nonzero. As each $\|t_\zeta\|^2 = \|t_\zeta^*t_\zeta\| = \|t_{s(\zeta)}\|$, the $t_\zeta$ are all nonzero. For each $v \in \Lambda^0$ and $n \in \NN^k$, we have $(L_v - \sum_{\mu \in v\Lambda^n} L_\mu L^*_\mu)e_v = e_v \not= 0$, so \cite[Theorem~3.15]{SWW14} implies that $\Pi^L \circ \iota^t_\Lambda$ is injective, and hence $\iota^t_\Lambda$ itself is injective.

To see that the $s_\zeta$ are all nonzero, we follow the argument of \cite{Yusnitha}. For each $v \in \Lambda^0$ and $n \in \NN^k$, the projection $\Delta_{n,v} \coloneqq L_v - \sum_{\lambda \in v\Lambda^n}
L_\lambda L^*_\lambda$ vanishes on $\clsp\{e_{\lambda g}\colon d(\lambda) \ge
n\}$. A direct calculation using
that $d(g \la \lambda) = d(\lambda)$ for all $\lambda$, shows that $L_\lambda L_g
L^*_\mu \Delta_{n, v} L_\nu L_h L_\eta^* e_{\zeta g} = 0$ whenever $d(\zeta)
> d(\eta)$ and $d(\zeta) - d(\eta) + d(\nu) \ge n$. In particular, for $a
\in \lsp\{L_\lambda L_g L^*_\mu \Delta_{n, v} L_\nu L_h L_\eta^*\colon \lambda,
\mu,\nu,\eta \in \Lambda, g,h \in \Gg, n \in \NN^k, v \in \Lambda^0\}$,
regarding $\NN^k$ as a directed set, $\lim_{n \in \NN^k} \big\|a|_{\clsp\{e_{\zeta g}\colon \zeta \in
\Lambda^n, g \in \Gg\}}\big\| = 0$. An approximation argument gives
\begin{equation}\label{eq:norm to zero}
\lim_{n \in \NN^k} \big\|\pi^L(a)|_{\clsp\{e_{\zeta g}\colon \zeta \in \Lambda^n, g \in \Gg\}}\big\| = 0
    \quad\text{ for all $a \in I$.}
\end{equation}
Fix $v \in \Lambda^0$, $n \in \NN^k$, and $\zeta \in v\Lambda^n$. Then $\|L_v e_\zeta\| = \|e_\zeta\|
= 1$, and so
\[
\lim_{n \in \NN^k} \big\|\pi^L(t_v)|_{\clsp\{e_{\zeta g}\colon \zeta \in
\Lambda^n, g \in \Gg\}}\big\| = 1.
\]
Hence, $t_v \not\in I$, so $s_v = t_v + I \ne 0$. Now by~\ref{itm:TCK2}, each $\|s_\zeta\|^2 = \|s_\zeta s^*_\zeta\| = \|s_{s(\zeta)}\| > 0$. The homomorphism $\iota^s_\Lambda$ intertwines the gauge actions of $\TT^k$ on $C^*(\Lambda, c)$ and $C^*(\Gg, \Lambda; c)$, so the gauge-invariant uniqueness theorem \cite[Corollary~7.7]{KPSiv} implies that $\iota^s_\Lambda$ is injective.

The subspace $\ell^2(\Gg) \subseteq \ell^2(\Gg \bowtie \Lambda)$ is invariant for $\pi^L_\Gg\colon C^*(\Gg, c) \to \Bb(\ell^2(\Gg \bowtie \Lambda))$, and the reduction of $\pi^L$ to $\ell^2(\Gg)$  is the
left regular representation $\lambda$ of $C^*(\Gg, c)$. Since $\Gg$ is amenable, $\lambda$ is faithful~\cite[Proposition~9.6]{Wil19}, so $\pi^L_\Gg = \pi^L \circ \iota^t_\Gg$ is injective, and hence $\iota^t_\Gg$ is injective.

Finally, fix $v \in \Lambda^0$, $n \in \NN^k$, and
$\lambda \in \Lambda$ such that $g \mapsto (g\la\lambda, g\ra\lambda)$ is
injective on $\Gg^{r(\lambda)}_{r(\lambda)}$. Again following Yusnitha
\cite[Proposition~3.6]{Yusnitha}, the space $\Hh_{n, \lambda} \coloneqq \clsp\{e_{g\lambda\colon g \in \Gg^{v}_{v}}\}$ is invariant for $\{L_g\colon g \in \Gg^v_v\}$, and $e_{g\lambda} \mapsto \overline{c(g, \lambda)} e_g$ induces an isomorphism $U \colon \Hh_{n,\lambda} \to \ell^2(\Gg^v_v)$ that intertwines the
reduction of $\pi^L|_{C^*(\Gg^v_v, c)}$ with the regular representation. Since $\Gg$ is amenable, so is $\Gg^v_v$ and so the reduction of $\pi^L|_{C^*(\Gg^v_v, c)}$ to $\Hh_{n,\lambda}$ is faithful \cite[Proposition~9.6]{Wil19}. So, for $a \in C^*(\Gg^v_v, c)$, we have $\big\|\pi^L(a)|_{\clsp\{e_{\zeta g}\colon \zeta \in \Lambda^n, g \in \Gg\}}\big\| = \|a\|$ for all $n$. So by~\eqref{eq:norm to zero}, $a \not\in I$, so $\iota^s_\Gg$ is injective on each $C^*(\Gg^v_v, c)$.

Fix a subset $V \subseteq \Gg^0$ that intersects each
$\Gg$-orbit exactly once. Then $P_V = \sum_{v \in V}
\delta_v \in \Mm C^*(\Gg, c)$ is a full projection, and $P_V C^*(\Gg, c) P_V \cong \bigoplus_{v \in
V} C^*(\Gg^v_v, c)$. Since $\iota_s^\Gg$ is injective
on this full corner, it is injective on all of $C^*(\Gg, c)$.
\end{proof}

\begin{rmk}
Let $(\Gg, \Lambda)$ be a self-similar action of an amenable groupoid on a row-finite $k$-graph with no sources, and let $c : (\Gg \bowtie \Lambda)^2 \to \TT$ be a normalised categorical $2$-cocycle. Proposition~\ref{prop:faithfulness} shows that $\Tt C^*(\Gg, \Lambda; c)$ is generated by copies of $\Tt C^*(\Lambda; c)$ and $C^*(\Gg, c)$. It would be interesting to determine when $\Tt C^*(\Gg, \Lambda; c)$ is a $C^*$-blend of these two subalgebras in the sense of \cite{Exel}; or when $C^*(\Gg, \Lambda; c)$ is a blend of $C^*(\Lambda, c)$ and $C^*(\Gg, c)$ (the corresponding result for Zappa--Sz\'ep products of Fell bundles over groupoids appears in \cite[Theorem~5.4]{DL23}). For example, it seems likely that a contracting condition like that of \cite[Section~9]{LRRW18} or \cite[Section~2.11]{Nek05} implies that each spanning element $s_\mu s_g s^*_\nu$ of $C^*(\Gg, \Lambda; c)$ belongs to $\clsp\{s_\alpha s^*_\beta s_g \mid \alpha,\beta \in \Lambda, g \in \Gg\}$. But we do not pursue this question here.
\end{rmk}

We now prove a gauge-invariant uniqueness theorem for $C^*(\Gg,
\Lambda; c)$. The an~Huef--Raeburn uniqueness theorem, a by-now ubiquitous tool in the study of $C^*$-algebras of graphs and related objects, goes back to \cite{aHR}. Our argument in the context of twisted $C^*$-algebras of self-similar actions on $k$-graphs generalises those of \cite{LRRW18, ABRW, Yusnitha} for untwisted actions on graphs and $k$-graphs; our analysis of the core is heavily based on Yusnitha's \cite{Yusnitha}.

\begin{prop}[cf. {\cite[Lemma~4.10]{Yusnitha}}]\label{prp:faithful on core}
Let $(\Gg,\Lambda)$ be a self-similar groupoid action on a row-finite $k$-graph with no sources. Let $c\colon (\Gg \bowtie \Lambda)^2 \to \TT$ be a normalised categorical $2$-cocycle.  Then
\[
C^*(\Gg,\Lambda;c)^{\gamma} = \clsp\big\{s_\mu s_g s^*_\nu\colon \mu,\nu \in \Lambda, d(\mu) = d(\nu), g \in \Gg^{s(\mu)}_{s(\nu)}\big\}.
\]
Let $S$ be a Cuntz--Krieger $(\Lambda, \Gg; c)$-family, and suppose that $\pi^S\colon C^*(\Gg, \Lambda; c) \to C^*(S)$ is injective on $\iota^s_\Gg(C^*(\Gg, c))$. Then $\pi^S$ is injective on $C^*(\Gg, \Lambda;c)^\gamma$.
\end{prop}
\begin{proof}
Let $\Phi\colon C^*(\Gg,\Lambda; c) \to C^*(\Gg, \Lambda; c)^\gamma$ be the faithful conditional
expectation satisfying $\Phi(a) = \int_{\TT^k} \gamma_z(a) \, dz$ \cite[Proposition~3.2]{RaeburnCBMS}. Then $\Phi(s_\mu s_g
s^*_\nu) = \delta_{d(\mu), d(\nu)} s_\mu s_g s^*_\nu$, and so
\begin{align*}
C^*(\Gg, \Lambda; c)^\gamma
    &= \Phi(C^*(\Gg, \Lambda; c))
    = \Phi\big(\clsp\big\{s_\mu s_g s^*_\nu\colon \mu,\nu \in \Lambda, g \in \Gg^{s(\mu)}_{s(\nu)}\big\}\big)\\
    &= \clsp\big\{s_\mu s_g s^*_\nu\colon \mu,\nu \in \Lambda, d(\mu) = d(\nu), g \in \Gg^{s(\mu)}_{s(\nu)}\big\}.
\end{align*}

For fixed $n \in \NN^k$,
\[
F_n \coloneqq \clsp\big\{s_\mu s_g s^*_\nu\colon \mu,\nu \in \Lambda^n, g \in \Gg^{s(\mu)}_{s(\nu)}\big\}
\]
is a $C^*$-subalgebra of $C^*(\Gg, \Lambda; c)^\gamma$. If $m, n \in \NN^k$, $\mu,\nu \in \Lambda^m$, and $g \in \Gg_{s(\nu)}^{s(\mu)}$, then
\begin{align*}
 s_\mu s_g s^*_\nu
    &= s_\mu s_g \sum_{\tau \in s(g)\Lambda^n} s_\tau s^*_\tau s^*_\nu
    = \sum_{\tau \in s(g)\Lambda^n} c(g, \tau)\overline{c(g\la\tau, g\ra\tau)} s_\mu s_{g\la\tau} s_{g\ra\tau} s^*_\tau s^*_\nu\\
    &= \sum_{\tau \in s(g)\Lambda^n} c(\mu,g\la\tau) \overline{c(\nu,\tau)} c(g, \tau)\overline{c(g\la\tau, g\ra\tau)} s_{\mu g\la\tau} s_{g\ra\tau} s^*_{\nu\tau}
    \in F_{m+n}.
\end{align*}
So $F_m \subseteq F_{m+n}$ and $C^*(\Gg, \Lambda;
c)^\gamma = \ol{\bigcup_n F_n}$. So to see that $\pi^S$ is
injective on $C^*(\Gg, \Lambda; c)^\gamma$, it suffices to show that it is
injective, and hence isometric, on each $F_n$.

Fix $n \in \NN^k$. Consider the equivalence relation $\sim$ on $\Lambda^n$ such that $\lambda \sim \mu$ if and only if $\Gg^{s(\lambda)}_{s(\mu)} \not=\varnothing$. Let $K \subseteq \Lambda^n$ be a set of representatives of $\Lambda^n/{\sim}$. For $\lambda \in \Lambda^n$, there exists $\mu \in K$ with $\mu \sim \lambda$, say $g \in \Gg^{s(\lambda)}_{s(\mu)}$. So $s_\lambda s^*_\lambda = s_\lambda u_g s^*_\mu (s_\mu s^*_\mu) s_\mu u_g^* s_\lambda$. Since $\sum_{\lambda \in \Lambda^n} s_\lambda s^*_\lambda$ is an approximate identity for $F_n$ it
follows that $P_K \coloneqq \sum_{\lambda \in K} s_\lambda s^*_\lambda$ is a full
projection in $\Mm F_n$. So it suffices to show that $\pi^S$ is injective on
$P_K F_n P_K$.

For distinct $\lambda,\mu \in K$ we have $s_\lambda s^*_\lambda F_k s_\mu
s^*_\mu = \{0\}$ by definition of $\sim$, and so $P_k F_k P_k \cong
\bigoplus_{\lambda \in K} s_\lambda s^*_\lambda F_k s_\lambda s^*_\lambda$.
So it suffices to show that $\pi^S$ is injective on each $s_\lambda
s^*_\lambda F_k s_\lambda s^*_\lambda$.

Fix $\lambda \in K$, and let $v \coloneqq s(\lambda)$. Since the $s_\mu s^*_\mu$, for $\mu \in \Lambda^n$,  are mutually orthogonal, $s_\lambda s^*_\lambda F_k
s_\lambda s^*_\lambda = \clsp\{s_\lambda s_g s^*_\lambda\colon g \in
\Gg^v_v\}$. Conjugation by $s_\lambda$ is an isomorphism of this
subalgebra onto $\clsp\{s_g\colon g\in \Gg^v_v\}$. Similarly, $\clsp\{S_\lambda S_gS^*_\lambda\colon g \in \Gg^v_v\} \cong \clsp\{S_g \colon g\in
\Gg^v_v\}$ via conjugation by $S_\lambda$. Since $\pi^S(s_\lambda) = S_\lambda$, and $\pi^S$
is injective on $\clsp\{s_g\colon g\in \Gg^v_v\}$, the result follows.
\end{proof}

We obtain a version of an Huef and Raeburn's gauge-invariant uniqueness theorem \cite{aHR}.

\begin{cor}[The Gauge-Invariant Uniqueness Theorem]\label{cor:giut}
Let $(\Gg,\Lambda)$ be a self-similar action of a groupoid on a row-finite $k$-graph with no sources, and let $c\colon (\Gg \bowtie \Lambda)^2 \to \TT$ be a normalised categorical $2$-cocycle. Let $S$ be a Cuntz--Krieger $(\Gg, \Lambda; c)$ family in a $C^*$-algebra $A$. If there is a strongly-continuous action $\beta\colon \TT^k \to \Aut(A)$ such that $\beta_z(S_\zeta) =
z^{d(\zeta)} s_\zeta$ for all $\zeta \in \Gg \bowtie \Lambda$, and if
$\pi^S\colon C^*(\Gg, \Lambda; c) \to C^*(S)$ is injective on $\iota^s_\Gg(C^*(\Gg, c))$, then $\pi^S$ is injective.
\end{cor}
\begin{proof}
The assumptions combined with Proposition~\ref{prp:faithful
on core} show that $\pi^S$ is injective on $C^*(\Gg, \Lambda; c)^{\gamma}$. Define $\Gamma \colon A \to A$ by $\Gamma(a) = \int_{\TT^k} \beta_z(a) \, dz$.
Since $\beta_z \circ \pi^S = \pi^S \circ \gamma_z$
for all $z$, we have $\pi^S \circ \Phi = \Gamma \circ \pi^S$, and then \cite[Lemma~3.14]{SWW14} shows that $\pi^S$ is injective.
\end{proof}

We now show that the isomorphism class of the twisted $C^*$-algebra of a self-similar action of a groupoid on a $k$-graph depends only on the cohomology class of the twisting $2$-cocycle. The argument is standard; see, for example, \cite[Proposition~5.6]{KPSiv}.

\begin{prop}
Let $(\Gg,\Lambda)$ be a self-similar action on a row-finite $k$-graph with no sources, and let $c_1, c_2\colon (\Gg \bowtie \Lambda)^2 \to
\TT$ be normalised categorical $2$-cocycles. Suppose that $b\colon \Gg
\bowtie \Lambda \to \TT$ is a categorical $1$-cochain such that $d_{\bowtie}^1(b)
c_1 = c_2$. For $i = 1, 2$ let $t^i$ be the universal
Toeplitz--Cuntz--Krieger family in $\Tt C^*(\Gg, \Lambda; c_i)$. Then there
is an isomorphism $\theta_b\colon \Tt C^*(\Gg, \Lambda; c_2) \to \Tt C^*(\Gg,
\Lambda, c_1)$ such that $\theta_b(t^2_\zeta) = b(\zeta)t^1_\zeta$ for all
$\zeta \in \Gg \bowtie \Lambda$. This
isomorphism descends to an isomorphism
$\widetilde{\theta}_b \colon C^*(\Gg, \Lambda; c_2) \to C^*(\Gg, \Lambda, c_1)$.
\end{prop}
\begin{proof}
Define $bt\colon \Gg \bowtie \Lambda \to \Tt C^*(\Gg, \Lambda; c_1)$ by
$
(bt)_\zeta \coloneqq b(\zeta) t^1_\zeta.
$
For $(\zeta, \eta) \in (\Gg \bowtie \Lambda)^2$,
\[
(bt)_\zeta (bt)_\eta
    = b(\zeta)t^1_\zeta b(\eta)t^1_\eta
    = b(\zeta\eta) d^1(b)(\zeta, \eta) c_1(\zeta,\eta) t^1_{\zeta\eta}
    = c_2(\zeta, \eta) (bt)_{\zeta\eta}.
\]
So $bt$ satisfies~\ref{itm:TCK1}. It also satisfies
\ref{itm:TCK2} and \ref{itm:TCK3}, and satisfies~\ref{itm:CK} if and only if $t$ does,
because the factors of $b(\zeta)$ and $\overline{b(\zeta)}$ in these relations cancel.

The universal property of $\Tt C^*(\Gg, \Lambda; c_2)$ gives a homomorphism
$\theta_b\colon \Tt C^*(\Gg, \Lambda; c_2) \to \Tt C^*(\Gg, \Lambda, c_1)$ such
that $\theta_b(t^2_\zeta) = b(\zeta)t^1_\zeta$ for all $\zeta \in \Gg \bowtie
\Lambda$, which descends to a homomorphism $ \widetilde{\theta}_b \colon C^*(\Gg, \Lambda;
c_2) \to C^*(\Gg, \Lambda, c_1)$. Since $d^1(\overline{b}) c_2 = c_1$, there is a corresponding homomorphism $\theta_{\overline{b}}\colon \Tt
C^*(\Gg, \Lambda; c_1) \to \Tt C^*(\Gg, \Lambda, c_2)$ such that
$\theta_{\overline{b}}(t^2_\zeta) = \overline{b(\zeta)}t^1_\zeta$, which also
descends to Cuntz--Krieger algebras. Since $\theta_{b} \circ
\theta_{\overline{b}}$ and $\theta_{\overline{b}} \circ \theta_b$ fix the
generators $t^i_\zeta$, they are the identity homomorphisms, and this descends to Cuntz--Krieger algebras as well.
\end{proof}

\subsection{Twists by total 2-cocycles}
\label{subsec:total twists}

We describe the twisted $C^*$-algebra of a self-similar action on a $k$-graph with respect to a total $2$-cocycle, and show that we obtain the same class of $C^*$-algebras as for categorical cohomology.

\begin{dfn}
Let $(\Gg,\Lambda)$ be a self-similar action of a groupoid on a row-finite $k$-graph with no sources. A function $\varphi\colon \Gg^{2}
	\sqcup (\Gg * \Lambda) \sqcup \Lambda^{2} \to \TT$ is a \emph{normalised total $\TT$-valued $2$-cocycle} on $(\Gg,
	\Lambda)$, if $\varphi_{2,0} \coloneqq \varphi|_{\Gg^{2}}$,
	$\varphi_{1,1} \coloneqq \varphi|_{\Gg * \Lambda}$ and $\varphi_{0,2} \coloneqq
	\varphi|_{\Lambda^{2}}$ satisfy
	\begin{enumerate}
		\item $\varphi_{2,0}\colon \Gg^{2} \to \TT$ is a normalised $\TT$-valued $2$-cocycle in the
		sense of \cite{Renault};
		\item $\varphi_{0,2}\colon \Lambda^{2} \to \TT$ is a normalised $\TT$-valued categorical
		$2$-cocycle in the sense of \cite{KPSiv}; and
		\item $\varphi_{1,1}(h,\lambda) = 1$ whenever $h \in \Gg^0$, or $\lambda \in \Lambda^0$, and for $(g, h, \lambda, \mu) \in \Gg * \Gg * \Lambda *
		\Lambda$,
		\begin{align*}
			\varphi_{1,1}(h\ra\lambda,\mu)  \ol{\varphi_{1,1}(h, \lambda\mu)}  \varphi_{1,1}(h, \lambda)  \varphi_{0,2}(\lambda,\mu)  \ol{\varphi_{0,2}(h\la(\lambda,\mu))} &= 1\quad\text{and}\\
			\varphi_{2,0}((g,h) \ra\lambda)  \ol{\varphi_{2,0}(g,h)}\,\ol{\varphi_{1,1}(h, \lambda)} \varphi_{1,1}(gh,\lambda)  \ol{\varphi_{1,1}(g, h\la\lambda)} &= 1.
		\end{align*}
	\end{enumerate}
\end{dfn}

\begin{rmk}
	In defining a normalised total $\TT$-valued 2-cocycle we have just written out explicitly what it means for $\varphi$ to be a cocycle in $C^2_{\Tot}(\Gg, \Lambda; \TT)$. This can be verified by computing what it means for a cochain to be in the kernel of $ d^2_{\Tot} \colon C^2_{\Tot}(\Gg, \Lambda; \TT) \to C^3_{\Tot}(\Gg, \Lambda; \TT)$, which satisfies $ d^2_{\Tot} \circ \varphi = \varphi \circ d_2^{\Tot}$, where $d_2^{\Tot}$ is from Section~\ref{subsec:total_complex}.
\end{rmk}

\begin{dfn}
Let $(\Gg,\Lambda)$ be a self-similar action of a groupoid on a row-finite $k$-graph with no sources. Fix a normalised cocycle $\varphi \in C^2_{\Tot}(\Gg, \Lambda;
\TT)$ and let $A$ be a $C^*$-algebra. A pair of functions $\Ttt \colon
\Lambda \to A$ and $\Www \colon \Gg \to A$ is a \emph{Toeplitz
$\varphi$-pair} if:
\begin{enumerate}[labelindent=0pt,labelwidth=\widthof{\ref{itm:T1}},label=(T\arabic*), ref=(T\arabic*),leftmargin=!]
\item \label{itm:T1} $\Ttt$ is a Toeplitz--Cuntz--Krieger $(\Lambda,
    \varphi_{0,2})$-family in the sense of \cite{SWW14},
\item \label{itm:T2} $\Www$ is a unitary representation of $(\Gg, \varphi_{2,0})$,
    and
\item \label{itm:T3} $\Www_g\Ttt_\lambda = \varphi_{1,1}(g,\lambda)
    \Ttt_{g\la\lambda}\Www_{g\ra\lambda}$ for all
    $(g,\lambda) \in \Gg * \Lambda$.
\end{enumerate}
We call $(\Www, \Ttt)$ a \emph{Cuntz--Krieger $\varphi$-pair}
if $\Ttt$ is a Cuntz--Krieger $(\Lambda,
\varphi_{0,2})$-family.
\end{dfn}

With $\Psi_{\bullet} \colon C_{\bullet}^{\bowtie}(\Gg,\Lambda) \to C_{\bullet}^{\Tot}(\Gg, \Lambda)$ as in Subsection~\ref{sec:Psi},
define $\Psi^{\bullet} \colon C^{\bullet}_{\Tot}(\Gg,\Lambda; \TT) \to C^{\bullet}_{\bowtie}(\Gg,\Lambda; \TT)$ by
$\Psi^k = \varphi \circ \Psi_k$. Then $\Psi^{\bullet}$ induces an isomorphism on cohomology. On $2$-cochains,
\[
\Psi^2(\varphi) (\lambda g, \mu h) =  \varphi_{2,0}(g\ra\mu, h)\varphi_{1,1}(g,\mu)\varphi_{0,2}(\lambda, g\la\mu)
\]
for all $(\lambda,g,\mu,h) \in \Lambda * \Gg * \Lambda * \Gg$.

We show that  $\Tt C^*(\Gg, \Lambda; \Psi^2(\varphi))$ is universal for Toeplitz $\varphi$-pairs,
and $C^*(\Gg, \Lambda; \Psi^2(\varphi))$ is universal for Cuntz--Krieger $\varphi$-pairs.

\begin{thm}\label{thm:total twist universal}
Let $(\Gg,\Lambda)$ be a self-similar action of a groupoid on a row-finite $k$-graph with no sources. Fix a normalised cocycle $\varphi \in C^2_{\Tot}(\Gg, \Lambda; \TT)$ and let $c \coloneqq \Psi^2(\varphi) \in C_{\bowtie}^2 (\Gg, \Lambda ; \TT)$.
\begin{enumerate}
\item\label{itm:2twists_1} Let $t\colon \Gg \bowtie \Lambda \to \Tt C^*(\Gg,\Lambda; c)$ be the universal Toeplitz--Cuntz--Krieger $(\Gg, \Lambda; c)$-family. There is a Toeplitz $\varphi$-pair $\Ttt, \Www$ in $\Tt C^*(\Gg,\Lambda; c)$
    given by $\Ttt = t|_\Lambda$ and $\Www = t|_\Gg$. Moreover $\Tt C^*(\Gg,\Lambda; c)$ is generated by the ranges of $\Ttt$ and $\Www$, and is universal in the sense that given any Toeplitz $\varphi$-pair
    $\Ttt', \Www'$ in a $C^*$-algebra $A$, there is a homomorphism $\rho\colon \Tt C^*(\Gg,\Lambda; c) \to A$ such that $\rho \circ \Ttt = \Ttt'$ and $\rho \circ \Www = \Www'$.
\item \label{itm:2twists_2} Let $s\colon \Gg \bowtie \Lambda \to C^*(\Gg, \Lambda; c)$ be the universal Cuntz--Krieger $(\Gg, \Lambda; c)$-family. Then there is a Cuntz--Krieger $\varphi$-pair $\Sss, \Uuu$ in $C^*(\Gg,\Lambda;
    c)$ given by $\Sss = s|_\Lambda$ and $\Uuu = s|_\Gg$. Moreover $C^*(\Gg,\Lambda; c)$ is generated by the ranges of $\Sss$ and $\Uuu$, and is universal in the sense that given any Cuntz--Krieger
    $\varphi$-pair $\Sss', \Uuu'$ in a $C^*$-algebra $A$, there is a homomorphism $\rho\colon C^*(\Gg,\Lambda; c) \to A$ such that $\rho \circ \Sss = \Sss'$ and $\rho \circ \Uuu = \Uuu'$.
\end{enumerate}
\end{thm}

The theorem follows from the following correspondence between $\varphi$ pairs and $(\Gg, \Lambda; c)$-families.

\begin{lem}\label{lem:pairs<->families}
Let $(\Gg,\Lambda)$ be a self-similar action of a groupoid on a row-finite $k$-graph with no sources. Fix a normalised cocycle $\varphi \in C^2_{\Tot}(\Gg, \Lambda;
\TT)$ and let $c \coloneqq \Psi^2(\varphi)  \in C_{\bowtie}^2 (\Gg, \Lambda ; \TT)$. If $\Ttt$, $\Www$ is a Toeplitz $\varphi$-pair in a $C^*$-algebra $A$, then
\[
t_{\lambda g} \coloneqq \Ttt_\lambda \Www_g
\]
defines a Toeplitz--Cuntz--Krieger $(\Gg, \Lambda;c)$-family in $A$. If $t$ is a Toeplitz--Cuntz--Krieger $(\Gg, \Lambda; c)$-family in $A$, then $\Ttt_\lambda \coloneqq t_\lambda$ for $\lambda \in \Lambda$ and
$\Www_g \coloneqq t_g$ for $g \in \Gg$ defines a Toeplitz $\varphi$-pair. Moreover, $\Ttt, \Www$ is a Cuntz--Krieger $\varphi$-pair if and only if $t$ is a Cuntz--Krieger $(\Gg, \Lambda; c)$-family.
\end{lem}
\begin{proof}For $g \in \Gg$ and $\lambda \in s(g)\Lambda$,
\begin{align*}
c(g,\lambda)
	= c(r(g) g, \lambda s(\lambda))
	= \varphi_{2,0}(g\ra\lambda, s(\lambda)) \varphi_{1,1}(g,\lambda)\varphi_{0,2} (r(g), g\la\lambda)
	= \varphi_{1,1}(g,\lambda).
\end{align*}
Also, for $h \in \Gg$ and $\mu \in \Lambda r(h)$, since $v \ra v = v\la v = v$ for $v \in \Lambda^0 = \Gg^0$, we have
\[
c(\mu, h)
	= c(\mu s(\mu), r(h) h)
	= \varphi_{2,0}(r(h), h) \varphi_{1,1}(s(\mu), r(h))\varphi_{0,2} (\mu, s(\mu))
	= 1.
\]
Hence, if $g \in \Gg$ and $\lambda \in s(g)\Lambda$, we have
\begin{equation}\label{eq:psivarphi on g,lambda}
c(g,\lambda) \overline{c(g\la\lambda,g\ra\lambda)}
	= c(g,\lambda) 1
	= \varphi_{1,1}(g,\lambda).
\end{equation}

Suppose that $\Ttt, \Www$ is a Toeplitz $\varphi$-pair. If $\lambda g, \mu h$ are composable in $\Gg \bowtie \Lambda$, then
\begin{align*}
\Ttt_\lambda \Www_g \Ttt_\mu \Www_h
	&= \varphi_{1,1}(g,\mu) \Ttt_\lambda \Ttt_{g\la\mu}\Www_{g\ra\mu} \Www_h\\
	&= \varphi_{2,0}(\lambda,g\la\mu) \varphi_{1,1}(g,\mu) \varphi_{0,2}(g\ra\mu, h)
		\Ttt_{\lambda (g\la\mu)} \Www_{(g\ra\mu) h}\\
	&=c(\lambda g, \mu h) \Ttt_{\lambda (g\la\mu)} \Www_{(g\ra\mu) h},
\end{align*}
so $t\colon \lambda g \mapsto \Ttt_\lambda \Www_g$ satisfies~\ref{itm:TCK1}. For~\ref{itm:TCK2}, we calculate
\[
t_{\lambda g}^* t_{\lambda g}
	= \Www_g^* \Ttt_\lambda^* \Ttt_\lambda \Www_g
	= \Www_g^* \Ttt_{s(\lambda)} \Www_g
	= \Www_g^* \Www_g
	= \Www_{s(g)}
	= s_{s(\lambda g)}.
\]
Relation~\ref{itm:TCK3} follows from \ref{itm:T1} by definition of a Toeplitz--Cuntz--Krieger $(\Lambda,\varphi_{0,2})$-family.

Now suppose that $t$ is a Toeplitz--Cuntz--Krieger $(\Gg, \Lambda; c)$-family, and define $\Ttt = t|_\Gg$ and $\Www = t|_\Lambda$. We have $\Ttt_v = \Www_v$ for $v \in \Lambda^0$ because $\Gg^{0} = \Lambda^0$.

Remark~\ref{rmk:consequences} implies that $\Ttt$ and $\Www$ satisfy \ref{itm:T1}~and~\ref{itm:T2}. For~\ref{itm:T3}, we calculate
\[
\Www_g \Ttt_\lambda
	= t_g t_\lambda
	\overset{\eqref{eq:psivarphi on g,lambda}}{=}c(g,\lambda) \overline{c(g\la\lambda,g\ra\lambda)} t_{g\la\lambda} t_{g\ra\lambda}
	= \varphi_{1,1}(g,\lambda) \Ttt_{g\la\lambda} \Www_{g\ra\lambda}.
\]
So $\Ttt$, $\Www$ is a Toeplitz $\varphi$-pair.
For the final assertion, observe that
\[
\sum_{\lambda \in v\Lambda^n} t_\lambda t^*_\lambda
= \sum_{\lambda \in v\Lambda^n} \Ttt_\lambda \Www_{s(\lambda)}
\Www_{s(\lambda)}^* \Ttt^*_\lambda
= \sum_{\lambda \in v\Lambda^n} \Ttt_\lambda \Ttt^*_\lambda.
 \qedhere
\]
\end{proof}

\begin{proof}[Proof of Theorem~\ref{thm:total twist universal}]
 \ref{itm:2twists_1}. The second statement of Lemma~\ref{lem:pairs<->families} shows that $\Ttt, \Www$ is a Toeplitz $\varphi$-pair. It generates $\Tt C^*(\Gg, \Lambda; c)$ because each $t_{\lambda g} = c(\lambda,
g) t_\lambda t_g = c(\lambda, g) \Ttt_\lambda \Www_g$; and given a Toeplitz $\varphi$-pair $\Ttt', \Www'$, the first statement  of Lemma~\ref{lem:pairs<->families} shows that
$t'_{\lambda g} \coloneqq \Ttt'_\lambda \Www'_g$ defines a Toeplitz--Cuntz--Krieger $(\Gg, \Lambda; c)$-family. So the universal property of $\Tt C^*(\Gg, \Lambda; c)$ gives a homomorphism $\rho$ such that
$\rho(t_{\lambda g}) = t'_{\lambda g}$. In particular, $\rho(\Ttt_\lambda) = \Ttt'_\lambda$, and $\rho(\Www_g) = \Www'_g$.

\ref{itm:2twists_2}. Apply~\ref{itm:2twists_1} together with the final statement of Lemma~\ref{lem:pairs<->families}.
\end{proof}

\end{document}